\documentclass[reqno]{amsart}
\usepackage{amsmath} 
\usepackage{amssymb}  
\usepackage{amsfonts}  
\usepackage{mathtools}  
\usepackage{physics} 
\usepackage{esint} 
\usepackage{braket} 
\usepackage{enumitem}
\usepackage[toc,page]{appendix} 
\usepackage[breakable]{tcolorbox}
\usepackage{xcolor}
\usepackage{geometry}
\usepackage[hidelinks]{hyperref}
\usepackage{adjustbox}
\usepackage{stackrel}
\usepackage{cite}

\usepackage[T1]{fontenc}
\usepackage{mathrsfs}
 \usetikzlibrary{arrows.meta}
\newcommand{\BBB}{\color{black}}
\newcommand{\DDD}{\color{black}}

\newcommand{\EEE}{\color{black}}

\newcommand{\MMM}{\color{black}}

\definecolor{tubeorange}{HTML}{F0BC42}
\definecolor{kred}{HTML}{8E1F2F}
\definecolor{AsRoma}{HTML}{FF7900}
\definecolor{mossgreen}{HTML}{6AA84F}

\numberwithin{equation}{section}

\newtheorem{theorem}{Theorem}[section]
\newtheorem{proposition}[theorem]{Proposition}

\newtheorem{lemma}[theorem]{Lemma}
\newtheorem{corollary}[theorem]{Corollary}
\newtheorem{definition}[theorem]{Definition}
\theoremstyle{remark}
\newtheorem{remark}[theorem]{Remark}

\newcommand{\E}{\mathfrak{E}}
\newcommand{\cR}{\mathcal{Q}}
\newcommand{\R}{\mathbb{R}}
\newcommand{\Rd}{\mathbb{R}^d}
\newcommand{\Rk}{\mathbb{R}^k}
\newcommand{\Rkd}
{\mathbb{R}^{k\times d}}
\newcommand{\Sd}{\mathbb{S}^{d-1}}
\newcommand{\m}{\mathfrak{m}}
\newcommand{\N}{\mathbb{N}}

\newcommand{\B}{\mathcal{B}}
\newcommand{\C}{\mathcal{C}}
\newcommand{\F}{\mathfrak{F}}
\newcommand{\cF}{\mathcal{F}}
\newcommand{\Op}{\mathcal{O}}
\newcommand{\e}{\varepsilon}

\mathchardef\emptyset="001F

\newcommand{\Lb}{\mathcal{L}} 
\newcommand{\Hd}{\mathcal{H}^{d-1}}
\newcommand{\mres}{\mathbin{\vrule height 1.6ex depth 0pt width
0.13ex\vrule height 0.13ex depth 0pt width 1.3ex}}

\newcommand{\GBVs}{{\rm GBV}_\star}
\newcommand{\BV}{{\rm BV}}
\newcommand{\GBV}{{\rm GBV}}
\newcommand{\SD}{{\rm SD}}

\newcommand{\SDtostar}{\stackbin[SD_\star]{}{\to}}

\title[Structured deformations for energies with general surface terms]{Structured deformations for energies \\ with general surface terms}

\author[D. Donati and M. Friedrich]{Davide Donati and Manuel Friedrich}

\address[D. Donati]{Scuola Internazionale Superiore di Studi Avanzati, via Bonomea 265, 34136 Trieste, Italy}
\email{ddonati@sissa.it}
\address[M. Friedrich]{Department of Mathematics, Johannes Kepler Universität Linz. Altenbergerstrasse 69,
4040 Linz, Austria}
\email{manuel.friedrich@jku.at}

\begin{document}

\begin{abstract}
We develop a variational theory of structured deformations for energies whose surface densities satisfy general growth conditions. This requires a formulation in the  generalised space ${\rm GBV}_\star$, introduced by Dal Maso and Toader, which is the natural setting for surface energies that are linear near the origin and bounded at infinity.
In this framework, we prove three main results: an approximation theorem for structured deformations, an integral representation theorem for abstract lower semicontinuous functionals, and an explicit representation formula for relaxed  energies. The proofs rely on new density results for functions of bounded variation and on Poincaré-type inequalities tailored to ${\rm GBV}_\star$. Our results extend the applicability of structured deformations to cohesive models in fracture mechanics.

\vspace{0.5 cm}

\end{abstract}

\maketitle

\section{Introduction}

\BBB 
The purpose of this paper is to establish a theory for the description of first-order structured deformations with general nonlinear surface densities.  \EEE 
Structured deformations were introduced by Del Piero \& Owen \cite{DelPieroOwen} as a mathematical framework to describe multiscale \BBB phenomena in materials.   In classical continuum mechanics, the behaviour of materials is   typically described \EEE by an injective map
$g\colon\Omega\to \R^d$, defined on a reference configuration $\Omega\subset \R^d$, whose gradient $\nabla g$ represents the local \BBB distortions. \EEE In the theory of structured deformations, this description is  enriched by replacing $g$ by a triple $(\kappa,g,G)$: the map $g$ still describes the macroscopic deformation, but \MMM it \EEE is now allowed to have discontinuities, which are concentrated on the closed set  $\kappa$, interpreted as a \emph{crack site}. The matrix field $G$ accounts for the contributions at the macroscopical level of smooth \MMM sub-macroscopic \EEE changes, while the tensor $\nabla g-G$ captures those arising from non-smooth sub-macroscopic changes, such as slips and separations (referred to as \emph{disarrangements} \cite{deseri2002energetics}).
 One of the   main contributions of \cite{DelPieroOwen} consists \BBB in \EEE an \emph{approximation result} that \BBB clarifies this interpretation: \EEE  every structured deformation $(\kappa,g,G)$ can be seen as the limit of  \emph{simple deformations}, i.e., \BBB triples of the form \EEE   $\{(\kappa_n,g_n,\nabla g_n)\}_n$,   such  that  the tensors $G$ and $\nabla g-G$  arise from the approximating sequence\DDD s \EEE $\{g_n\}_n$ \BBB   and $\{\nabla g_n\}_n$  in the following way: $G$ is \MMM  the weak limit \EEE of $\nabla g_n$, while $\nabla g-G$ is given by $\nabla (\lim_n g_n)-\lim_n\nabla g_n$, thus representing a measure \MMM of  noncommutability of \EEE the operations $\nabla$ and $\lim_n$. \MMM The \EEE theory of structured   deformations can  \EEE model a broad class of mechanical phenomena in which the macroscopic response is influenced by microscopic slips, cracks, or other disarrangements, including elasticity, plasticity, fracture, and the mechanics of defective crystalline solids. 

Adopting this point of view, Choksi \& Fonseca \cite{ChoksiFonseca} laid the foundations for a variational theory of structured deformations. In their formulation, the triple $(\kappa,g,G)$ is  replaced by a \MMM pair \EEE $(g,G)$, where the deformation map $g$ is  a function in ${\rm SBV}(\Omega;\Rk)$   and $G\in L^p(\Omega;\Rkd)$ for some $p\geq 1$. The crack site $\kappa$ is no longer prescribed independently,
but is instead represented by the jump set $J_g$ of the deformation map. Relying on  a previous \EEE result of Alberti \cite{Alberti}, they were able to recover an analogue of the approximation theorem of Del Piero \& Owen: every pair $(g,G)$ can be approximated, in a suitable sense, by sequences $\{(u_n,\nabla u_n)\}_n$ with $\{u_n\}_n\subset {\rm SBV}(\Omega;\Rk)$, $\nabla u_n$ being the approximate gradient of $u_n$. 
\BBB Departing from a general class of energy functionals with bulk and surface densities 
\begin{align}\label{Intro: relaxation}
\mathcal{E}(u) = \int_\Omega \Psi(x,\nabla u)\,{\rm d}x + \int_{J_u\cap \Omega} \gamma(x,[u],\nu_{u})\,{\rm d}\Hd,
\end{align}
this result justified  the assignment of a \emph{relaxed formulation} $I^p$ representing the energetically most favourable way to approximate a structured deformation, namely  \EEE  
\begin{align}\notag 
       I^p(g,G)\coloneq \inf\Big\{&\liminf_{n\to+\infty} \mathcal{E}(u_n) \colon \,  \{u_n\}_n\subset {\rm SBV}(\Omega;\Rk) \text{ and } \{(u_n,\nabla u_n)\}_n \text{ converges to }(g,G)\Big\}.\label{Intro: relaxation}
\end{align}
Here, $[u]\coloneq u^+-u^-$ is the difference of the unilateral traces \MMM $u^+,u^-$ \EEE  of $u$ on  $J_u$, whose unit normal is given by $\nu_u$,   $\Psi$ is a non-negative integrand with $p$-growth, for $p\geq1$, and $\gamma$ satisfies
\begin{equation}\label{Intro: growth Choksi Fonseca}
 c|\zeta| \leq  \gamma(x,\zeta,\nu)\leq C|\zeta|
\end{equation}
for suitable positive constants $c$ and $C$.
Under these assumptions, Choksi \& Fonseca \MMM identified \EEE  effective bulk and surface
energy densities \MMM which are \EEE  defined through cell formulas (computed  by looking  at simple deformations $(u,\nabla u)$ only) \MMM and  which \EEE allow to represent  $I^p$ as an integral functional.

This variational approach has been developed \BBB further over the last years,  \EEE  including higher-order and multi-level theories  
\cite{FonsHafParo,J.Elasticity,BarrosoMatiasZappale2,BarrosoMatiasZappale1,BaiaMatiasSantos,BarMatMorOwenARMA}, \DDD  linearised models \cite{FriMatZap}, as well as dimension reduction \cite{DimensionReduction2014, DimensionReduction2018} and homogenisation \cite{AmarMatiasMorandotti}\EEE.  For a comprehensive introduction to the theory and for an exhaustive list of  contributions on the topic,
we refer the interested reader to the recent monograph
\cite{MorandottiBook}.

While being satisfactory for \MMM an \EEE  array of applications, \BBB the \EEE growth condition \eqref{Intro: growth Choksi Fonseca} excludes, however, a large plethora of \BBB relevant settings\DDD, \EEE\EEE such as \BBB models for defects in crystals or \EEE  cohesive zone models in fracture mechanics. \BBB There,  \EEE it is natural to assume that for small jump openings the function $\zeta\mapsto\gamma(x,\zeta,\nu)$
exhibits a linear behaviour, while it saturates for large jump openings, approaching a positive constant as $|\zeta|\to +\infty$. 
To treat \BBB such  applications \EEE in the framework of structured deformations, it is thus necessary to relax the bounds in 
\eqref{Intro: growth Choksi Fonseca} and to  consider surface densities satisfying 
 \begin{equation}\label{Intro: cohesive growth}
     c ( \EEE|\zeta|\land 1  ) \EEE\leq  \gamma(x,\zeta,\nu)\leq C  ( \EEE|\zeta|\land 1  ) \EEE,
\end{equation}
where $a\land b\coloneq \min \{a,b\}$ for $a,b\in\R$. \BBB The main scope of this paper  \EEE is to systematically  study the relaxation of energies of the form \eqref{Intro: relaxation}   under the growth \MMM assumption \EEE  \eqref{Intro: cohesive growth}.

Some \BBB observations \EEE in this direction  have already been \BBB made, \EEE see, e.g., the discussions in \cite[Remark~3.3]{ChoksiFonseca} or \cite[Remark~3.1]{MatiasMorandottiZappale}. However, the suggested approach \BBB in \cite{ChoksiFonseca, MatiasMorandottiZappale} \EEE comes at the expense of altering the definition of the relaxation $I^p$ by considering only approximating sequences with \emph{equibounded ${\rm BV}$-norm}. As the lower bound in \eqref{Intro: cohesive growth} does not guarantee this equiboundedness property, this may lead to a\DDD n  ill\EEE-posed existence theory for minimisation problems associated to the relaxed energy $I^p$.  A natural way to circumvent this limitation is  to generalise the admissible class of deformation fields $g$ 
to the natural \MMM domain \EEE  of energies with surface densities as in \eqref{Intro: cohesive growth}. This corresponds to considering the space of generalised functions of bounded variation
${\rm GBV}_\star(\Omega;\Rk)$, introduced in the scalar case by Dal Maso \& Toader
\cite{DalToa22} and extended to the vector-valued setting in
\cite{DonatiGBV}. \BBB In fact, this space has been tailored   to study problems in  cohesive fracture and fractured elastoplastic materials featuring models with surface densities of the form \eqref{Intro: cohesive growth}, and has already been employed  successfully in dealing with several variational problems, see \cite{DalToa22,DonatiGBV,DalToaConvex,DalToaHomo,dal2025homogenization}. \MMM In particular, in this setting suitable compactness results are available, see Remark~\ref{rem:compactness}.

Our goal is to provide a rather complete picture of the relaxation of $I^p$ \EEE in the class of  admissible  deformation fields $g \in {\rm GBV}_\star(\Omega;\Rk)$, \EEE   under the growth \MMM assumption \EEE \eqref{Intro: cohesive growth}, without requiring any additional constraint on the norms of the approximating sequences. 
\BBB Our results can be summarised as follows: \EEE 
\begin{itemize}
    \item [(a)] an extension of the approximation theorem of Choksi \& Fonseca, \BBB where we prove \EEE that every pair $(g,G)$, with $g\in {\rm GBV}_\star(\Omega;\Rk)$ and  $G\in L^1(\Omega;\Rkd)$, can be
approximated by means   of a sequence $\{u_n\}_n\subset {\rm SBV}(\Omega;\mathbb R^k)$  (see Theorem~\ref{thm:approximation SDstar});
\item [(b)] an integral representation result for abstract functionals by means of the global method for relaxation of Bouchitté, Fonseca, \& Mascarenhas \cite{BFM1998} (see Theorem~\ref{thm: integral star}); 
\item [(c)] an explicit integral representation for the relaxation of \eqref{Intro: relaxation}, under the growth \MMM  condition \EEE \eqref{Intro: cohesive growth} (see Theorems~\ref{thm: relaxation partial}, \ref{thm:cell formulas}, and \ref{thm:relaxation full}).
\end{itemize}

Besides the technical difficulties inherent to the function space, the approach exhibits two major conceptual difficulties that render the analysis substantially different from the existing literature on structured deformations on the one hand, and from models formulated on ${\rm GBV}_\star$ on the other hand: 

\begin{itemize}
    \item[(i)] The first is related to the approximation theorem: in contrast to the result in \cite{ChoksiFonseca},  in our setting  it is not meaningful to ask that the  total variation $|Du_n|(\Omega)$ of the approximating sequence is controlled by $ |Dg|(\Omega) + \Vert G \Vert_{L^1(\Omega)}$, but    the total variation on both sides must be replaced by the nonlinear quantity $\mathcal{V}$ defined by       
\begin{equation*}
\mathcal{V}(g,O)\coloneq \sum_{i=1}^k
\Big( \int_O |\nabla g_i|\,{\rm d}x + |D^c g_i|(O)+ \int_{J_{g_i}\cap O} |[g_i]|\wedge 1 \,{\rm d}\mathcal H^{d-1}
\Big) \BBB \quad \text{for $O \subset \Omega$}. \EEE
\end{equation*}
Here,   \MMM $\nabla g_i$ \EEE corresponds to the approximate gradient, $D^cg_i$ to the Cantor part, and $J_{g_i}$ to the jump set \MMM of the $i$-th component of $g$. \EEE  Although  the space ${\rm GBV}_\star$ is far more general than ${\rm BV}$, it retains the fine structure needed to study free discontinuity problems, at the expense of   replacing the control of the total variation $|Dg|$ by the weaker, energetically natural quantity   $\mathcal{V}$.  

\item[(ii)] In the abstract integral representation result based on the global method, we cannot resort to the by-now classical approach in ${\rm (G)SBV}$ and ${\rm GBV}_\star$, which relies on a \emph{perturbation-truncation} argument, see e.g. \cite{BachBraidesCica,BachBraidesZeppieri,BarchiesFocardi,BarchiesiLazzaroniZeppieri,CagnettiDet,DalToaConvex,dal2025homogenization,DalToaHomo}.   More precisely, the usual approach consists in perturbing  a functional satisfying \eqref{Intro: cohesive growth}   by a surface term, leading to a growth condition of the form \eqref{Intro: relaxation}, which allows \MMM to \EEE employ the integral representation results in  \cite{BFML2002} or in  \cite{BFM1998}. Afterwards, the perturbation parameter is sent to zero by means of a suitable truncation technique. While this strategy has been successfully employed in ${\rm GBV}_\star$ in \cite{DalToaConvex,DalToaHomo,dal2025homogenization}, the presence of the matrix field $G$ prevents us from carrying out the same analysis. This requires careful modifications and local truncations in blow-up constructions around bulk and surface points, thereby avoiding any perturbation argument.
\end{itemize}
 
 \BBB 
We further comment on the three main results (a)--(c) highlighting ingredients in the proof. \EEE

(a)
 The requirement of a uniform control \MMM on \EEE $\mathcal{V}$    along the approximating sequence, which is an
essential feature of the Approximation Theorem~\ref{thm:approximation SDstar}, makes the proof of this result highly non-trivial. In particular, it requires  a novel density result for functions of bounded variation.  More precisely, it relies crucially  on Theorem~\ref{thm: Goffmann-Serrin} (see also \MMM Remark~\ref{remark: vVec}(i),(iii)), \EEE which shows that  every function $u\in {\rm GBV}_\star(\Omega;\Rk)$ can be approximated 
by a sequence $\{u_n\}_n\subset {\rm SBV}(\Omega;\Rk)\cap L^\infty(\Omega;\Rk)$ in such a way that $\mathcal{V}(u_n,\Omega\setminus J_{u_n})$ and $\mathcal{V}(u_n,J_{u_n})$ converge to $\mathcal{V}(u,\Omega\setminus J_u)$ and to $\mathcal{V}(u, J_u)$, respectively.  This is achieved by exploiting a classical result of Goffman  \& Serrin \cite[Theorem~4$^\prime$]{GoffmanSerrin} to approximate the function $u$ far from the jump set $J_u$, while a hands-on construction is performed close to $J_u$.
This result can also be adapted to more general energies with linear growth (see Remark~\ref{remark: vVec}(ii)) and,   combining it  with the recent approximation result \cite[Theorem~1.1]{CFI-JFA} for functions with  possibly infinite jump set, it shows  that ${\rm GBV}_\star$-functions can be approximated in energy by piecewise affine functions  with piecewise affine jump set.  We hope that this result might be of independent interest and that our proposed  strategy   may find applications beyond the setting of the present work.

(b) In the second part of the paper, we consider 
a class $\F^p_\star$   of relevant local functionals satisfying two-sided growth conditions \BBB in terms of $\mathcal{V}(g, \MMM \Omega) \EEE  + \|G\|_{L^p}^p$, see Definition  \EEE \ref{def: Functionals}.  We then employ the global method 
\cite{BFM1998,BFML2002} to characterise its
absolutely continuous and its jump part as integral functionals.  \BBB As customary in this approach, \EEE the
corresponding bulk and surface densities are obtained through asymptotic
cell formulas on small cubes, with affine boundary data for the bulk part
and pure jump boundary data for the surface part (see \eqref{def f} and \eqref{def Psi}).   Both formulas   also \EEE take into account  constraints on the average  of the matrix fields $G$.

We stress that in the setting of structured deformations we have to face some serious obstacles   in implementing   this well-known strategy. In particular, as mentioned already before, we are forced to  adopt an approach deviating from the one which is frequently employed for generalised  functions of bounded variation.
This is mainly due to the fact that the relaxation of \eqref{Intro: relaxation} does not satisfy the  `truncation properties' that would allow us to implement the abovementioned  perturbation-truncation method.  To solve this issue, we follow an approach more similar to the one employed in integral representation results in ${\rm SBV}^p$ \cite{BFML2002} and in   ${\rm (G)SBD}^p$ \cite{ContiFocardiIurlanoSBDp,CriFriSolo}, which is based on a careful analysis of the blow-ups of such functions  (see \cite[Lemmas~2 and 3]{BFML2002}, \cite[Lemma~3.4]{ContiFocardiIurlanoSBDp}, \MMM  and \cite[Lemmas~5.1 and 6.1]{CriFriSolo}). \EEE In these works, such analysis is made possible by the use of Poincaré and Korn-Poincaré inequalities that allow for a better control on the quantities at play. \BBB Following this reasoning, \EEE in Theorem~\ref{Thm: Poincaré inequality} we provide an adaptation to ${\rm GBV}_\star(\Omega;\Rk)$ of the Poincaré inequality of De Giorgi, Carriero, \& Leaci \cite{CarrieroLeaci}, which we  then use in Lemmas  \ref{lemma blow up ac} and \ref{lemma: blow up jump points} to perform the required blow-up analysis. With these tools at our disposal, we are then able to proceed with the global method.

(c) In the final part of the paper, we provide an integral representation for the relaxation of  \eqref{Intro: relaxation} under the growth condition \eqref{Intro: cohesive growth}. As an application of Theorem~\ref{thm: integral star}, we first prove in Theorem~\ref{thm: relaxation partial} that \BBB the \EEE bulk and surface parts can always be represented as integral functionals, whose densities we denote by $f_{\rm bulk}$ and $f_{\rm surf}$. We then specialise to the case $p>1$, \BBB corresponding to   superlinear growth of the bulk density $\Psi$. \EEE Here, some truncation properties are available (see Lemmas \ref{lemma:CagnettiDetFree W} and  \ref{lemma:I truncated}), and we are able to prove (see Theorem~\ref{thm:cell formulas}) that $f_{\rm bulk}$ and $f_{\rm surf}$ are given by cell formulas consistent with the ones identified by Choksi \& Fonseca (see \eqref{bulk dens relax} and \eqref{surface dens p>1} and compare with \cite[Remark~3.3]{ChoksiFonseca}).  Finally, in the case \MMM that \EEE $\Psi$ and $\gamma$ are independent of the spatial variable $x$, we obtain a complete representation of the
relaxed functional, including its Cantor part,  which is expressed in
terms of the recession function of the bulk energy density. This is done relying  on the relaxation result \cite[Theorem~3.12]{BFM1998}, similarly to what was done, for instance, for measure-valued structured deformations \cite[Theorem~2.4]{MeasureSD} and for structured deformations in linearised elasticity \cite[Theorem~2.12]{FriMatZap}.  
\medskip

We briefly describe the organisation of the paper.  In Section \ref{sec: preliminaries} we fix
the notation and recall the definition and the main properties of the space
${\rm GBV}_\star$. In Section~\ref{setting and main results} we set the notation concerning \emph{generalised} structured deformations ${\rm SD}^p_\star$ and present the statements of the main results.  Section~\ref{sec: Approximation Lemma} is devoted to the proof \BBB of the \EEE approximation
theorem and to the auxiliary density results on which it relies.  In
Section~\ref{section: Global Method} we prove the abstract integral representation theorem for
functionals in $\mathfrak{F}^p_\star$.  In Section~\ref{section Relaxation} we apply the results of Section~\ref{section: Global Method} to analyse the properties of the relaxed energy.  The appendices collect several new technical
tools concerning functions in ${\rm GBV}_\star(\Omega;\Rk)$. In Appendix~\ref{sec: GBV} we give a brief introduction to the fine properties of generalised \BBB functions \EEE of bounded variation. This is complemented by the analysis carried out in Appendix \ref{sec:slicing}, where we study their slicing properties.  Appendix \ref{sec: Poincarè} is devoted to the proof of Poincaré-type inequalities, which are then used to obtain the blow-up lemmas of  Appendix \ref{sec:blow up}. Finally, in Appendix \ref{sec:lower} we present a well-known result concerning lower semicontinuous integral functionals with linear growth.

\EEE

\section{Preliminaries on generalised functions of bounded variation}\label{sec: preliminaries}

In this section we fix  notation and present some \BBB important definitions.  
\medskip

\noindent \MMM \textbf{Notation:} \EEE  Throughout the paper, \MMM $d, k \in \N$,   and  \EEE $\Omega\subset \Rd$ is always a bounded open subset of $\Rd$ with Lipschitz boundary. 
For $m\in\N$ the scalar product between two vectors $a,b\in\R^m$   is denoted by $a\cdot b$ , while the Euclidean norm of $\mathbb{R}^m$ is denoted by  $|\, \cdot \,|$. \EEE
For every $\rho>0$ and $x\in\Rd$, the open ball  with \EEE centre $x$ and radius $\rho$ is denoted by $B_\rho(x)$; we omit the indication of the centre $x$ if $x=0.$ The unit sphere of $\Rd$ is denoted by
 $\mathbb{S}^{d-1}\coloneq \{\nu\in\Rd: |\nu|=1\}$. Given $x\in\Rd$, $\nu\in\Sd$, and $\rho>0$, we let $Q_\nu(x,\rho)$ \MMM be \EEE a cube centred at $x$, of side-length $\rho$, and with two faces orthogonal to $\nu$. When $\nu=e_d$, we omit the dependence on $\nu$ and  assume that $Q(x,\rho)\coloneq x+(-\rho/2,\rho/2)^d.$  We also simply write $Q$ if $x = 0$ and $\rho = 1$. \EEE
 
 We identify the vector space $\Rkd$ with the space of  $k\times d$ matrices, \MMM  endowed   with the Frobenius norm. \EEE   Given  $A\in\R^{k\times d}$, its $ij$-th component is denoted by $A_{ij}$.
 The tensor product $a\otimes b\in\Rkd$ between two vectors $a\in\Rk$ and $b\in\Rd$ is the matrix whose $ij$-th component is given by $(a\otimes b)_{ij}\coloneq a_ib_j$.
 
For $m\in\N$,  the symbol  $\mathcal{M}_b(\Omega;\R^m)$ denotes the space of all bounded $\R^m$-valued Radon measures. We omit the indication of the target space \MMM if \EEE $m=1$. The space of all positive Radon measures is denoted by  $\mathcal{M}^+(\Omega)$. The subspace of $\mathcal{M}^+(\Omega)$ of all bounded positive measures is denoted by $\mathcal{M}^+_b(\Omega)$.  Given $\mu \in\mathcal{M}_b(\Omega; \R^m \EEE )$ and $\lambda\in \mathcal{M}^+_b(\Omega)$, ${\rm d}\mu/{\rm d}\lambda$ is the Radon-Nikod\'ym derivative of $\mu$ with respect to $\lambda$, while $|\mu|$ denotes the total variation measure of  $\mu$.

Let  $\Lb^m$  and $\mathcal{H}^m$ denote the $m$-dimensional Lebesgue measure and the  $m$-dimensional Hausdorff measure, respectively.  Given a Borel measure $\mu$ on $\Omega$ and a $\mu$-measurable set $E\subset \Omega$, we let $\mu\mres E$ be the restricted measure defined by $(\mu\mres E) (B)\coloneq \mu(E\cap B)$ for every Borel set  $B\subset \Omega$. Given an $\Lb^d$-measurable subset $ E \subset \EEE \Rd$, we denote by $L^0(E;\R^m)$ the collection of all $\Lb^d$-measurable functions $u\colon E\to \R^m$. We endow the space $L^0(E;\R^m)$ with the topology associated with the convergence in measure, which we recall to be separable and metrisable. In particular, a distance inducing this topology is given by 
  \begin{equation}\notag
      d_E(u,v)\coloneq \int_{E}|u-v|\land 1\,{\rm d}x,
  \end{equation}
  where $a\land b\coloneq \min\{a,b\}$ for any $a,b\in\R$.
  
  Given an $\Lb^d$-measurable set $E\subset \Rd$, we denote by $\chi_E$ its characteristic function, defined by $\chi_E(x)=1$ if $x\in E$ and $\chi_E(x)=0$ otherwise. If $E$ is of finite perimeter, we denote by $\partial^*E$ its reduced boundary.  Moreover, given \EEE  $x\in\Rd$ and $\rho>0$,  we set 
\begin{equation}\label{def blow up set}
    \Omega_{\rho,x}\coloneq x+\rho(\Omega-x).
\end{equation}
  The collection of all open (resp.\ Borel) subsets of $\Omega$ is denoted by $\Op(\Omega)$ (resp.\ $\B(\Omega$)). Given two open sets $O',O''\subset \Rd$ we write $O'\subset \subset O''$ if  $\overline{O'} \subset O''$. \EEE   
\medskip 
\medskip

\noindent \MMM \textbf{(Generalised) functions of  bounded variation.}  \EEE
We recall the definition and the main properties of the function spaces \BBB used in this work. \EEE In particular, we present the spaces ${\rm GBV}(\Omega;\Rk)$ and ${\rm GBV}_\star(\Omega;\Rk)$, the first introduced by  De Giorgi \& Ambrosio \cite{DeGioAmbro}  and later investigated in \cite{AmbrosioExistence}\EEE, the latter by Dal Maso \& Toader \cite{DalToa22} in the scalar setting and later generalised to the vector-valued case in \cite{DonatiGBV}. \BBB We give only the basic definitions here and refer the reader to Appendix \ref{sec: GBV} for further properties, including \EEE  the definition of approximate gradient and jump set of $\Lb^d$-measurable functions.\EEE

The space $\BV(\Omega;\Rk)$ is the space of all functions $u\in L^1(\Omega;\Rk)$ whose distributional derivative $Du$ belongs to $\mathcal{M}_b(\Omega;\Rkd)$. We recall that the measure $Du$ can be decomposed as 
\begin{equation*}
Du=\nabla u\Lb^d +D^cu+([u]\otimes\nu_u)\Hd\mres J_u,\end{equation*}
where 
\begin{itemize}
    \item $\nabla u\in L^1(\Omega;\Rkd)$ is the {\it approximate gradient} of $u$  (see \eqref{approximate grad});\EEE
\item the {\it Cantor part} $D^cu$ is a measure that is singular with respect to $\Lb^d$ and that vanishes on all Borel sets $B\in\mathcal{B}(\Omega)$ that are $\sigma$-finite with respect to $\Hd$;
\item  $[u]\coloneq u^{+}-u^-$ is the difference of the two unilateral traces of $u$ on the jump set $J_u$, whose unit normal is $\nu_u$ (see   \eqref{def Jump set}).\EEE
\end{itemize}
The space ${\rm GBV}(\Omega;\Rk)$ is the space of all functions $u\in L^0(\Omega;\Rk)$ such that $\Phi\circ u\in \BV_{\rm loc}(\Omega;\Rk)$ for every Lipschitz function $\Phi\colon\Rk\to \Rk$ whose gradient $\nabla \Phi$ has compact support. We recall that $\GBV(\Omega;\Rk)$ is not a vector space and that, in particular, $\GBV(\Omega;\Rk)\MMM \supsetneq \EEE \GBV(\Omega)^{k}$ if $k> 1$ (see \cite[Example~4.27]{AFP}). 

 It is well-known that every function $u\in {\rm GBV}(\Omega;\Rk)$ admits an approximate gradient $\nabla u $  (see {\rm (b)} of Proposition~\ref{Prop:fine prop GBV}).
It can also be shown (see {\rm (d)} of Proposition~\ref{Prop:fine prop GBV}) that \BBB there is \EEE a positive (possibly unbounded) measure $|D^cu|$, which generalises the total variation of the Cantor part of the distributional gradient of a function of bounded variation. 

We now recall the definition of the space $\GBVs(\Omega;\Rk)$.
To this end, we introduce the functional $\mathcal{V}\colon L^0(\Omega;\Rk)\times \B(\Omega)\to [0,+\infty]$ which  for every $O\in\Op(\Omega)$ is defined by  
\begin{equation}\label{Def V BV}
\mathcal{V}(u,O)\coloneq 
\begin{cases}\displaystyle
    \sum_{i=1}^k\Big(\int_O|\nabla u_i|\,{\rm d}x+|D^cu_i|(O)+\int_{J_{u_i}\cap O}|[u_i]|\land 1\,{\rm d}\Hd \Big)& \text{ if $u\in \GBV(O)^k$,}\\
    +\infty & \text{ otherwise}.
    \end{cases}
\end{equation}
\BBB Here, \EEE $u_i$ denotes the $i$-th component of $u$. 
The definition is then extended to every Borel set $B\in\B(\Omega)$ by setting
\begin{equation}\label{finalVdef}
    \mathcal{V}(u,B)\coloneq \inf\{\mathcal{V}(u,O)\colon \,  O\supset B,\, O\in\Op(\Omega)\}.
\end{equation}
\begin{remark}\label{remark: lowersemicontinuity}
  It follows from   \cite[Theorem~6.1]{DalToa22}   that for every $O\in\Op(\Omega)$ the functional $\mathcal{V}(\cdot,O)$ is lower semicontinuous with respect to the $L^0(\Omega;\Rk)$-convergence.
\end{remark}

For $O\in\Op(\Omega)$,   the space $\GBVs(O;\Rk)$ is   defined as 
\begin{equation}\label{def canonica}
   \GBVs(O;\Rk)\coloneq \big\{u\in {\rm GBV}(O)^k\colon \, \mathcal{V}( u,O)<+\infty\big\},
\end{equation}
or, equivalently, as 
\begin{align*}
    {\rm GBV}_\star(O;\Rk)\coloneq \{u\in {\rm GBV}(O)^k&\colon\, \nabla u\in L^1(O;\Rkd), |D^cu|\in \mathcal{M}_b(O),\\
    &\ |[u]|\in L^1_{\Hd}(J_u\setminus J^1_u), \text{ and }\Hd(J^1_u)<+\infty\},
\end{align*}
where, for $r\geq0$, we set 
\begin{equation}\label{def Jru}
    J^r_u\coloneq \{x\in J_u\colon\,|[u]|\geq r\}.
\end{equation}
 It was shown in  \cite[Theorem~3.9]{DalToa22}  and \cite[Proposition~3.3]{DonatiGBV} that ${\rm GBV}_\star(O;\Rk)$ is a vector space. Moreover, it follows from  definition \eqref{def canonica} that for every $u\in \GBVs(O;\Rk)$ and $i\in\{1,\dots,k\}$ the $i$-th component $u_i$ belongs to $ \GBVs(O)$. 

\BBB To \EEE  each  $u\in \GBVs(O;\Rk)$, one can associate a {\it  matrix-valued} measure $D^cu\in\mathcal{M}_b( O\EEE;\Rkd)$ that generalises the Cantor part of the distributional derivative for functions of bounded variation. We remark that the total variation of $D^cu$ coincides with the measure $|D^cu|$ introduced before for functions in ${\rm GBV}$  (see {\rm (d)} of Proposition~\ref{Prop:fine prop GBV}). We refer the reader to Proposition~\ref{Prop:Cantor GBV} for the precise definition of this measure. \EEE

\section{Setting of the problem and main results}\label{setting and main results}

In this section, we present the  setting and the \EEE main results of the paper. We begin by introducing the class of structured deformations that will be our main object of investigation.

\medskip 

\noindent {\bf Structured Deformations.} Throughout   the work,   let $p \ge 1$.  We specify only when needed if $p=1$ or $p>1$. We recall that $\Omega\subset \Rd$ always denotes a bounded, open set with Lipschitz boundary.  Given $O\in  \mathcal{O}(\Omega) \EEE $, the class of {\it  generalised  structured deformations}   $\SD^p_\star(O;\Rk)$   is defined by setting
\begin{equation*}
   {\rm SD}_\star^p(O;\Rk)\coloneq \big\{(g,G)\colon \,g\in {\rm GBV}_\star(O;\Rk)\, \text{ and }G\in L^p(O;\Rkd)\big\}.
\end{equation*}
 We use the term `generalised' as the deformations $g$ belong to a class  of generalised ${\rm BV}$-functions.   
The first main result of this work is  the following approximation result for structured deformations $\SD_\star^p(\Omega;\Rk)$. \MMM Recall the definition of $\mathcal{V}$
 in \eqref{finalVdef}. \EEE \begin{theorem}[Approximation Theorem\EEE]\label{thm:approximation SDstar}
\BBB There exists a constant $C=C(d)>0$ such that for all $O\in\Op(\Omega)$  with Lipschitz boundary  and all \EEE $(g,G)\in \SD^p_\star(O;\Rk)$ we can find a sequence $\{u_n\}_n\subset {\rm SBV}(O;\Rk)$ such that $\{u_n\}_n$ converges to $g$  in $L^0(O;\Rk)$ as $n\to+\infty$,  $\nabla u_n=G$ $\Lb^d$-a.e. \BBB in $O$ \EEE for all $n\in\N$, and 
\begin{equation}\label{claim main approximation theorem}
\limsup_{n\to+\infty}\mathcal{V}(u_n,O)\leq  C\Big(\mathcal{V}(g,O)+\|G\|_{L^1(O)}\Big).
\end{equation}
\end{theorem}
 We briefly comment on this result. \BBB  If \eqref{claim main approximation theorem} is dropped, \EEE the theorem is an immediate consequence of \cite[Theorem~1.2]{Silhavy2015} (see also \cite[Theorem~2.12]{ChoksiFonseca}). Indeed, using the truncation properties of ${\rm GBV}_\star$-functions described in Proposition~\ref{prop:truncation}, we can approximate $g$ by means of ${\rm BV}$-functions, to which \cite[Theorem~1.2]{Silhavy2015} applies. A diagonal argument then gives a sequence converging to $g$ in $L^0(O;\Rk)$ \BBB which satisfies \EEE the required gradient constraints, but which may fail to satisfy \eqref{claim main approximation theorem}. To overcome this difficulty, we will depart from the results currently available in the literature, and rely on new density results for functions of bounded variation tailored to our functional setting, namely Theorem~\ref{thm: Goffmann-Serrin} and Lemma~\ref{lemma piecewise constants}. In particular, \BBB Theorem~\ref{thm: Goffmann-Serrin} provides \EEE an approximation result for ${\rm BV}$-functions by means of ${\rm SBV}$-functions that applies to a general class of energies, including $\mathcal{V}$ (see Remark~\ref{remark: vVec}(ii)).  We hope that this result, which is of independent interest, may find applications also in \BBB other \EEE contexts, for instance, when addressing the $\Gamma$-limsup inequality in the asymptotic analysis of free discontinuity problems. \EEE 

We endow $\SD^p_\star(\Omega;\Rk)$ with the following notion of convergence.
\begin{definition}\label{def: convergences}
   We say that a sequence $\{(g_n,G_n)\}_n\subset \SD^p_\star(\Omega;\Rk)$ converges to $(g,G)\in \SD^p_\star(\Omega;\Rk)$ as $n\to+\infty$, in symbols $(g_n,G_n)\SDtostar(g,G)$  in $\Omega$,   if and only if $g_n\to g$ in $L^0(\Omega;\Rk)$ and  
\begin{equation*}
    G_n\rightharpoonup G \text{ in $L^p(\Omega;\Rkd)$ if $p>1$}\quad \text{ and }\quad  G_n \rightharpoonup^* G \text{ weakly$^*$ in $\mathcal{M}_b(\Omega;\Rkd)$ if $p=1$}
    \end{equation*}
    as $n\to+\infty.$  
    
     Given a sequence $\{u_n\}_n\subset {\rm SBV}(\Omega;\Rk)$, we say that $\{u_n\}_n$ converges to $(g,G)\in\SD^p_\star(\Omega;\Rk)$,   in symbols $u_n\SDtostar(g,G)$, if $(u_n,\nabla u_n)\SDtostar(g,G)$.\EEE
\end{definition}

\noindent{\bf Integral representation for abstract functionals.}
We  present a result concerning the integral representation of   functionals defined on  ${\rm SD}^p_\star(\Omega;\Rk)$. We begin by introducing the class $\F^p_\star$, which will be one of the main objects of study in this paper.

\begin{definition}\label{def: Functionals}
   $\mathfrak{F}^p_\star$ is the class of all functionals $\mathcal{F}\colon {\rm SD}^p_\star(\Omega;\Rk)\times\B(\Omega) \to [0,+\infty)$ satisfying the following conditions: 
\begin{enumerate}[label=\textup{(${\rm H}$\arabic*)}, start=1]
    \item \label{hyp:H1}{\rm (}{\it Measure property}{\rm )} for every $(g,G)\in {\rm SD}^p_\star(\Omega;\Rk)$ the set function $\mathcal{F}(g,G,\cdot)$ is a Radon measure on $\B(\Omega)$;
    \item \label{hyp:H2}{\rm (}{\it Locality}{\rm )} for every $O\in\Op(\Omega)$ and \MMM pairs \EEE $(g,G), (u,U)\in {\rm SD}^p_\star(\Omega;\Rk)$, with $g=u$ and  $G=U$
 $\Lb^d$-a.e.\ in $O$, it holds $\mathcal{F}(g,G,O)=\mathcal{F}(u,U,O)$;
    \item \label{hyp:H3} {\rm (}{\it Lower semicontinuity}{\rm )} for every $O\in\mathcal{O}(\Omega)$ the functional $\mathcal{F}(\cdot,\cdot,O)$ is lower semicontinuous with respect to the convergence in ${\rm SD}^p_\star(\Omega;\Rk)$;  
    \item \label{hyp:H4}{\rm (}{\it Growth conditions}{\rm )} there exist positive constants $\alpha,\beta$, with $\alpha\leq \beta$, such that 
 \begin{equation}\label{upper bound}
\alpha\Big(\mathcal{V}(g,B)+\|G\|^p_{L^p(B)}\Big)\leq \mathcal{F}(g,G,B)\leq \beta \Big(\mathcal{V}(g,B)+\|G\|^p_{L^p(B)}+\Lb^d(B)\Big)
 \end{equation}
 for every $(g,G)\in {\rm SD}^p_\star(\Omega;\Rk)$ and $B\in\B(\Omega)$.
\end{enumerate}
\end{definition}

\begin{remark}
 Combining properties \ref{hyp:H1} and \ref{hyp:H2} we obtain that every functional $\cF\in\mathfrak{F}^p_\star$ satisfies the following locality property: if $B\in\B(\Omega)$, $(g,G), (u,U)\in {\rm SD}^p_\star(\Omega;\Rk)$, and $g=u$  and $G=U$ $\Lb^d$-a.e.\ \MMM in \EEE  a neighbourhood of $B$, then $\cF(g,G,B)=\cF(u,U,B)$.
 
    For this reason, given $O\in\Op(\Omega)$ with Lipschitz boundary and   $(g,G)\in {\rm SD}^p_\star(O;\Rk)$, for every $B\in\B(O)$ we can define  $\cF(g,G,B)\coloneq \cF(u,U,B)$, where $(u,U)\in \SD^p_\star(\Omega;\Rk)$ is any pair such that $g=u$ and $G=U$ $\Lb^d$-a.e.\ on $O$. For instance, this extension can be performed by using Lemma \ref{lemma: Extension of GBV}. \EEE The locality property described above implies that the value of $\cF(g,G,B)$  does not depend on the chosen extension $(u,U)$. 
\end{remark}

In Section \ref{section: Global Method} we  prove \EEE an integral representation result for  functionals in \EEE  $\F^p_\star$. To give the precise statement of this result, we introduce a decomposition of elements in $\F^p_\star$ that mimics the usual decomposition of the distributional derivative for functions of bounded variation.

\begin{definition}\label{decomposition}
    Let  $\cF\in \mathfrak{F}^p_\star$ and $(g,G)\in\SD^p_\star(\Omega;\Rk)$. By \ref{hyp:H1} and \ref{hyp:H4}, $\mathcal{F}(g,G,\cdot)$ is a bounded Radon measure on $\B(\Omega)$.  Its absolutely continuous and singular parts with respect to $\Lb^d$ are denoted by $\mathcal{F}^a(g,G,\cdot)$ and $\mathcal{F}^s(g,G,\cdot)$, respectively.  We also introduce the bounded Radon measures $\mathcal{F}^c(g,G,\cdot)$ and $\mathcal{F}^j(g,G,\cdot)$ defined by 
    \begin{gather*}
\mathcal{F}^c(g,G,B)\coloneq \mathcal{F}^s(g,G,B\setminus J_g) \text{ and }\mathcal{F}^j(g,G,B)\coloneq \mathcal{F}(g,G, B\cap J_g)
    \end{gather*}
    for every $B\in\B(\Omega)$.
We \MMM obtain \EEE 
\begin{equation*}
\mathcal{F}(g,G,B)=\mathcal{F}^a(g,G,B)+\mathcal{F}^c(g,G,B)+\mathcal{F}^j(g,G,B)
\end{equation*}
for every $B\in \B(\Omega)$. 
\end{definition}

\begin{remark}
    Note that by Remark~\ref{remark: lowersemicontinuity} the functional $(g,G,B)\mapsto \mathcal{V}(g,B)+\|G\|^p_{L^p(B)}$ belongs to $\F^p_\star$. Moreover, if $\cF\in\F^p_\star$ and $(g,G)\in \SD^p_\star(\Omega;\Rk)$, then  
    \begin{align*}
        &\text{$\cF^a(g,G,\cdot)$ is absolutely continuous with respect to $\Lb^d$,}\\
        &\text{$\cF^c(g,G,\cdot)$ is absolutely continuous with respect to $|D^cg|$,}\\
&\text{$\cF^j(g,G,\cdot)$ is absolutely continuous with respect to $ \mathcal{V}(g,\cdot)\mres J_g$.   }
    \end{align*}
\end{remark}
Our goal is to represent  $\mathcal{F}^a$ and $\mathcal{F}^j$ as integral functionals. As customary in integral representation results for free discontinuity functionals, the bulk and surface integrands \MMM are \EEE obtained by  considering \EEE the asymptotic behaviour of the minimum values of certain constrained problems on small cubes.  To this end, given $(g,G)\in\SD^p_\star(\Omega;\Rk)$ and $O\in\Op(\Omega)$ with Lipschitz boundary, we set
\begin{equation*}
    \C^p_\star(g,G,O)\coloneq \Big\{(u,U)\in \SD^p_\star(\Omega;\Rk)\colon \,  u=g \, \text{ near }\partial O, \, \int_O(G-U)\,{\rm d}x=0\Big\}
\end{equation*}
and consider the auxiliary minimisation \MMM problem \EEE  given by 
\begin{gather}\label{minimisation problem}
    \m(g,G,O)\coloneq \inf \Big\{\mathcal{F}(u,U,O)\colon \, (u,U)\in \C^p_\star(g,G,O)\Big\}.
\end{gather}
 The boundary conditions  that we generally \EEE  use for these minimisation problems are given by   affine functions \EEE and pure jump functions: for $x\in \Omega$, $ a \in \Rk$, and  $A\in \mathbb{R}^{k\times d}$ we set $\ell_{x,a,A}(y)\coloneq  a+ A(y-x)$, and for $x\in \Omega$, $ \zeta, \theta \in \mathbb R^k$,  $\nu \in \mathbb S^{d-1}$ we let 
\begin{align}\label{stepfun}
v_{x,\zeta,\theta, \nu}(y) \coloneq  \begin{cases} \zeta &\hbox{if } (y-x)\cdot \nu > 0,\\
\theta &\hbox{if } (y-x)\cdot \nu \leq 0.\end{cases}
\end{align}
Our candidate bulk and surface integrands are given by  
\begin{align}\label{def f}
f_{\rm bulk}(x, a, A, L)\coloneq 	\limsup_{\rho \to 0^+}\frac{\m(\ell_{x,a,A}, L, Q(x,\rho))}{\rho^d}
	\end{align}
	 for all $x\in \Omega$, $ a \in \mathbb R^k$,   $A, L\in \Rkd$,   and
\begin{align}\label{def Psi}
	f_{\rm surf}(x, \zeta, \theta, \nu)\coloneq  \limsup_{\rho \to 0^+}\frac{\m(v_{x,\zeta, \theta,\nu}, 0, Q_{\nu}(x,\rho))}{\rho ^{d-1}}
	\end{align}
for all $x\in \Omega$, $ \zeta, \theta \in \Rk$, and $\nu \in \mathbb S^{d-1}$, where we recall that $Q_\nu(x,\rho)$ is a cube centred at $x$, of side-length $\rho$, and with two faces orthogonal to $\nu$, while $Q(x,\rho)\coloneq x+(-\rho/2,\rho/2)^d.$

The following  theorem \EEE constitutes our main result concerning the class $\F^p_\star$. \BBB It \EEE  gives   an integral representation of the bulk and surface parts   by \EEE means of the functions $f_{\rm bulk}$ and $f_{\rm surf}$ above.\EEE 

\begin{theorem}[Integral representation]\label{thm: integral star}
     Let $\mathcal{F}\in\F^p_\star$. 
Then, for every $(g,G) \in \SD^p_\star(\Omega;\Rk)$ and $B \in \B(\Omega)$  it  holds 
	\begin{align}\label{repr a star thm}
	&\mathcal{F}^a(g,G,B)= \int_B \!f_{\rm bulk}\big(x,g,   \nabla  g, G\big)\,{\rm d} x,\\& \label{repre j star thm}\mathcal{F}^j(g,G,B)=\int_{J_g\cap B}\!\!\!\!\!f_{\rm surf}\big(x, g^+, g^-,\nu_g\big) 
	\,{\rm d} \Hd,
    \end{align}
where $f_{\rm bulk}$ and $f_{\rm surf}$ are given by \eqref{def f} and \eqref{def Psi}. In particular, if $(g,G) \in \SD^p_\star(\Omega;\Rk)$ \BBB with \EEE  $|D^cg|(\Omega)=0$, then for all $B\in\B(\Omega)$ \BBB it holds that \EEE 
\begin{equation*}
    \cF(g,G,B)=\int_B \!f_{\rm bulk}\big(x,g,   \nabla  g, G\big)\,{\rm d} x+\int_{J_g\cap B}\!\!\!\!\!f_{\rm surf}\big(x, g^+, g^-,\nu_g\big) 
	\,{\rm d} \Hd.
\end{equation*}
\end{theorem}

\begin{remark}\label{remark: invariance}
  \MMM If \EEE $\cF\in \F^p_\star$ is invariant under   translation,  the functions \MMM $f_{\rm bulk}$  \EEE and $f_{\rm surf}$ have  simpler expressions. More precisely, if 
    \begin{equation}\label{invariance} 
        \cF(g+a,G,B)= \cF(g,G,B)
    \end{equation}
    for all $(g,G)\in\SD^p_\star(\Omega;\Rk)$, $a\in\Rk$, and $B\in \B(\Omega)$, then 
    \begin{equation}\notag 
        f_{\rm bulk}(x,a,A,L)=f_{\rm bulk}(x,0,A,L)\quad \text{ and }\quad 
f_{\rm surf}(x,\zeta,\theta,\nu)=f_{\rm surf}(x,\zeta-\theta,0,\nu)
    \end{equation}
    for all $x\in\Omega$, $a\in\Rk$, $A,L\in\Rkd$, $\zeta,\theta\in\Rk$, and $\nu\in\Sd$. To simplify the notation, when  the invariance property \eqref{invariance} holds, we will write $f_{\rm bulk}(x,A,L)$ and $f_{\rm surf}(x,\zeta-\theta,\nu)$ in place of $f_{\rm bulk}(x,0,A,L)$ and $f_{\rm surf}(x,\zeta-\theta,0,\nu)$.
\end{remark}

\noindent{\bf Relaxation of \MMM bulk-surface \EEE energies.}  Our goal is to employ the \MMM integral representation result in \EEE  Theorem~\ref{thm: integral star}  to prove a relaxation result for   energies defined on ${\rm SBV}(\Omega;\Rk)$, in the spirit of the seminal work \cite{ChoksiFonseca}. More precisely, we   consider a bulk integrand $\Psi\colon\Rd\times \Rkd\to [0,+\infty)$ and a surface integrand $\gamma\colon \Rd\times \Rk\times \Sd\to [0,+\infty)$, and the associated integral functional $\MMM \mathcal{E} \EEE \colon {\rm SBV}(\Omega;\Rk)\times \B(\Omega)\to [0,+\infty)$ defined by 
\begin{equation*}
\mathcal{E}(u,B)\coloneq \int_{B}\Psi(x,\nabla u)\,{\rm d}x+\int_{J_u\cap B}\gamma(x,[u],\nu_u)\,{\rm d}\Hd
    \end{equation*}
for all $(u,B)\in {\rm SBV}(\Omega;\Rk)\times \B(\Omega)$.

Under suitable hypotheses on $\Psi$ and $\gamma$  specified below, \EEE we  study the functional
\begin{equation}\label{def Istar}
I^p_\star(g,G)\coloneq \inf\Big\{\liminf_{n\to+\infty} \mathcal{E}(u_n,\Omega)\colon \, 
\{u_n\}_{n}\subset {\rm SBV}(\Omega;\Rk)\text{ with }  \MMM u_n \EEE \SDtostar(g,G) \text{ in }\Omega \Big\} 
\end{equation}
for all $(g,G)\in \SD^p_\star(\Omega;\Rk)$.  To this end, 
for $O\in \Op(\Omega)$  we   consider   the localised functional $I^p_\star(\cdot,\cdot,O)$ defined  for all $(g,G)\in\SD^p_\star(\Omega;\Rk)$ by 
\begin{equation}\label{def Istar localised}
\hspace{-0.3 cm}I^p_\star(g,G,O)\coloneq \inf\Big\{\liminf_{n\to+\infty} \mathcal{E}(u_n,O)\colon \,
\{u_n\}_{n}\subset {\rm SBV}(O;\Rk)\text{ with }  u_n\SDtostar(g,G) \text{ in } O \Big\}.
\end{equation}
 We then extend this definition to every Borel set $B\subset \Omega$ by setting 
\begin{equation}\label{def Borel relax}
\hspace{-0.3 cm}I^p_\star(g,G,B)\coloneq \inf\big\{ I^p_\star(g,G,O)\colon \,  O\supset B \text{ with }O\in\Op(\Omega) \big\}.
\end{equation}
It will be shown in Theorem \ref{thm: relaxation partial} that, under suitable hypotheses on $\Psi$ and $\gamma$, the set function $O\mapsto I^p_\star(g,G,O)$ given by \eqref{def Istar localised} coincides with the trace on $\Op(\Omega)$ of a bounded Radon measure, so that definitions \eqref{def Istar localised} and \eqref{def Borel relax} agree on $\Op(\Omega)$. \EEE

We now introduce the sets of assumptions we will make on $\Psi$ and $\gamma$.   Whereas the \MMM assumptions \EEE on $\Psi$ are modelled after those in \cite{ChoksiFonseca}, for the surface density \MMM we  \EEE fundamentally depart from the usual assumption of 1-homogeneous densities by considering bounded densities which are close to  1-homogeneous functions  at the origin. \EEE
We assume that $\Psi$ satisfies the following properties:
\setlist[enumerate]{label=($\Psi$\arabic*)}
\begin{enumerate}
\item \label{(W1)_p} (measurability) $\Psi$ is a Carath\'eodory function;
\item \label{(W2)_p}  (Lipschitz continuity) there is a   constant $C_\Psi>0$ such that
	\begin{equation*}
		|\Psi(x,A_1) - \Psi(x,A_2)| \leq C_\Psi |A_1 - A_2| \big(1+|A_1|^{p-1}+|A_2|^{p-1}\big)
	\end{equation*}
    for $\Lb^d$-a.e.\ $x\in\Omega$ and all $A_1,A_2 \in  \Rkd$;
\item \label{W3} (control from above) there is $A_0 \in \mathbb R^{k\times d}$ such that 
$\Psi(\cdot, A_0)\in L^\infty(\Omega)$;  
\item \label{W4} (lower bound)  there is a   constant $c_\Psi>0$ such that  
	\begin{equation*}
		c_\Psi |A|^{p}-\frac{1}{c_\Psi}\leq \Psi(x,A)
	\end{equation*}
    for $\Lb^d$-a.e.\ $x \in \Omega$ and all $A \in \Rkd$;
\item\label{Wcontinuous}  
(continuity in $x$)  there is a continuous function $\omega_\Psi\colon[0,+\infty)\to[0,+\infty)$, with $\omega_\Psi(0)=0$, such that for every $x_1,x_2\in\Omega$ and $A\in\Rkd$ it holds
		\begin{equation*}
	|\Psi(x_1,A)-\Psi(x_2,A)|\leq\omega_\Psi(|x_1-x_2|)\big(\Psi(x_1,A)+\Psi(x_2,A)\big).
		\end{equation*}
\end{enumerate}

\begin{remark}\label{remark: Improved upper bound}
    Any function $\Psi$ satisfying  {\rm \ref{(W2)_p}} and {\rm \ref{W3}}  also satisfies the following property: 
    \begin{itemize}
        \item[{\rm ($\Psi$3')}](upper bound) there exists a   constant $\overline{C}_\Psi >0$ such that         \begin{equation}\label{bound above W} 
            \Psi(x,A)\leq \overline{C}_\Psi\big(|A|^p+1)
        \end{equation}
        for $\Lb^d$-a.e.\ $x\in\Omega$ and all $A\in\Rkd.$
    \end{itemize}
    Indeed, letting $A_0$ be the matrix of {\rm \ref{W3}}, we may use {\rm \ref{(W2)_p}} to obtain for all $A\in\Rkd$
    \begin{equation*}
        \Psi(x,A)\leq \Psi(x,A_0)+ C_\Psi(|A| + |A_0| + |A|^p +|A_0|^p +|A||A_0|^{p-1} +|A_0||A|^{p-1}).
    \end{equation*}
    In light of {\rm \ref{W3}}, this inequality implies  \eqref{bound above W} for $\Lb^d$-a.e.\ $x\in\Omega$ and for all $A\in\Rkd$, for a  sufficiently \EEE large constant $\overline{C}_\Psi$ depending only on $A_0$, $C_\Psi$, $\Omega$, and $p$. 
\end{remark}

In the following list we collect the assumptions we make on the surface integrand $\gamma$:
\setlist[enumerate]{label=($\gamma$\arabic*)}
\begin{enumerate}

\item\label{(gamma1)} (symmetry) for every $x \in \Omega$, $\zeta \in  \Rk$, and 
$\nu \in \Sd$ it holds
	\begin{equation*}
		\gamma (x, \zeta, \nu)= \gamma (x,-\zeta, -\nu);
	\end{equation*}
\item\label{(gamma2)} (upper and lower bounds) there are two \MMM constants \EEE   $c_\gamma,C_\gamma >0$ such that for all $x\in\Omega$, $\zeta \in \Rk$, and $\nu \in \Sd$ it holds   
	\begin{equation*}
		c_\gamma (|\zeta| \wedge 1) \leq \gamma(x,\zeta, \nu) \leq C_\gamma (|\zeta|  \wedge 1);
	\end{equation*}
\item \label{(gamma3)} (sub-additivity) for all $x\in\Omega$, $\zeta_1$, $\zeta_2 \in  \Rk$, and $\nu \in \Sd$ it holds
		\begin{equation*}
			\gamma(x, \zeta_1 + \zeta_2, \nu) \leq \gamma(x,\zeta_1, \nu) +\gamma(x,\zeta_2, \nu);
		\end{equation*}
 \item \label{(gamma4)} (continuity in $x$)  there is a continuous function $\omega_\gamma\colon[0,+\infty)\to[0,+\infty)$, with $\omega_\gamma(0)=0$, such that for every $x_1,x_2\in\Omega$, $\zeta \in \Rk$, and $\nu \in \Sd$ it holds
		\begin{equation*}
	|\gamma(x_1,\zeta,\nu)-\gamma(x_2,\zeta,\nu)|\leq\omega_\gamma(|x_1-x_2|)(|\zeta|\land 1);
		\end{equation*}
        \item \label{(gamma5)}(quasi-monotonicity) there exists a constant $\beta_\gamma\geq 1$ such that for every $x\in\Omega$  and $\nu\in\Sd$ it holds
        \begin{equation*}
           \gamma(x,\zeta_1,\nu)\leq \gamma(x,\zeta_2,\nu)
        \end{equation*}
        whenever $\beta_\gamma|\zeta_1|\leq |\zeta_2|$;  
\item\label{(gamma6)} (rate of growth at the origin) there exists a continuous nondecreasing function $\vartheta\colon [0,+\infty)\to [0,+\infty)$, with $\vartheta(0)=0$ and $\vartheta(h)\geq \frac{c_\gamma}{C_\gamma}\EEE h-1$ for all $h\geq 0$,  with the constants  $c_\gamma,C_\gamma >0$  of \ref{(gamma2)},  such that for all $x\in\Omega,$ $\zeta\in\Rk$, and $\nu\in\Sd$  the limit
    \begin{equation}\label{existence gamma zero}
\gamma^0(x,\zeta,\nu)\coloneq \lim_{s\to 0^+}\frac{1}{s}\gamma(x,s\zeta,\nu)
    \end{equation}
    exists and \BBB satisfies \EEE
    \begin{equation*} \big|\frac{1}{s}\gamma(x,s\zeta,\nu)-\gamma^0(x,\zeta,\nu)\big|\leq \vartheta(s|\zeta|)\frac{1}{s}\gamma(x,s\zeta,\nu)
    \end{equation*}
    for all $s>0$.\EEE
\end{enumerate}

\begin{remark}[\MMM Assumptions \ref{(gamma1)}--\ref{(gamma5)}\EEE]
\MMM Properties \ref{(gamma1)} and \ref{(gamma3)} are standard assumptions on surface densities, while the growth condition \ref{(gamma2)} is a core ingredient in our study, see \eqref{Intro: cohesive growth} in the introduction for the motivation. The continuity condition \ref{(gamma4)} is also a standard assumption, slightly adjusted to our setting by using $(|\zeta|\land 1)$ in place of $|\zeta|$. \EEE 
Property \ref{(gamma5)} was introduced in \cite{CagnettiDet} and was  later employed in several other works (see, for instance, \cite{CagnettiStoc,FriedrichCompactness,DalToa22,dal2025homogenization}).  In that paper, it is also shown (see \cite[Remark~3.2]{CagnettiDet} and \cite[Remark~3.4]{dal2025homogenization}) that this property  is weaker than monotonicity in $|\zeta|$ of $\gamma$.
\end{remark}

\begin{remark}[\MMM Assumption \ref{(gamma6)}\EEE]
 A condition closely related to \ref{(gamma6)} was already considered in \cite[Remark~3.5]{CagnettiBV} or \cite[Remark~3.3]{ChoksiFonseca}, and was first introduced in this \BBB exact \EEE form in \cite{DalToaHomo, dal2025homogenization}, \MMM see in particular \cite[(4.9)]{dal2025homogenization}. \EEE   In \cite[Remark~4.7]{DalToaHomo} and \cite[Remark~6.5]{dal2025homogenization}  \BBB it is noted \EEE that condition $\vartheta(h)\geq \frac{c_\gamma}{C_\gamma}\EEE h-1$ for all $h\geq 0$ is necessary and sufficient for the existence of functions $\gamma$ satisfying condition {\rm \ref{(gamma6)}}.  \EEE We also observe  that  $\gamma^0$ is positively $1$-homogeneous, i.e.,
\begin{equation}\label{homogeneity of gamma0}
\gamma^0(x,t\zeta,\nu)=t\gamma^0(x,\zeta,\nu)
\end{equation}
    for all $x\in\Omega$, $\zeta\in\Rk$, $\nu\in\Sd$, and $t\geq 0$. Note that, together with \ref{(gamma2)}, \BBB \eqref{homogeneity of gamma0} \EEE implies  
    \begin{equation}\label{bounds gammazero}
        c_\gamma|\zeta|\leq \MMM \gamma^0 \EEE (x,\zeta,\nu)\leq C_\gamma|\zeta|
    \end{equation}
    for all $x\in\Omega$, $\zeta\in\Rk$, and $\nu\in\Sd$.
\end{remark}
 
\begin{remark}[Properties of $\BBB \gamma^0 \EEE $]
\label{remark:gamma zero uniform}
  The integrand $\BBB \gamma^0 \EEE$ satisfies properties closely related to \ref{(gamma3)}--\ref{(gamma4)}.  Indeed,  it follows directly from \ref{(gamma3)} and \ref{(gamma6)} that $\BBB \gamma^0 \EEE$ satisfies \ref{(gamma3)} with $\gamma$ replaced by $\BBB \gamma^0 \EEE$, while  \ref{(gamma4)} and \ref{(gamma6)} imply that
    	\begin{equation}\label{uniform gamma zero} 
	|\gamma^0(x_1,\zeta,\nu)-\gamma^0(x_2,\zeta,\nu)|\leq\omega_\gamma(|x_1-x_2|)|\zeta|
		\end{equation}
for all $x_1,x_2\in\Omega$, $\zeta\in\Rk$, and $\nu\in\Sd$.    
\end{remark}
\EEE

The following is the first of the three results concerning the integral representation of \eqref{def Istar}. Note that we do not assume    \ref{Wcontinuous} and \EEE \ref{(gamma5)}--\ref{(gamma6)} \EEE nor put any restrictions on the value of~$p$.

\begin{theorem}[Relaxation I]\label{thm: relaxation partial}
Let $\Psi$ and $\gamma$ be integrands satisfying assumptions {\rm\ref{(W1)_p}}--{\rm \ref{W4}} and {\rm \ref{(gamma1)}--\ref{(gamma4)}}.  
Then {\rm(}up to a shift{\rm )} \EEE the functional $I^p_\star$ defined by \eqref{def Borel relax} belongs  to $\F^p_\star$   and   for all $(g,G)\in \SD^p_\star(\Omega;\Rk)$ and $B\in\B(\Omega)$ it holds
	\begin{align}\label{repr bulk relax}
	&(I^p_\star)^a(g,G,B)= \int_B \!f_{\rm bulk}\big(x,   \nabla  g, G\big)\,{\rm d} x,\\& \label{repre surf relax}(I^p_\star)^j(g,G,B)=\int_{J_g\cap B}\!\!\!\!\!f_{\rm surf}\big(x, [g], \nu_g\big) 
	\,{\rm d} \Hd,
    \end{align}
    where $f_{\rm bulk}$ and $f_{\rm surf}$ are given by \eqref{def f} and \eqref{def Psi}, replacing $\cF$ with $I^p_\star$ {\rm(}see also Remark~\ref{remark: invariance}{\rm)}.  In particular, if $(g,G) \in \SD^p_\star(\Omega;\Rk)$ \BBB with \EEE  $|D^cg|(\Omega)=0$, then  for all $B\in\B(\Omega)$ it holds that 
\begin{equation*}
    I^p_\star(g,G,B)=\int_B \!f_{\rm bulk}\big(x, \nabla  g, G\big)\,{\rm d} x+\int_{J_g\cap B}\!\!\!\!\!f_{\rm surf}\big(x, [g],\nu_g\big) 
	\,{\rm d} \Hd.
\end{equation*}

    \EEE
\end{theorem}

Our second main result gives a more explicit expression for the densities $f_{\rm bulk}$ and $f_{\rm surf}$ appearing in  \eqref{repr bulk relax} and \eqref{repre surf relax} when $p>1$. To present the result, we  first introduce some further notation. \EEE For $\zeta\in\R^{k}$ and $\nu\in\Sd$  we set  $v_{\zeta,\nu}\coloneq  v_{0,\zeta,0,\nu}$ (see \eqref{stepfun}). 
  Given $A,L\in\R^{k\times d}$,  $\zeta\in\Rk$, and $\nu\in\Sd$, we consider the classes  
  \begin{align}  
\C_{p,\star}^{\rm bulk}(A,L)\coloneq  \Big\{u\in {\rm SBV}(Q;\Rk) \colon \,   u=\ell_{0,0,A} \EEE &\text{ near $ \partial Q$}, \
\int_Q \nabla u\,{\rm d} x=L, \  \|\nabla u\|_{L^p(Q)}^p<+\infty \Big\}
\label{bulk competitors}
\end{align}
and 
\begin{equation}\label{surf competitors}
\hspace{-0.1 cm}\C_{p,\star}^{\rm surf}(\zeta,\nu)\coloneq  \Big\{u\in {\rm SBV}(Q_\nu;\Rk)\colon  \,  u=v_{\zeta,\nu} \text{  near $\partial Q_\nu$\EEE},  \   \nabla u=0 \ \text{ $\Lb^d$-a.e.\ on \EEE $ Q_\nu$}\Big\}.
\end{equation}
We set
\begin{equation}\label{bulk dens relax}
H_p(x,A,L)\coloneq  \inf_{u\in\C_{p,\star}^{\rm bulk}(A,L)}\Big\{  
\int_Q \Psi\big(x,\nabla  u(y)\big )\,{\rm d} y+\int_{J_u\cap Q} \gamma^0\big( x,[u](y),\nu_u(y)\big)\,{\rm d}\Hd(y)\Big\}
\end{equation}
 for all $x\in\Omega$ and $A,L\in\R^{k\times d}$, and 
\begin{equation}\label{surface dens p>1}
 h_p(x,\zeta,\nu)\coloneq  \inf_{u\in\C_{p,\star}^{\rm surf}(\zeta,\nu)}\Big\{ 
 \! \int_{J_u\cap Q_\nu}  \gamma\big(x,[u](y),\nu_u(y)\big)\,{\rm d}\Hd(y)\Big\}
\end{equation}
for all $x\in\Omega$, $\zeta\in\R^k$, and   $\nu\in\Sd$. 
\MMM Note that, in contrast to \eqref{surface dens p>1}, in \eqref{bulk dens relax} the density $\gamma^0$ appears which is the relevant quantity here as affine boundary conditions are effectively approximated by infinitesimally small jumps, cf.\ also \cite[Remark~3.3]{ChoksiFonseca}. \EEE 

The following constitutes a characterisation of $(I^p_\star)^a$ and $(I^p_\star)^j$ in terms of the densities $H_p$ and~$h_p$.  
We remark that the restriction $p>1$ is due to the nature of the truncation techniques we employ (see Lemma~\ref{lemma:convergence truncations}), which are ill-behaved when $p=1$. \BBB This is \EEE due to concentration phenomena that may occur when dealing with the weak$^*$-convergence of measures.

\begin{theorem}[Relaxation II]\label{thm:cell formulas}
      Let $p>1$ and let $\Psi$ and $\gamma$ be integrands satisfying assumptions  {\rm \ref{(W1)_p}--\ref{Wcontinuous}} \EEE  and {\rm \ref{(gamma1)}--\ref{(gamma6)}}.  Then,   for all $(g,G)\in \SD^p_\star(\Omega;\Rk)$ and $B\in\B(\Omega)$ equalities \eqref{repr bulk relax} and \eqref{repre surf relax} hold with $f_{\rm bulk}(x,\nabla g,G)$ and $f_{\rm surf}(x,[g],\nu_g)$ replaced by $H_p(x,\nabla g,G)$ and $h_p(x,[g],\nu_g)$, where $H_p$ and $h_p$ are the functions defined by \eqref{bulk dens relax} and \eqref{surface dens p>1}, respectively.

\end{theorem}
 
\begin{proposition}[Properties of $H_p$ and $h_p$]\label{prop:integrands}
 Let $p>1$ and let $\Psi$ and $\gamma$ be integrands satisfying assumptions  {\rm \ref{(W1)_p}--\ref{Wcontinuous}} \EEE and {\rm \ref{(gamma1)}--\ref{(gamma6)}}. Given $L\in\Rkd$, let $H_p^L\colon\Omega\times \Rkd\to [0,+\infty)$ be the function defined for $x\in\Omega$ and $A\in\Rkd$  by $ H^L_p(x,A)\coloneq H_p(x,A,L)$. Then $H^L_p$ satisfies  {\rm \ref{(W1)_p}--\ref{Wcontinuous}}
 with $p$ replaced by $1$, while $h_p$ satisfies {\rm \ref{(gamma1)}--\ref{(gamma5)}}.

 Moreover, for every $x\in\Omega$ and $A\in\Rkd$ there exists a constant $C>0$, depending only on $x$ and on $A$, such that 
 \begin{equation*}
     |H_p(x,A,L_1)-H_p(x,A,L_2)|\leq C|L_1-L_2|(1+|L_1|^{p-1}+|L_2|^{p-1})\quad \text{ for all $L_1,L_2\in\Rkd$}. \end{equation*}
Finally, there exist constants \BBB $c_H, C_H>0$ \EEE such that 
\begin{equation*}
    c_H(|A|+|L|^p)-\frac{1}{c_H}\leq H_p(x,A,L)\leq C_H(|A|+|L|^p+1)
\end{equation*}
for all $x\in\Omega$ and $A,L\in\Rkd$.
\end{proposition}

The final result of this paper provides a complete integral representation  for $I^{p}_\star$, including its Cantor part $(I^p_\star)^c$. This comes upon assuming a stronger hypothesis on $\Psi$ and $\gamma$, namely that $\Psi$ and $\gamma$ are independent of the $x$-variable.  This hypothesis ensures that the densities $H_p$ and $h_p$  are independent of $x$, and that $I^p_\star$ satisfies the following invariance property: for every $(g,G)\in\SD^p_\star(\Omega;\Rk)$ and $x\in\Rd$, it holds
\begin{equation}\label{independence day}
I^p_\star(g(\cdot-x),G(\cdot-x),x+B)=I^p_\star(g,G, B)
\end{equation}
for every $B\in\B(\Omega)$ with $x+B\subset \Omega$. This property makes possible to rely on the integral representation of Bouchitté, Fonseca, \& Mascarenhas \cite[Theorem~3.12]{BFM1998} for functionals defined on ${\rm BV}(\Omega;\Rk)$.  
\begin{theorem}[Relaxation III]\label{thm:relaxation full}
    Let $p>1$  and let $\Psi$ and $\gamma$  be $x$-independent integrands  satisfying assumptions {\rm \ref{(W1)_p}}--{\rm \ref{Wcontinuous}} \EEE and  {\rm \ref{(gamma1)}--\ref{(gamma6)}}.  Then, for every $(g,G)\in \SD_\star^p(\Omega;\Rk)$ and $B\in \B(\Omega)$ we have 
    \begin{equation}\notag 
I^p_\star(g,G,B)=\int_{B}H_p(\nabla g,   G)\,{\rm d}x+\int_{B}H_p^\infty\Big(\frac{{\rm d}D^cg}{{\rm d}|D^cg|},0\Big)\,{\rm d}|D^cg|+\int_{J_{g}\cap B}h_p([g],\nu_g)\,{\rm d}\Hd,
    \end{equation}
  where $H_p$ and $h_p$ are the functions defined by \eqref{bulk dens relax} and \eqref{surface dens p>1} without $x$-dependence, and \EEE  $H_p^\infty(\cdot,0)$ denotes the recession function of $H_p(\cdot,0)$, i.e., the function defined for all $A\in\Rkd$ by 
    \begin{equation}\label{def Recession}
    H_p^\infty(A,0)\coloneq \limsup_{t\to +\infty}\frac{H_p(tA,0)}{t}.
    \end{equation}
\end{theorem}

\begin{remark}[Compactness]\label{rem:compactness}
    Relying on arguments first devised in \cite{FriedrichCompactness}, some sufficient conditions for coerciveness of integral functionals on ${\rm GBV}_\star(\Omega;\Rk)$  were proved in \cite[Theorem~5.5]{DalToa22} and \cite[Theorem~4.8]{DonatiGBV}. More precisely, in those papers the authors consider  open sets $\Omega$, $\widetilde{\Omega}$,  with Lipschitz boundary  satisfying \EEE $\Omega\subset   \widetilde{\Omega}$,   a boundary datum $w\in W^{1,1}(\widetilde{\Omega};\Rk)$, and the functional  $\cF\colon {\rm GBV}_\star(\widetilde\Omega;\Rk)\to [0,+\infty)$ defined by  
    \begin{equation*} \cF(u)\coloneq \begin{cases}\displaystyle
    \int_{\widetilde{\Omega}}f (\nabla u)\,{\rm d}x+\int_{\widetilde{\Omega}}f^\infty \Big(\frac{{\rm d}D^cu}{{\rm d}|D^cu|}\Big)\,{\rm d}|D^cu|+\int_{ J_{u}\cap \widetilde{\Omega}}g ([u],\nu_u)\,{\rm d}\Hd &\text{if $u=w$  on $\widetilde{\Omega}\setminus \Omega$}, \\
    +\infty &\text{otherwise},
    \end{cases}
    \end{equation*}
    for some Borel functions $f \colon \Rkd\to [0,+\infty)$ and $g\colon\Rk\times \Sd\to [0,+\infty)$, with $f$ satisfying  
    \begin{gather*}
        c_F |A|-\frac{1}{c_F}\leq f (A)\leq C_F |A|+C_F\quad \text{for all $A\in\Rkd$},
    \end{gather*}
    for some   constants $c_F, C_F>0$, and with $g$ satisfying \ref{(gamma2)} and \ref{(gamma5)}. Under these hypotheses, they show that all sequences of functions $\{u_n\}_n\subset {\rm GBV}_\star(\widetilde{\Omega};\Rk)$ with $\sup_n\cF(u_n)<+\infty$ admit modifications $\{y_n\}_n\EEE\subset {\rm GBV}_\star(\widetilde{\Omega};\Rk)$ such that $\cF(y_n)\leq \cF(u_n)+1/n$   for all $n\in\N$,   and such that (a subsequence of) $\{y_n\}_n$ converges in $L^0(\widetilde{\Omega};\Rk)$ to some $u\in {\rm GBV}_\star(\widetilde{\Omega};\Rk)$ with $\cF(u)<+\infty$.

Taking into account some minor modifications due to the presence of the matrix variable $G$, 
 these compactness results readily  adapt to the functional $I^p_\star$, thanks to  Proposition~\ref{prop:integrands} and Theorem~\ref{thm:relaxation full}.  This procedure guarantees that $I^p_\star$ admits minimisers attaining the boundary condition $w$ on  $\widetilde{\Omega}\setminus \Omega$.   \EEE
\end{remark}

\section{Approximation Theorem in \texorpdfstring{${\rm SD}^p_\star(\Omega;\Rk)$}{SDp*(Ω;Rk)}:  Proof of Theorem \ref{thm:approximation SDstar}}\label{sec: Approximation Lemma}

This section is devoted to the proof of Theorem \ref{thm:approximation SDstar}. The approach we follow is similar to the one  in \cite[Theorem~1.2]{Silhavy2015}   (see also \cite[Theorem~2.12]{ChoksiFonseca}),   but the proof in our case is   technically   more  \EEE demanding. To explain why, we recall that the main ingredients of the proof of \cite[Theorem~2.3]{Silhavy2015} are essentially three: (1) Anzellotti-Giaquinta's Theorem \cite[Theorem~1]{AnzellottiGiaquinta}, (2)   \cite[\MMM Lemma~2.9]{ChoksiFonseca},    which states that   every $u\in C^\infty_c(\Rd)$ \BBB can be approximated \EEE by means of piecewise constant functions whose total variation is controlled by the $L^1$ norm of $\nabla u$, and (3) Alberti's Theorem (see Theorem~\ref{Alberti's Theorem}).  
\subsection{Proof of the Approximation Theorem}
 
To proceed with this strategy, \EEE we need to prove suitable adaptations of the first two results in our functional setting,  calling for nontrivial new arguments. \EEE    We first give the statements of all these results and show how they allow to obtain Theorem~\ref{thm:approximation SDstar}, postponing their proofs to the end of the section.  The first result, whose proof is based on a classical result of Goffman \& Serrin \cite[Theorem~ 4$^\prime$]{GoffmanSerrin}, shows that, if $O\in\Op(\Omega)$ is a Lipschitz set,  every
 $u\in {\rm BV}(O;\Rk)$ can be approximated by a sequence of ${\rm SBV}(O;\Rk)$-functions $\{u_n\}_n$ in such a way that $\mathcal{V}(u_n,O)$ converges to  $\mathcal{V}(u,O)$. \EEE 
\begin{theorem}[Density result in $\BV$]\label{thm: Goffmann-Serrin}
 Let $O',O\in\Op(\Omega)$, with $O'\subset \subset O$, and let  $u\in {\rm BV}(O;\Rk)$.  Assume  that $\Hd(J_u\cap\partial O')=0.$  \EEE Then there exists a sequence $\{u_n\}_{n}\subset {\rm SBV}(O';\Rk)\cap L^\infty(O';\Rk)$ such that $u_n\to u$ strongly in $L^1(O';\Rk)$ as $n\to+\infty$ and 
\begin{equation}\label{limsup leq F}
\limsup_{n\to+\infty}\mathcal{V}(u_n,O')\leq \mathcal{V}(u,O).
\end{equation}
Moreover, if $O$ has Lipschitz boundary, then we can find another sequence, still denoted by $\{u_n\}_{n}$, with $\{u_n\}_{n}\subset  {\rm SBV}(O;\Rk)\cap L^\infty(O;\Rk)$, such that $u_n\to u$ strongly in $L^1(O;\Rk)$ as $n\to+\infty$ and 
\begin{equation}\label{full convergence}
    \lim_{n\to+\infty}\mathcal{V}(u_n,O)=\mathcal{V}(u,O).
\end{equation}
\end{theorem}

\begin{remark}\label{remark: vVec}
  (i)  Replacing  $L^1$-convergence by $L^0$-convergence\EEE, the statement also holds for $\GBV_\star$-functions by performing an additional truncation procedure, as explained at the beginning of the proof of Theorem~\ref{thm:approximation SDstar} below. 
    
    (ii) The   theorem holds in a slightly weaker form also for more general energies  than  $\mathcal{V}$. More precisely, the arguments \BBB in the proof \EEE can be adapted with no difficulty to energies of the form   
    \begin{equation*}
        F(u,O)\coloneq \int_{O} f  (\nabla u)\,{\rm d}x+\int_{O} f^\infty \EEE \Big(\frac{{\rm d}D^cu}{{\rm d}|D^cu|}\Big)\,{\rm d}|D^cu|+\int_{J_u\cap O} g ([u],\nu_u)\,{\rm d}\Hd,
    \end{equation*}
   with $f \EEE \colon\Rkd\to [0,+\infty)$ convex, $ g\colon\Rk\times\Sd\to [0,+\infty)$ with $\zeta\mapsto g\EEE(\zeta,\nu)$ continuous for all $\nu\in\Sd$, and 
   \begin{equation*}
      f\EEE(A)\leq C(|A|+1) \text{ for all $A\in\Rkd$}\quad \text{ and }\quad g(\zeta,\nu)\leq C|\zeta|\land 1\text{ for all $\zeta\in\Rk$ and $\nu\in\Sd$}
   \end{equation*}
   for some constant $C>0$,  where  \MMM the recession function \EEE  $f^\infty$ is defined as in \eqref{def Recession} with $f $  in place of $H_p$. Indeed, this is the setting where \cite[Theorem~4$^\prime$]{GoffmanSerrin} can still be employed.  In this case, \eqref{full convergence} has to be replaced by 
   \begin{equation*}
        \limsup_{n\to+\infty}{F}(u_n,O)\leq {F}(u,O),
   \end{equation*}
   as, in general, $F$ could fail to be $L^1(O;\Rk)$-lower semicontinuous.

(iii) A careful inspection of the proof of Theorem~\ref{thm: Goffmann-Serrin} reveals that the approximating sequence $\{u_n\}_n$ actually satisfies a property stronger than \eqref{full convergence}.  More precisely, we have
    \begin{equation*}
        \lim_{n\to+\infty} \mathcal{V}(u_n,O\setminus J_{u_n})= V(u,O\setminus J_u) \quad \text{ and } \lim_{n\to+\infty} \mathcal{V}(u_n,J_{u_n})= V(u,J_u).
    \end{equation*}
    The same observation applies to the functionals $F$ of (ii), with the  caveat that equalities  need to be replaced by inequalities if $F$ is not lower semicontinuous.
\end{remark}
\EEE

The general  idea is to   approximate the sequence obtained in Theorem~\ref{thm: Goffmann-Serrin}  by a sequence of piecewise constant functions. The existing results in the literature, see e.g.\ \cite[Lemma~2.9]{ChoksiFonseca}, are not expedient for us as they guarantee  convergence of the total variation but not of the functional $\mathcal{V}$, as required in our setting. \EEE

\begin{lemma}\label{lemma piecewise constants}
  Let $O\in\Op(\Omega)$ and $u\in {\rm BV}(O;\Rk)$. Then there exists a sequence of piecewise constant functions  $\{u_n\}_n\subset {\rm SBV}(O;\Rk)\cap L^\infty(O;\Rk)$ such that $u_n\to u$ strongly in $L^1(O;\Rk)$ as $n\to+\infty$ and 
  \begin{equation*}
      \lim_{n\to+\infty} \mathcal{V}(u_n,O)=\mathcal{V}(u,O).
  \end{equation*}
\end{lemma}
\begin{remark}\label{re:lastonestanding}
 Replacing $L^1$-convergence by $L^0$-convergence, the statement also holds for $\GBV_\star$-functions by performing an additional truncation procedure, as explained at the beginning of the proof of Theorem~\ref{thm:approximation SDstar} below.  
\end{remark}

Before proceeding with the proof of Theorem~\ref{thm:approximation SDstar}, we recall for the reader's convenience  the statement of Alberti's Theorem \cite[Theorem~3]{Alberti}.
\begin{theorem}[Alberti's Theorem] \label{Alberti's Theorem} There exists a constant $C=C(d)>0$ with the property that for every $O\in\Op(\Omega)$ and $G\in L^1(O;\Rkd)$ there exists a function $u\in{\rm SBV}(O;\Rk)$  such that $\nabla u=G$ $\Lb^d$-a.e.\ in $O$ and 
\begin{equation*}
    \int_{J_u}|[u]|\,{\rm d}\Hd\leq C\|G\|_{L^1(O)}.
\end{equation*}
\end{theorem}

We are \BBB ready to give \EEE the proof of Theorem ~\ref{thm:approximation SDstar}.
\medskip

\noindent{\it Proof of Theorem \ref{thm:approximation SDstar}.} \EEE
Let us fix $(g,G)\in {\rm SD}_\star^p(O;\Rk)$. By Proposition~\ref{prop:truncation}\EEE, for every   $n\in\N$ the function $g^n\coloneq \Phi_n\circ g$ (see \eqref{def PhiR}) belongs to ${\rm BV}(O;\Rk)$ and \EEE 
\begin{equation*}
\sup_{n\in\N}\mathcal{V}(g^n,O)= \mathcal{V}(g,O), \quad \text{ and }\quad  \{g^n\}_n  \text{ converges to }g \text{ in $L^0(O;\Rk)$ as $n\to+\infty$}.
\end{equation*}
  We apply Theorem~\ref{thm: Goffmann-Serrin} to $g^n$ and obtain a sequence $\{u_{\ell,n}\}_\ell\subset {\rm SBV}(O;\Rk)\cap L^\infty(O;\Rk)$ converging to $g^n$ strongly in $L^1(O;\Rk)$ as $\ell\to+\infty$ and such that
\begin{equation}\label{stima Goffmann}
    \lim_{\ell\to+\infty}\mathcal{V}(u_{\ell,n},O)= \mathcal{V}(g^n,O)\leq \mathcal{V}(g,O).
\end{equation}
We can now use Alberti's Theorem~\ref{Alberti's Theorem} to obtain a function $h\in {\rm SBV}(O;\Rk)$ with $\nabla h=G$ $\Lb^d$-a.e.\ in $O$ and \begin{equation}\label{estimate alberti proof}
    |Dh|(O)\leq C\|G\|_{L^1(O)}
\end{equation} for a suitable constant $C$ depending only on $d$. Applying Lemma~\ref{lemma piecewise constants} to $u_{\ell,n}-h$,   we find a sequence of piecewise constant functions $\{z_{\ell,n}^j\}_j\subset {\rm SBV}(O;\Rk)\cap L^\infty(O;\Rk)$ which converges to $u_{\ell,n}-h$ strongly in $L^1(O;\Rk)$ as $j\to+\infty$ and with 
\begin{equation}\label{stima Alberti}
    \MMM \lim_{j\to+\infty} \EEE \mathcal{V}(z^j_{\ell,n},O)= \mathcal{V}(u_{\ell,n}-h,O)\leq C\big(\mathcal{V}(u_{\ell,n},O)+|Dh|(O)\big)
\end{equation}
for a possibly  larger \EEE  constant $C$.
 Next, \EEE  we set $u_{\ell,n}^j\coloneq h+z_{\ell,n}^j\in {\rm SBV}(O;\Rk) $   and observe that by construction $\nabla u_{\ell,n}^j=G$  $\Lb^d$-a.e.\ in $O$ and that $\{u_{\ell,n}^j\}_j$ converges to $u_{\ell,n}$ strongly in $L^1(O;\Rk)$ as $j\to+\infty$. Moreover, by \eqref{stima Goffmann}--\eqref{stima Alberti}   we have,  for all $n \in \N$, \EEE

\begin{equation*}
 \limsup_{\ell\to+\infty} \MMM \lim_{j\to+\infty} \EEE \mathcal{V}(u^j_{\ell,n},O)\le C\big(\mathcal{V}(g,O)+\|G\|_{L^1(O)}\big).
\end{equation*}
Since the $L^0$-convergence is metrisable, we may use a diagonal argument to obtain sequences $\{\ell_n\}_n $ and $\{j_n\}_n$, with $\ell_n\to+\infty$ and $j_n\to +\infty$ as $n\to+\infty$, so that, setting $u_n\coloneq u^{j_n}_{\ell_n,n}$, we have $u_n\to g$ in $L^0(O;\Rk)$ and 
\begin{equation*}
    \limsup_{n\to+\infty} \mathcal{V}(u_n,O)\leq C\Big(\mathcal{V}(g,O)+\|G\|_{L^1(O)}\Big).
\end{equation*}
 This concludes \EEE the proof.\qed

\medskip

\subsection{Proof of the density results}
In the rest of the section, we prove Theorem~\ref{thm: Goffmann-Serrin} and Lemma~\ref{lemma piecewise constants}.
\EEE
\medskip

\noindent \emph{Proof of Theorem~\ref{thm: Goffmann-Serrin}. }  Thanks to a truncation argument based on Proposition \ref{prop:truncation}, we may assume that $u\in L^\infty(O; \Rk \EEE)$. Moreover, \EEE it is not restrictive to assume $k=1$ because we can argue componentwise.  We begin  by  proving that there exists $\{u_n\}_n\subset {\rm SBV}(O')\cap L^\infty(O')$ converging to $u$ strongly in $L^1(O')$ such that \eqref{limsup leq F} holds. \EEE

Since $J_u\cap O'$ is $(\Hd,d-1)$-rectifiable, there exist a sequence $\{K_\ell\}_\ell$ of pairwise disjoint compact sets of $J_u\cap O'$, 
 a sequence $\{M_\ell\}_\ell$ of $(d-1)$-dimensional orientable $C^1$-manifolds, with $K_\ell\subset \subset M_\ell\subset\subset O'$ for every $\ell\in\N$, and a Borel set $N\subset O'$, with $\Hd(N)=0$, such that $J_u\cap O'=N\cup \big(\bigcup_{\ell\in\N} K_\ell\big)$.
Let $\nu_{M_\ell}$ be a continuous unit normal to $M_\ell$. 
 It is not restrictive to assume that for every $\ell\in\N$ there exist $\sigma_\ell>0$ and $\nu_\ell\in\Sd$, with $\nu_\ell\cdot \nu_{M_\ell}> 0$ on $M_\ell$, such that $y+t\nu_\ell\in O'\setminus M_\ell$ for every $y\in M_\ell$  and   for every $t\in(-\sigma_\ell,0)\cup(0,\sigma_\ell)$.
 
\medskip
\noindent \emph{Step 1.} (Approximation away from the jump set)  Let us fix $\e>0$. There exists  $\ell_\e\in\N$ such that, setting $K^\e\coloneq \bigcup_{\ell=1}^{\ell_\e}K_\ell$ and $N^\e\coloneq  \big(\bigcup_{\ell > \EEE \ell_\e}K_\ell\big)\cup N$, we have 
 \begin{equation}\label{smalnness epsilon}
     J_u\cap O'=K^\e\cup N^\e\quad \text{ and }\quad  \int_{N^\e}|[u]|\,{\rm d}\Hd \le \EEE \e. 
 \end{equation}
 Moreover, we may assume that there exists $\eta^\e>0$ such that  $\text{dist}(M_{j},M_\ell)>\eta^\e$ for every $j,\ell\in\{1,\dots,\ell_\e\}$ with $j\neq \ell$ and that $\sigma_\ell\leq \frac{\eta^\e}{2}$ for every $\ell\in\{1,\dots,\ell_\e\}$. For technical reasons, it will be convenient to introduce  $(d-1)$-dimensional \BBB manifolds $M'_\ell$, $M''_\ell$ \EEE such that $K_\ell \BBB \subset \subset  M_\ell'' \EEE\subset \subset  M_\ell'\subset \subset M_\ell$. For every $\sigma>0$  we define  
 \begin{equation*}
     M_\ell(\sigma)\coloneq \big\{y+t\nu_\ell\colon\, y\in M_\ell\,\,\text{ and }\,\, t\in(-\sigma,\sigma)\big\}
 \end{equation*}
 and give a similar definition for $M'_\ell(\sigma)$ \BBB and $M''_\ell(\sigma)$. \EEE We also
 set \begin{equation}\label{def Mell sigma}
     M^\e(\sigma)\coloneq \bigcup_{\ell=1}^{\ell_\e} M_\ell'(\sigma)
 \end{equation}
 Later we will send $\sigma \to 0$. For the moment, we choose a fixed $\sigma$ with \EEE   \begin{equation}\label{sigmaepsilon}    0<\sigma<\sigma^\e\coloneq \min\big\{\sigma_\ell/4\colon \,\ell\in\{1,\dots,\ell_\e\}\big\},
 \end{equation}
 and observe that with this choice we have $M_\ell(4\sigma)\subset \subset O'$  for every $\ell\in\{1,\dots,\ell_\e\}$ and $M_j(4\sigma)\cap M_\ell(4\sigma)=\emptyset$ for every $j,\ell\in\{1,\dots,\ell_\e\}$ with $j\neq \ell$.  Moreover,  we have that $\tau_\ell\coloneq \text{dist}(K_\ell,O'\setminus M'_\ell(\sigma))>0$.  We refer to Figures \ref{fig:Asroma}  and \ref{fig: subdivision} for an illustration.

 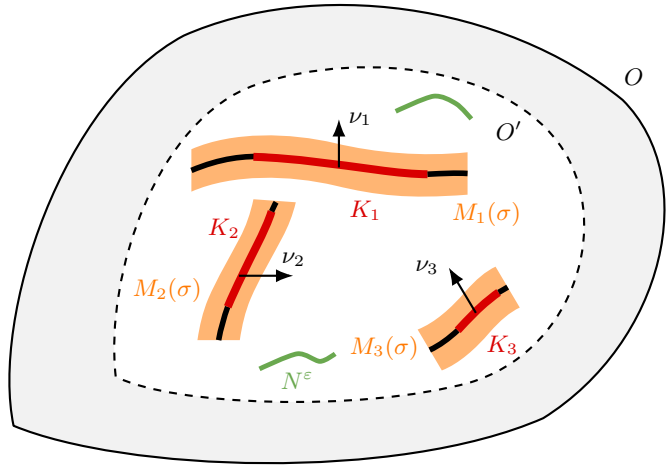
\begin{figure}[b]
    \centering
\begin{tikzpicture}[
    line join=round,
    >=Latex,
    strip/.style={fill=AsRoma, opacity=.55},
    Mpart/.style={black, line width=2.1pt, line cap=butt},
    Kpart/.style={red!85!black, line width=3.0pt, line cap=butt},
    normal/.style={->, thick},
    lab/.style={font=\small}
]

\fill[gray!10]
  (-4.25,-2.45)
  .. controls (-4.55,-1.00) and (-3.60,1.85) .. (-2.00,2.70)
  .. controls (-0.30,3.45) and (2.20,3.15) .. (3.80,1.85)
  .. controls (4.85,0.80) and (4.35,-1.35) .. (2.75,-2.35)
  .. controls (0.80,-3.20) and (-2.80,-3.05) .. (-4.25,-2.45)
  -- cycle;

\draw[thick]
  (-4.25,-2.45)
  .. controls (-4.55,-1.00) and (-3.60,1.85) .. (-2.00,2.70)
  .. controls (-0.30,3.45) and (2.20,3.15) .. (3.80,1.85)
  .. controls (4.85,0.80) and (4.35,-1.35) .. (2.75,-2.35)
  .. controls (0.80,-3.20) and (-2.80,-3.05) .. (-4.25,-2.45)
  -- cycle;

\node[lab] at (3.95,2.15) {$O$};

\fill[white]
  (-2.90,-1.80)
  .. controls (-3.05,-0.45) and (-2.35,1.60) .. (-1.05,2.05)
  .. controls (0.45,2.55) and (2.45,2.25) .. (3.08,1.08)
  .. controls (3.62,0.05) and (3.10,-1.48) .. (1.78,-1.88)
  .. controls (0.30,-2.28) and (-2.18,-2.18) .. (-2.90,-1.80)
  -- cycle;

\draw[thick,dashed]
  (-2.90,-1.80)
  .. controls (-3.05,-0.45) and (-2.35,1.60) .. (-1.05,2.05)
  .. controls (0.45,2.55) and (2.45,2.25) .. (3.08,1.08)
  .. controls (3.62,0.05) and (3.10,-1.48) .. (1.78,-1.88)
  .. controls (0.30,-2.28) and (-2.18,-2.18) .. (-2.90,-1.80)
  -- cycle;

  \begin{scope}[yshift=3.0 cm, xshift=1.8 cm]
\draw [line width=1.8pt ,smooth,color=mossgreen] plot coordinates {(-1,-1.3) (-0.5,-1.1) (-0.2,-1.2) (0,-1.4)};
\end{scope}

\node[lab] at (2.30,1.45) {$O'$};

\fill[strip]
  (-1.90,1.20)
    .. controls (-1.35,1.44) and (-0.60,1.42) .. (0.00,1.28)
    .. controls (0.55,1.15) and (1.15,1.10) .. (1.75,1.16)
  --
  (1.75,0.60)
    .. controls (0.95,0.54) and (0.55,0.59) .. (0.00,0.72)
    .. controls (-0.60,0.86) and (-1.35,0.88) .. (-1.90,0.64)
  -- cycle;

\draw[Mpart]
  (-1.90,0.92)
    .. controls (-1.52,1.08) and (-1.20,1.12) .. (-0.88,1.10);

\draw[Kpart]
  (-1.08,1.10)
    .. controls (-0.50,1.08) and (-0.00,1.00) .. (0.54,0.93)
    .. controls (0.82,0.89) and (1.06,0.87) .. (1.22,0.87);

\draw[Mpart]
  (1.22,0.87)
    .. controls (1.42,0.88) and (1.59,0.89) .. (1.75,0.88);

\draw[normal,->] (0.05,0.96) -- (0.05,1.60)
  node[right,lab] {$\nu_1$};

\node[lab,text=red!85!black] at (0.4,0.39) {$K_1$};

\node[lab,text=AsRoma] at (2,0.34) {$M_1(\sigma)$};

\fill[strip]
  (-1.82,-1.32)
    .. controls (-1.79,-1.12) and (-1.75,-0.95) .. (-1.68,-0.78)
    .. controls (-1.56,-0.50) and (-1.38,-0.20) .. (-1.26,0.03)
    .. controls (-1.18,0.18) and (-1.12,0.35) .. (-1.08,0.54)
  --
  (-0.52,0.50)
    .. controls (-0.56,0.33) and (-0.62,0.18) .. (-0.70,0.03)
    .. controls (-0.82,-0.20) and (-1.00,-0.50) .. (-1.12,-0.78)
    .. controls (-1.19,-0.95) and (-1.23,-1.12) .. (-1.26,-1.32)
  -- cycle;

\draw[Mpart]
  (-1.54,-1.32)
    .. controls (-1.51,-1.16) and (-1.47,-1.02) .. (-1.42,-0.88);

\draw[Kpart]
  (-1.42,-0.88)
    .. controls (-1.31,-0.62) and (-1.16,-0.32) .. (-1.04,-0.08)
    .. controls (-0.95,0.10) and (-0.88,0.25) .. (-0.84,0.38);

\draw[Mpart]
  (-0.84,0.38)
    .. controls (-0.81,0.45) and (-0.78,0.50) .. (-0.78,0.50);

\draw[normal] (-1.27,-0.47) -- (-0.55,-0.47)
  node[above,lab] {$\nu_2$};

\begin{scope}[yshift=-0.2 cm]
\draw [line width=1.8pt,smooth,color=mossgreen] plot coordinates {(-1,-1.5) (-0.5,-1.3) (-0.2,-1.4) (0,-1.3)};

\node[lab,text=mossgreen] at (-0.5,-1.65) {$N^\e$};
\end{scope}

\node[lab,text=red!85!black] at (-1.47,0.16) {$K_2$};

\node[lab,text=AsRoma] at (-2.22,-0.66) {$M_2(\sigma)$};

\fill[strip]
  (1.09,-1.18)
    .. controls (1.32,-1.09) and (1.54,-0.89) .. (1.72,-0.68)
    .. controls (1.86,-0.52) and (1.99,-0.41) .. (2.12,-0.35)
  --
  (2.44,-0.89)
    .. controls (2.31,-0.95) and (2.18,-1.06) .. (2.04,-1.22)
    .. controls (1.86,-1.43) and (1.64,-1.63) .. (1.41,-1.72)
  -- cycle;

\draw[Mpart]
  (1.25,-1.45)
    .. controls (1.38,-1.39) and (1.50,-1.30) .. (1.62,-1.19);

\draw[Kpart]
  (1.62,-1.19)
    .. controls (1.76,-1.04) and (1.90,-0.89) .. (2.03,-0.79)
    .. controls (2.09,-0.74) and (2.13,-0.71) .. (2.16,-0.69);

\draw[Mpart]
  (2.16,-0.69)
    .. controls (2.21,-0.66) and (2.25,-0.64) .. (2.28,-0.62);

\draw[normal,->] (1.86,-0.96) -- (1.50,-0.36)
  node[left,lab] {$\nu_3$};

\node[lab,text=red!85!black] at (2.21,-1.38) {$K_3$};

\node[lab,text=AsRoma] at (0.65,-1.42) {$M_3(\sigma)$};

\end{tikzpicture}
\caption{The sets $M^\e(\sigma)$ and $N^\e$ for $\ell_\e=3$.} 
  \label{fig:Asroma}
\end{figure}

We consider a sequence of convolution kernels $\{\rho_n\}_{n}\subset  C^\infty_c(\Rd)$ with $\text{supp}(\rho_n)\subset B(0,\frac{1}{n})$ and define $v_n\coloneq u*\rho_n$. Observe that, if  $O''\in \Op(O)$ is such that $O'\subset \subset O''\subset \subset O$, \EEE  $x\in O'\setminus M^\e(\sigma)$, and if $\frac{1}{n}\leq \min\big\{\text{dist}(O',\Rd\setminus O''\EEE),\min\{\tau_\ell\colon \,\ell\in\{1,\dots,\ell_\e\}\}\big\}$, then  $B(x,\frac{1}{n})\subset O'' \EEE\setminus K^\e$.  Since $v_n\coloneq u*\rho_n$,   using a result of Goffman \& Serrin \cite[Theorem~4$^\prime$]{GoffmanSerrin} and \eqref{smalnness epsilon} we may estimate   
\begin{align*}
\limsup_{n\to+\infty}\mathcal{V}(v_n,O'\setminus M^\e(\sigma))&= \limsup_{n\to+\infty}\int_{O'\setminus M^\e(\sigma)}|\nabla  v_n|\,{\rm d}x \leq \int_{O'' \EEE}|\nabla u|\,{\rm d}x+|D^su|( O'' \EEE \setminus K^\e)\notag\\&  = \mathcal{V}(u,  O'' \EEE \setminus J_u)+\int_{N^\e\cup (J_u\cap(O''\setminus O'))\EEE} |[u]|\,{\rm d}\Hd  \notag\\
&\leq  \mathcal{V}(u, O \setminus J_u)+\int_{ J_u\cap(O''\setminus O')\EEE} |[u]|\,{\rm d}\Hd +\e.\EEE
\end{align*}
 As $O''$ was arbitrary and we assumed that $\Hd(J_u\cap \partial O')=0$, we may let $O''\searrow O'$ and deduce that  
\begin{equation}
  \limsup_{n\to+\infty}\mathcal{V}(v_n,O'\setminus M^\e(\sigma))\leq   \mathcal{V}(u, O \setminus J_u) +\e\label{stima convolution}.
\end{equation}
\EEE

\medskip
\noindent \emph{Step 2.} (Construction at the jump set)
We introduce the function $\phi^\sigma\colon [-3\sigma,3\sigma]\to [-3\sigma,-\sigma]\cup[\sigma,3\sigma]$ defined by    
 \begin{equation*}
\phi^\sigma(t)\coloneq \begin{cases}\frac{t}{2}-\frac{3}{2}\sigma &\text{ if }-3\sigma\leq t\leq -\sigma,\\
    t-\sigma &\text{ if }-\sigma \leq t\leq 0,  \\
        t+\sigma &\text{ if }0 <t\leq  \sigma,\\
        \frac{t}{2}+\frac{3}{2}\sigma &\text{ if } \sigma\leq t\leq 3\sigma,
    \end{cases}
\end{equation*}
and consider the function defined on $M'_\ell(3\sigma)$  by
\begin{equation*}
v_n^{\sigma,\ell}(y+t\nu_\ell)\coloneq v_n(y+\phi^\sigma(t)\nu_\ell)\quad \text{ for every } y\in M'_\ell \text{ and }t\in [-3\sigma, 3\sigma].
\end{equation*}
We observe that $v_n^{\sigma,\ell}=v_n$ on the top and bottom parts of the boundary  (see Figure \ref{fig: subdivision}) \EEE of $M'_\ell(3\sigma)$, given by $M'_\ell\pm3\sigma \nu_\ell$. To deal with the lateral part of the boundary, namely $\partial M'_\ell+[-3\sigma,3\sigma]\nu_\ell$, we consider a further cut-off function $\gamma_\ell\in C^1_c(M'_\ell;[0,1])$   with $\gamma_\ell= 1$ in  \BBB the \EEE neighbourhood $M''_\ell$ \BBB (recall that \EEE $K_\ell\subset \subset M''_\ell\subset \subset M'_\ell$). We set 
\begin{equation}\label{def unsigmaell}
u_n^{\sigma,\ell}(y+t\nu_\ell)\coloneq \gamma_\ell(y)v_n^{\sigma,\ell}(y+t\nu_\ell)+(1-\gamma_\ell(y))v_n(y+t\nu_\ell)\quad \text{ for  all \EEE } y\in M'_\ell \text{ and }t\in [-3\sigma, 3\sigma].
\end{equation}
Since $u_n^{\sigma,\ell}=v_n$ on the full boundary of $M'_\ell(3\sigma)$, we can define a function $u^{\sigma,\e}_n$ on $O'$ by setting  
\begin{equation}\label{def unsigmaeps}
  u^{\sigma,\e}_n\coloneq  \begin{cases}  u^{\sigma,\ell}_n \quad &\text{ on }M'_\ell(3\sigma)\,\, \text{ for $\ell\in\{1,\dots,\ell_\e\}$},\\
        v_n &\text{ on }O'\setminus M^\e(3\sigma).
    \end{cases}
\end{equation}
 It follows directly from the definition that $u_n^{\sigma,\e}\in {\rm SBV}(O')\cap L^\infty(O')$, \BBB that \EEE $u_n^{\sigma,\e}$ is a  smooth   Lipschitz function on $O'\setminus \big(\bigcup_{\ell=1}^{\ell_\e} M'_\ell\big)$, and \BBB that \EEE for every $n\in\N$ we have
\begin{gather}\label{jump set modifiche}
    J_{u_n^{\sigma,\e}}\subset \bigcup_{\ell=1}^{\ell_\e} M'_\ell\subset\subset O',\\
     \notag\nu_{u_n^{\sigma,\e}}= \MMM \nu_{M_\ell} \EEE \quad \Hd\text{-a.e. on }M'_\ell\cap J_{u_n^{\sigma,\e}},\\  \label{jump modifiche}[u_n^{\sigma,\e}]=\big(v_n(\cdot+\sigma\nu_\ell)-v_n(\cdot-\sigma\nu_\ell)\big)\gamma_\ell\quad \text{on }  M'_\ell\cap J_{u_n^{\sigma,\e}},
\end{gather}
where  we have used \BBB  $M_\ell \subset O'$, \EEE the inequality $\nu_\ell\cdot\nu_{M_\ell}>0$, \MMM and the definitions of  $\phi^\sigma$ and  $v_n^{\sigma,\ell}$, see   Figure  \ref{fig: subdivision}\EEE.  

\medskip
\noindent \emph{Step 3.} ($L^1$-convergence) We \MMM  estimate  \EEE the $L^1$-distance of $u_n^{\sigma,\e}$ from $u$ on $O'$. We have
\begin{align}\notag
\int_{O'}|u_n^{\sigma,\e}-u|\,{\rm d}x& \le 
\int_{O'}|u_n^{\sigma,\e}-v_n|\,{\rm d}x\ + \int_{O'}|v_n-u|\,{\rm d}x\
\\\label{L1 distance}
&= 
\sum_{\ell=1}^{\ell_\e}\int_{M'_{\ell}(3\sigma)}|u_n^{\sigma,\e}-v_n|\,{\rm d}x
+  \int_{O'} \EEE |v_n-u| \,{\rm d}x.
\end{align}
 Recalling the definition of $u_n^{\sigma,\ell}$, \EEE we observe that  
\begin{equation*}
\int_{M'_{\ell}(3\sigma)}|u_n^{\sigma,\e}-v_n|\,{\rm d}x
\leq \int_{M'_{\ell}(3\sigma)}\big(|v_n^{\sigma,\ell}|+|v_n|\big) \, {\rm d}x.
\end{equation*}
After a change of variables we obtain a constant $C_\ell \ge 1$ such that 
\begin{equation}\label{new delta0}
\int_{M'_{\ell}(3\sigma)}|v_n^{\sigma,\ell}|\, {\rm d}x\leq
C_\ell\int_{M'_{\ell}(3\sigma)}|v_n| \, {\rm d}x,
\end{equation}
hence,
\begin{equation}\label{def deltaell}
\int_{M'_{\ell}(3\sigma)}|u_n^{\sigma,\e}-v_n|\,{\rm d}x
\leq 2 \EEE C_\ell\int_{M'_{\ell}(3\sigma)}|v_n|\,{\rm d}x
\leq  2\EEE C_\ell\int_{M'_{\ell}(3\sigma)}|v_n-u|\,{\rm d}x
+ 2\EEE C_\ell\int_{M'_{\ell}(3\sigma)}|u|\,{\rm d}x.
\end{equation}
Since the sets $M'_{\ell}(3\sigma)$ are pairwise disjoint and contained in  $O'$ and $\{v_n\}_n$ converges to $u$ strongly in $L^1(O')$, \BBB it holds 
\begin{equation}\label{def omegaeps}
\lim_{\sigma\to 0^+} \sum_{\ell=1}^{\ell_\e}\omega_\ell(\sigma)=0, \quad \MMM \text{ where }\ \  \EEE\omega_\ell(\sigma)\coloneq \lim_{n\to+\infty}\int_{M'_\ell(3\sigma)}|v_n-u|\,{\rm d}x+\MMM \int_{M'_\ell(3\sigma)} \EEE |u|\,{\rm d}x.
\end{equation}
Finally, combining \eqref{L1 distance}--\eqref{def omegaeps} and using again that $\{v_n\}_n$ converges to $u$ strongly in $L^1(O')$, 
 we obtain, for every $0<\sigma<\sigma^\e$,  
\begin{equation}\label{usinman-u}
\limsup_{n\to+\infty}\int_{O'}|u_n^{\sigma,\e}-u|\,{\rm d}x\le 
2\EEE \sum_{\ell=1}^{\ell_\e}C_\ell\omega_\ell(\sigma).\EEE
\end{equation}

\medskip
\noindent \emph{Step 4.} (Estimate on $\mathcal{V}$)  We estimate $\mathcal{V}(u_n^{\sigma,\e},O')$, given by
\begin{equation}\label{da stimare}
    \mathcal{V}(u_n^{\sigma,\e},O')=\int_{O'}|\nabla u_n^{\sigma,\e}|\,{\rm d}x+\int_{O'\cap J_{u_n^{\sigma,\e}}}|[u_n^{\sigma,\e}]|\land 1\,{\rm d}\Hd.
\end{equation}
As for the first term, we have   by \eqref{stima convolution} \MMM and \eqref{def unsigmaeps}  \EEE
\begin{align}\label{stimella}
 \limsup_{n\to+\infty}   \int_{O'}|\nabla u_n^{\sigma,\e}|\,{\rm d}x&=\limsup_{n\to+\infty} \Big(\int_{O'\setminus M'(3\sigma)}|\nabla v_n|\,{\rm d}x+\sum_{\ell=1}^{\ell_\e}\int_{M'_{\ell}(3\sigma)}|\nabla u_n^{\sigma,\e}|\,{\rm d}x\Big)\notag \\ 
    &  \leq \mathcal{V}(u,O\setminus  J_u\EEE)  + \e \EEE +\limsup_{n\to+\infty}\Big(\sum_{\ell=1}^{\ell_\e}\int_{M'_{\ell}(3\sigma)}|\nabla u_n^{\sigma,\e}|\,{\rm d}x\Big).
\end{align}

To estimate the last term in the previous inequality, we note that by \eqref{def unsigmaell} and \eqref{def unsigmaeps} we have
\begin{equation*}
    \nabla u_n^{\sigma,\e}=\gamma_\ell \nabla v^{\sigma,\ell}_n+(1-\gamma_\ell)\nabla v_n+\nabla \gamma_\ell\big(v_{n}^{\sigma,\ell}-v_n\big)\quad \text{$\Lb^d$-a.e.\ in $M'_\ell(3\sigma)$,}
\end{equation*}

\begin{figure}[b]
\centering\hspace{-1 cm}
\begin{tikzpicture}[
  >=Latex,
  line cap=round,
  line join=round
]

\def\xL{0.0}
\def\xR{8.0}
\def\xpL{1.2}
\def\xpR{6.8}
\def\xppL{3.15}
\def\xppR{5.65}
\def\xkL{3.45}
\def\xkR{5.35}

\def\sig{0.84}

\pgfmathdeclarefunction{f}{1}{%
  \pgfmathparse{0.55 + 0.09*#1 + 0.22*sin(55*#1)}%
}

\newcommand{\filltube}[3]{%
  \fill[#3]
    plot[domain=#1:#2, samples=80] (\x,{f(\x)+\sig})
    --
    plot[domain=#2:#1, samples=80] (\x,{f(\x)-\sig})
    -- cycle;
}

\filltube{\xL}{\xpL}{green!20}
\filltube{\xpR}{\xR}{green!20}

\def\xa{1.2}
\def\xb{1.9}
\def\xc{2.55}
\def\xd{3.15}

\filltube{\xa}{\xb}{cyan!30}
\filltube{\xb}{\xc}{cyan!18}
\filltube{\xc}{\xd}{cyan!30}

\def\xe{5.65}
\def\xf{6.05}
\def\xg{6.45}
\def\xh{6.8}

\filltube{\xe}{\xf}{cyan!30}
\filltube{\xf}{\xg}{cyan!18}
\filltube{\xg}{\xh}{cyan!30}

\filltube{\xppL}{\xppR}{orange!30}

\foreach \xx in {\xb,\xc,\xf,\xg}{
  \draw[dashed, blue!70!black]
    (\xx,{f(\xx)-\sig}) -- (\xx,{f(\xx)+\sig});
}

\foreach \xx in {\xa,\xd,\xe,\xh}{
  \draw[blue!60!black]
    (\xx,{f(\xx)-\sig}) -- (\xx,{f(\xx)+\sig});
}

\draw[black, thick, domain=\xL:\xR, samples=160]
  plot (\x,{f(\x)});

\draw[blue!70!black, very thick, domain=\xpL:\xpR, samples=140]
  plot (\x,{f(\x)});

\draw[orange!85!black, very thick, domain=\xppL:\xppR, samples=140]
  plot (\x,{f(\x)});

\draw[red!80!black, line width=2.7 pt, domain=\xkL:\xkR, samples=120]
  plot (\x,{f(\x)});

\draw[->, thick]
  (4.4,{f(4.4)}) -- ++(0,1.45)
  node[above] {$\nu_\ell$};

\node[below,red!80!black] at (4.45,{f(-3.45)+0.24}) {$K_\ell$};

\node[green!50!black] at (-0.2,1.85) {$M_\ell(\sigma)\setminus M'_\ell(\sigma)$};
\node[cyan!70] at (2.25,-0.60) {$M'_\ell(\sigma)\setminus M''_\ell(\sigma)$};
\node[orange!80!black] at (5.4,2.45) {$M''_\ell(\sigma)$};

\node[cyan!70] at (6.2,-0.55) {\scriptsize -\,-supp({$\gamma_\ell$)}-\,-};

\end{tikzpicture}

\caption{Subdivision of $M_\ell(\sigma)$}
    \label{fig: subdivision}
\end{figure}
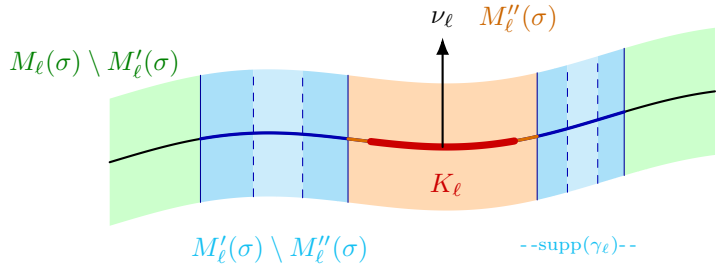
\noindent so that
\begin{align}\notag
    \int_{M'_{\ell}(3\sigma)}|\nabla u_n^{\sigma,\e}|\,{\rm d}x&\leq    \int_{M'_{\ell}(3\sigma)}| \nabla v^{\sigma,\ell}_n|\,{\rm d}x+   \int_{M'_{\ell}(3\sigma)\setminus M''_{\ell}(3\sigma)}|\nabla v_n|\,{\rm d}x\\&\quad \label{stima m3sigma}+\|\nabla \gamma_\ell\|_{L^\infty(M'_\ell(3\sigma))}  \int_{M'_{\ell}(3\sigma)}|v_{n}^{\sigma,\ell}-v_n|\,{\rm d}x
\end{align}
for all $\ell\in\{1,\dots,\ell_\e\}$. \BBB We \EEE estimate each term separately. 
For the first term, a change of variables shows that \EEE for every $\ell\in\{1,\dots,\ell_\e\}$ there exists a constant $C_\ell>0$ such that   
\begin{equation}\label{estimate with Ck} 
 \int_{M'_{\ell}(3\sigma)}|\nabla  v_n^{\sigma,\ell}\EEE|\,{\rm d}x\leq C_\ell\int_{M'_\ell(3\sigma)\setminus M'_\ell(\sigma)}|\nabla v_n|\,{\rm d}x.
\end{equation}
We   observe that $M'_\ell(3\sigma)\setminus M'_\ell(\sigma)\subset \subset M_\ell(4\sigma)\BBB \setminus M_\ell \EEE $.  Thus, we can apply once again \cite[Theorem~4$^\prime$]{GoffmanSerrin}, and  we obtain \begin{gather}
\label{stima con lambda}\limsup_{n\to+\infty}\int_{M'_{\ell}(3\sigma)\setminus M'_\ell(\sigma)} |\nabla v_n|\,{\rm d}x\leq \int_{\BBB M_\ell(4\sigma) \setminus M_\ell \EEE}|\nabla u|\,{\rm d}x+|D^su|(\BBB M_\ell(4\sigma) \setminus M_\ell \EEE)=:\lambda_\ell(\sigma)
\end{gather}
for every $\ell\in\{1,\dots,\ell_\e\}$.  
Note that $\lambda_\ell(\sigma)\to 0^+$ as $\MMM \sigma \EEE\to 0^+$ because  $\BBB M_\ell(4\sigma) \setminus M_\ell \EEE \searrow \emptyset$ as $\sigma\to 0^+$.  Similarly, we can use \cite[Theorem~4$^\prime$]{GoffmanSerrin} to estimate the second term in \eqref{stima m3sigma} as 
\begin{equation}\label{forgotaboutit}
\limsup_{n\to+\infty}\int_{M'_{\ell}(3\sigma)\setminus M''_{\ell}(3\sigma)}|\nabla v_n|\,{\rm d}x\leq \int_{ M_{\ell}(4\sigma)\setminus K_{\ell}}\!\!\!\!\!\!\!\!\!\!|\nabla u|\,{\rm d}x+|D^su|\big(M_{\ell}(4\sigma)\setminus \MMM K_{\ell}\EEE \big)=:\mu_\ell(\sigma)
\end{equation}  
for every $\ell\in \{1,\dots, \ell_\e\}$. We observe that, since \MMM $M_{\ell}(4\sigma)\setminus K_{\ell}\searrow M_\ell\setminus K_\ell$,  \EEE  it holds 
\begin{equation}\label{small mueps}
  \limsup_{\sigma\to 0^+}\mu_\ell(\sigma)\leq \int_{M_\ell\setminus K_\ell}|[u]|\,{\rm d}\Hd =0,
\end{equation}
where we used that  $\mathcal{H}^{d-1}(J_u \cap (M_\ell \setminus K_\ell)) = 0$. 
Finally, for the last term in \eqref{stima m3sigma}, we use  \eqref{new delta0}, the second inequality in \eqref{def deltaell}, and \eqref{def omegaeps} \EEE  to get
\begin{equation}\label{mbublè}
   \limsup_{n\to+\infty} \|\nabla \gamma_\ell\|_{L^\infty(M'_\ell(3\sigma))}  \int_{M'_{\ell}(3\sigma)}|v_{n}^{\sigma,\ell}-v_n|\,{\rm d}x\leq C_{\ell} \omega_\ell(\sigma)
\end{equation}
for a constant $C_{\ell} \MMM > 0 \EEE$ depending only on  $\ell$. 
Combining the previous inequality \MMM with \eqref{stima m3sigma}--\eqref{forgotaboutit}, \EEE we obtain, for all $\ell\in\{1,\dots,\ell_\e\}$,
\begin{equation}\label{max pezzali}
    \limsup_{n\to+\infty} \int_{M'_{\ell}(3\sigma)}|\nabla u_n^{\sigma,\e}|\,{\rm d}x\leq C_\ell\lambda_\ell(\sigma)+\mu_\ell(\sigma)+C_{\ell}\omega_\ell(\sigma),
\end{equation}
 where $C_\ell>0$ is (the maximum of) the  constant(s) appearing in \eqref{estimate with Ck} and \eqref{mbublè},  respectively. \EEE
 
We now estimate the second term  on \EEE the right-hand side of \eqref{da stimare}. To this end, we observe that by \eqref{jump set modifiche} and \eqref{jump modifiche} we have      
\begin{equation}\label{uguale g}
\int_{J_{u_n^{\sigma,\e}}\cap M_\ell'}|[u_n^{\sigma,\e}]|\land 1\,{\rm d}\Hd= \int_{ M'_\ell}|(v_n(\cdot+\sigma\nu_\ell)-v_n(\cdot-\sigma\nu_\ell))\gamma_\ell \MMM | \EEE \land 1\,{\rm d}\Hd.
\end{equation}
Let $ \partial^{\pm \sigma}_\ell \EEE\coloneq M_\ell\pm \sigma\nu_\ell$.   As $|[u]|\Hd\mres J_u$ is a finite measure and the sets $\partial^{\pm \sigma}_\ell \EEE$ are pairwise disjoint, there exists a set $S\subset (0,+\infty)$, with $(0,+\infty)\setminus S$ at most countable, such that $\Hd( \partial^{\pm \sigma}_\ell \EEE\cap J_u)=0$ for every $\sigma\in S$. By well-known properties of convolutions of ${\rm BV}$-functions, this implies that for every $\sigma \in S$  the trace of $v_n$ on $ \partial^{\pm \sigma}_\ell \EEE$ converges in $L^1_{\Hd}( \partial^{\pm \sigma}_\ell \EEE)$ to the trace of $u$ on $ \partial^{\pm \sigma}_\ell \EEE$, which is well-defined because the traces on both sides coincide. By the Dominated Convergence Theorem, this implies 
\begin{gather*}
\limsup_{n\to+\infty}\int_{M_\ell}|(v_n(\cdot+\sigma\nu_\ell)-v_n(\cdot-\sigma\nu_\ell))\gamma_\ell|\land 1\,{\rm d}\Hd 
    =
\int_{M_\ell}|(u(\cdot+\sigma\nu_\ell)-u(\cdot-\sigma\nu_\ell))\gamma_\ell|\land 1 \, \MMM {\rm d} \EEE \Hd.
\end{gather*}
Since  $(u(\cdot+\sigma\nu_\ell)-u(\cdot-\sigma\nu_\ell))\gamma_\ell  \to [u]\gamma_\ell$ in $L^1_{\Hd}(M_\ell)$ as $\sigma \to 0^+$ (with $\sigma \in S$), \EEE  
using again the Dominated Convergence Theorem we deduce that  
\begin{gather*}\notag 
   \lim_{\substack{\sigma\to 0^+\\\sigma \in S}} \int_{M_\ell}|(u(\cdot+\sigma\nu_\ell)-u(\cdot-\sigma\nu_\ell))\gamma_\ell|\land 1\,{\rm d}\Hd =\int_{  M_\ell}|[u]\gamma_\ell|\land 1\,{\rm d}\Hd.
\end{gather*}
Therefore, recalling \eqref{uguale g}, for every $\ell\in\{1,\dots,\ell_\e\}$ there exists a function $ \eta_\ell\EEE\colon S\to [0,+\infty)$, with $\eta_\ell\EEE(\sigma)\to 0$ as $\sigma\to 0^+$, such that for every $\sigma\in S$ we have   
\begin{align}\notag
\limsup_{n\to+\infty}\int_{ M'_\ell}|[u_n^{\sigma,\e}]|\land 1\,{\rm d}\Hd &\leq \int_{ K_\ell}\!\!\!\!|[u]|\land 1\,{\rm d}\Hd+\int_{M_\ell\setminus K_\ell}\!\!\!|[u]|\land 1\,{\rm d}\Hd+ \eta_\ell(\sigma)\EEE,\\ \label{uguale g2}
& =\EEE \int_{  K_\ell}\!\!\!\!|[u]|\land 1\,{\rm d}\Hd + \eta_\ell(\sigma),\EEE
\end{align}
 where in  the equality \EEE we have used $\mathcal{H}^{d-1}(J_u \cap (M_\ell \setminus K_\ell)) = 0$. 
 
Combining the last inequality with   \eqref{da stimare}, \eqref{stimella}, and \eqref{max pezzali}  we obtain   \EEE 
\begin{align}\notag 
    \limsup_{n\to +\infty}\mathcal{V}(u_n^{\sigma,\e},O'\EEE)&\leq \mathcal{V}(u, O\setminus J_u \EEE)+\!\sum_{\ell=1}^{\ell_\e}\int_{K_\ell}\!\!|[u]|\land 1\,{\rm d}\Hd\!+  \e   +\!\sum_{\ell=1}^{\ell_\e}\big( C_\ell  \lambda_\ell(\sigma)+\mu_\ell(\sigma)+C_{\ell}\omega_\ell(\sigma)+\eta_\ell(\sigma)\big) \EEE\\&\leq  \mathcal{V}(u,O)+\sum_{\ell=1}^{\ell_\e}\big(C_\ell  \lambda_\ell(\sigma)+\mu_\ell(\sigma)+C_{\ell}\omega_\ell(\sigma)+\eta_\ell(\sigma)\EEE\big)+   \e. \EEE
    \label{Fusigman}
\end{align}
\EEE

\medskip
\noindent \emph{Step 5.} (Conclusion of the \MMM proof \EEE in $O'$)
We now fix two sequences of positive numbers $\{\e_j\}_{j}$ and $\{\sigma_j\}_{j}$ converging to $0$ such that for every $j\in\N$ we have $0<\sigma_j<\sigma^{\e_j}$
(see \eqref{sigmaepsilon}) and
$$
{\sum_{\ell=1}^{\ell_{\e_j}}\big( C_\ell  \lambda_\ell(\sigma_j)+\mu_\ell(\sigma_j)\EEE+ 2 C_{\ell} \omega_\ell(\sigma_j)+\eta_\ell(\sigma_j)\EEE\big)\MMM < \EEE \e_j}
$$
(see  \MMM  \eqref{def omegaeps}--\eqref{usinman-u}, \EEE \eqref{stima con lambda}, \eqref{small mueps}, \EEE and \eqref{uguale g2}).  
From this inequality and from \eqref{usinman-u} and \eqref{Fusigman},   we obtain for every $j\in\N$ that
$$
\limsup_{n\to+\infty }\int_{O'}|u^{\sigma_j,\e_j}_{n}-u|\,{\rm d}x\MMM < \EEE \e_j\quad \text{ and }\quad \limsup_{n\to +\infty}\mathcal{V}(u^{\sigma_j,\e_j}_n,O')\MMM < \EEE
\mathcal{V}(u,O)+  2\e_j. \EEE
$$
Therefore, there exists $n_j\in \N$ such that 
\begin{equation}\label{stima finale}
\int_{O'}|u^{\sigma_j,\e_j}_{n}-u|\,{\rm d}x\leq \e_j\quad \text{ and }\quad \mathcal{V}(u^{\sigma_j,\e_j}_n,O')\leq 
\mathcal{V}(u,O)+  2\e_j  \EEE
\end{equation}
for every $n\ge n_j$. It is not restrictive to assume that $n_j<n_{j+1}$ for every $j\in \N$. We set $u_{n}\coloneq u^{\sigma_j,\e_j}_{n_j}$ for $n_j\leq n<n_{j+1}$ and conclude observing that by construction $u_n\in {\rm SBV}(O')\cap L^\infty(O')$ and that from \eqref{stima finale} it follows that $\{u_n\}_{n}$ converges strongly in $L^1(O')$ as $n\to+\infty$ and that \eqref{limsup leq F} holds true.

\medskip
\noindent \emph{Step 6.} ($O$ with Lipschitz boundary.) Assume now that $O$ has Lipschitz boundary and let us prove that there exists $\{u_n\}_n\subset {\rm SBV}(O)\cap L^\infty(O)$ converging to $u$ in $L^1(O)$ such that \eqref{full convergence} holds. To this end, we let  $\{\e_j\}_j$ be a sequence of positive numbers converging to $0$ as $j\to+\infty$ and let $\{O_j\}_j$  be an increasing sequence of open sets with Lipschitz boundary, compactly contained in $O$, such that $O_j\nearrow O$ as $j\to +\infty$ and 
\begin{equation}\label{derivationve on layer}
   |D u|\big(O\setminus \overline{O_j}\big)\leq \e_j\quad \text{for all $j\in\N$}.
\end{equation}  As $J_u$ is $\sigma$-finite with respect to $\Hd$, we may choose the sequence $\{O_j\}_j$ in such a way that $\Hd(\partial O_j\cap J_u)=0$ for all $j\in\N $. 
We also recall that (see the comment below \eqref{sigmaepsilon})  
\begin{equation}\label{positive distance}
{\rm dist} (M^{\e_j}(\sigma^{\e_j}),\partial O_j)>0
\end{equation} for every $j\in\N$, where   $M^{\e_j}(\sigma)$ is given by \eqref{def Mell sigma}.

For every $0<\sigma<\sigma^{\e_j}$, we consider the functions $u^{\sigma,j}_n\coloneq u^{\sigma,\e_j}_n$ constructed in Step 2 and given by \eqref{def unsigmaeps}. 
Note that \eqref{positive distance} implies
that in a neighbourhood of $\partial O_j$,  independent of $n$, the function $u_n^{\sigma,j}$ is given by $v_n\coloneq u*\rho_n$, where $\rho_n$ is a smooth convolution kernel. Therefore, letting $(u^{\sigma,j}_n)_{\partial O_j^-}$ (resp.\ \MMM $u_{\partial O^-_j}$) \EEE be the trace from the interior of $u^{\sigma,j}_n$ (resp.\ $u$) on $\partial O_j$, recalling that $\Hd(J_u\cap \partial O_j)=0$, by the same properties of the convolutions of ${\rm BV}$-functions invoked \BBB below \EEE \eqref{uguale g} we get
\begin{equation}\label{L1 traces}
\{(u^{\sigma,{j}}_n)_{\partial O_j^-}\}_n\text{  converges to $u_{\partial O_j^-}$ strongly in $L^1_{\Hd}(\partial O_j)$ as $n\to +\infty$.}
\end{equation}

Next, we apply Anzellotti-Giaquinta's Theorem \cite[Theorem~1]{AnzellottiGiaquinta} on $O\setminus \overline{O_j}$ and obtain a sequence $\{w^j_n\}_n\subset C^\infty(O\setminus \overline{O_j})\cap L^\infty(O\setminus \overline{O_j})\EEE$  (recall that we assumed $u\in L^\infty(O)$) \EEE such that \begin{equation}\label{L1 convergence outside}
   \text{$\{w^j_{n}\}_n$ converges to $u$ strongly in $L^1(O\setminus \overline{O_j})$ as $n\to+\infty$}
\end{equation}and $|Dw^j_n|(O\setminus \overline{O_j})$ converges to $|Du|(O\setminus \overline{O_j})$ as $n\to+\infty$. Thus, by \eqref{derivationve on layer} 
\begin{equation}\label{stima V fuori}
  \limsup_{n\to+\infty}\mathcal{V}(w_n^j,O\setminus \overline{O_j})  \leq  \limsup_{n\to+\infty}|Dw^j_n|(O\setminus \overline{O_j})\leq \e_j \quad \text{for all $j\in\N$}.
\end{equation}
Moreover,  letting $(w^{j}_n)_{\partial O_j^+}$ (resp.\ \MMM $u_{\partial O^+_j}$) \EEE be the trace from the interior of $w^{j}_n\EEE$ (resp.\ $u$) on $\partial O_j^+$, by the continuity of the trace operator under strict convergence, we obtain 
\begin{equation}\label{convergence traces outside}
    \{(w^{j}_n)_{\partial O_j^+}\}_n\text{  converges to $u_{\partial O_j^+}$ strongly in $L^1_{\Hd}(\partial O_j)$ as $n\to +\infty$.}
\end{equation}

We let $\{z^{\sigma,j}_n\}_n\subset {\rm SBV}(O;\Rk)$ be the functions defined by \begin{equation*}
   z_n^{\sigma,j}\coloneq  \begin{cases}u^{\sigma,j}_n&\text{on $O_j$},\\
        w^j_n &\text{on $O\setminus \overline{O_j}$},
    \end{cases}
\end{equation*}
and observe that, by \eqref{L1 traces} and \eqref{convergence traces outside}, we have
\begin{equation}\label{no artificial jumps}\limsup_{n\to+\infty}\int_{J_{z_n^{\sigma,j}}\cap \partial O_j}|[z^{\sigma, j}_n]|\,{\rm d}\Hd=\limsup_{n\to+\infty}\int_{\partial O_j}|(w^{j}_n)_{\partial O_j^+}-(u^{\sigma,{j}}_n)_{\partial O_j^-}|\,{\rm d}\Hd=0
\end{equation}
for all $j\in\N$, where in the last equality we have used that $\Hd(J_u\cap \partial O_j)=0$ for all $j\in\N$. Finally, combining \eqref{usinman-u}, \eqref{Fusigman}, \eqref{L1 convergence outside}, \eqref{stima V fuori}, and \eqref{no artificial jumps} the same diagonal argument employed in Step 5 allows to define a sequence $\{u_n\}_n\subset {\rm SBV}(O) \MMM \cap L^\infty(O) \EEE $ converging to $u$ strongly in $L^1(O)$ such that  
\begin{equation*}
    \limsup_{n\to+\infty}\mathcal{V}(u_n,O)\leq \mathcal{V}(u,O).
\end{equation*}
In view of the lower semicontinuity of $\mathcal{V}$ (see Remark~\ref{remark: lowersemicontinuity}), this inequality implies \eqref{full convergence}, concluding the proof.
 \EEE 
\qed 
\medskip

\noindent\emph{Proof of Lemma \ref{lemma piecewise constants}.} \EEE
We first give the proof in the scalar case $k=1$ and for $u\in {\rm BV}(O)\cap L^\infty(O)$. Without loss of generality, we may assume that $\mathcal{V}(u, O\EEE\setminus J^1_u)>0$, as otherwise the function $u$ is already piecewise constant, see \cite[Theorem~4.23]{AFP}.   Moreover, we can assume that $u \ge 0$.

\medskip
\noindent \emph{Step 1.} (Scalar case $k=1$)
We let $M\coloneq \lceil\|u\|_{L^\infty(O)}\rceil<+\infty$, where for $a\in\R$ the symbol $\lceil a \rceil$ denotes the smallest integer greater than or equal to \EEE $a$. For $t\in (-\infty,+\infty)$ we define $E_{t}\coloneq \{x\in O: u(x)>t\}$. By the coarea formula we have 
 \begin{equation}\label{coarea lemma}
  \int_{0}^{M}|D\chi_{{E_t}}|(O\setminus J^{1}_{u}) \,{\rm d}t=\int_{0}^{M}\Hd\big(\partial^*E_t\cap (O\setminus J^1_{u})\big)\,{\rm d}t =\mathcal{V}(u,O\setminus J^{1}_{u}).
 \end{equation}
 Thanks to the well-known approximation properties of Lebesgue integrals by means of Riemann sums (see, for instance, \cite[Lemma~3.1]{DM-DonCantor})  we can find a sequence $\{\delta_n\}_n$ converging to $0$ as $n\to+\infty$ and $s\in  [0,1]\EEE$ such that, setting $t_n^{j}\coloneq s+\frac{j}{n}$ for   $j\in \mathcal{J}^n\coloneq \{j\in \mathbb{Z}: t_n^{j}\in (0,M)\}$, \EEE  we have that  $E_{t_n^{j}}$ \EEE is of finite perimeter in $O$ and  
 \begin{gather} \label{riemann integral}
    \Big| \sum_{j\in \mathcal{J}^n\EEE}\frac{1}{n}|D\chi_{E_{t_n^{j}}}|(O\setminus J^{1}_{u})- \int_{0}^{M}|D\chi_{{E_t}}|(O\setminus J^{1}_{u})\,{\rm d}t \Big|\leq \delta_n
 \end{gather}
for every $n\in\N$. We then define  
\begin{equation}\label{defpc}
u_n\coloneq \frac{1}{n}\sum_{ j\in \mathcal{J}^n\EEE}\chi_{E_{t_{n}^{ j\EEE}}}\in {\rm BV}(O)\cap L^\infty(O)
\end{equation}
and observe that 
\begin{equation*}
    |u_n(x)-u(x)|\leq \frac{1}{n}\quad \text{ for $\Lb^d$-a.e.\ }x\in O \text{ and for all $n\in\N$},
\end{equation*}
which implies  strong $L^1(O)$-convergence of 
$\{u_n\}_n$  to $u$ as $n\to+\infty$. 

We want to estimate $\mathcal{V}(u_n,O)$ by means of $\mathcal{V}(u,O)$. 
We write 
 \begin{equation}\label{first decomposition}
 \mathcal{V}(u_n,O)=\mathcal{V}(u_n, J^{1}_u)+\mathcal{V}(u_n,O\setminus J^{1}_u).
\end{equation} The first term can be bounded from above as 
\begin{equation}\label{estimate J1u constant}
    \mathcal{V}(u_n, J^{1}_u)\leq \Hd(J^{1}_u) = \mathcal{V}(u,J^{1}_u).
\end{equation}
As for the second  term, \EEE we observe that  by \eqref{defpc} \EEE  
\begin{align*}
    \mathcal{V}(u_n,O\setminus J^{1}_u)& \le  \int_{(J_{u_n}\setminus J^{1}_{u})\cap O}|[u_n]|\,{\rm d}\Hd= |Du_n|(O\setminus J^{1}_u)  \leq \sum_{j\in \mathcal{J}^n\EEE}\frac{1}{n}|D\chi_{ E_{t_{n}^{j}\EEE}}|(O\setminus J^{1}_{u}),
\end{align*}
which by \eqref{riemann integral} and \eqref{coarea lemma} implies 
\begin{equation*}
    \limsup_{n\to+\infty} \mathcal{V}(u_n,O\setminus J^{1}_u)\leq \mathcal{V}(u,O\setminus J^1_u).
\end{equation*}
Combining the previous inequality with \eqref{first decomposition} and \eqref{estimate J1u constant}, we obtain  
\begin{equation*}
     \limsup_{n\to+\infty} \mathcal{V}(u_n,O)\leq \mathcal{V}(u,O\setminus J^1_u)+ \mathcal{V}(u,J_u^{1}) =\mathcal{V}(u,O).
\end{equation*}
Moreover, as $\{u_n\}_n$ converges to $u$ strongly in $L^1(O)$, the lower semicontinuity of $\mathcal{V}(\cdot,O)$  with respect to the strong $L^1$-convergence (see Remark~\ref{remark: lowersemicontinuity}) implies  $
    \lim_{n\to+\infty}\mathcal{V}(u_n,O)=\mathcal{V}(u,O)$.
This concludes the proof in the case $u\in {\rm BV}(O)\cap L^\infty(O)$. 

 If $u\in {\rm BV}(O)$, for every $m\in\N$ we consider the truncation $u^{(m)}\coloneq (u\lor -m)\land m$ and use the previous step to obtain a sequence $\{u_{m,n}\}_n$ of piecewise constant functions converging to $u^{(m)}$ strongly in $L^1(O)$ as $n\to+\infty$ and such that 
\begin{equation*}
\lim_{n\to+\infty}\mathcal{V}(u_{m,n},O)=\mathcal{V}(u^{(m)},O).
\end{equation*}
Since   by \cite[Remark~2.7]{DalToaConvex} 
\begin{equation*}
    \lim_{m\to+\infty}\mathcal{V}(u^{(m)},O)=\mathcal{V}(u,O) \quad \text{ and }\quad u^{(m)}\to u\text{ strongly in $L^1(O)$ as $m\to+\infty$},
\end{equation*}
a diagonal argument allows us to conclude.

\MMM \medskip
\noindent \emph{Step 2.} (Vector-valued case $k>1$)  \EEE
In the vector-valued case $k\geq 2$, we apply the result we have just proved in the scalar case to each component. 
For each $i\in\{1,\dots,k\}$ and $n\in\N$ this yields  a finite collection of
sets of finite perimeter  $\{E^{j}_{n,i}\}_{j}$,   and numbers  $\{t^{j}_{n,i}\}_{j}$ \EEE such that $ u_{n,i}\coloneq \sum_{j}  t^{j}_{n,i}\chi E^{j}_{n,i} \EEE \in {\rm BV}(O)\cap L^\infty(O)$ converges to $u_i$ strongly in $L^1(O)$ as $n\to+\infty$ and 
\begin{equation}\label{component V piecewise}
    \lim_{n\to+\infty}\mathcal{V}(u_{n,i}\EEE,O)=\mathcal{V}(u_i,O).
\end{equation}
 For every $n\in\N$,  we let $u_n\in {\rm BV}(O)\cap L^\infty(O)$ be the function whose $i$-th component is given by $u_{n,i}$, which is  clearly piecewise constant.  \EEE 
By \eqref{component V piecewise} we then have  
\begin{equation*}
   \lim_{n\to+\infty} \mathcal{V}(u_n,O) =\mathcal{V}(u,O).
\end{equation*}
As  $\{u_{ n,i\EEE}\}_n$ converges to $u_i$ strongly in $L^1(O)$  for all $i\in\{1,\dots,k\}$, the proof  is concluded.\qed

\section{Integral representation of  bulk and surface parts:  Proof of Theorem~\ref{thm: integral star}}\label{section: Global Method}

 This section is devoted to the proof of Theorem~\ref{thm: integral star}, that is, the integral representation of the bulk and surface parts   of functionals in $\F^p_\star$. The arguments we  employ here are those of the Global Method of Bouchitté, Fonseca, \& Mascarenhas  \cite{BFM1998} \MMM  (see also \cite{BFML2002}),  \EEE  later adapted to the setting of structured deformations  in several different works, see for instance \cite {FonsHafParo,FriMatZap,BarrosoMatiasZappale1,BarrosoMatiasZappale2}.   Due to the different growth conditions and underlying function spaces considered here, we cannot always rely on the theory developed in the above mentioned works.   Therefore,   at several points, significant modifications are necessary and, in particular,  Poincar\'e inequalities and blow-up limits  tailor-made for our functional setting  are required. These arguments are presented in the appendices, see  Lemmas  \ref{products with smooth functions}, \ref{lemma blow up ac}, and   \ref{lemma: blow up jump points}, and Corollary  \ref{cor:small sets}. \EEE

\subsection{The global method}
 For notational convenience, \EEE in this section we  depart from the notation introduced in \MMM Section \ref{setting and main results}, \EEE and will use $(u,U)$, $(v,V)$, and $(w,W)$ to denote structured deformations in place of the usual $(g,G)$. We keep  $\mathcal{F}\colon \SD^p_\star(\Omega;\Rk)\times \B(\Omega)\to [0,+\infty)$ fixed throughout the section, and we specify at each step the properties we assume on $\cF$. 
We begin by proving that $\cF$ satisfies a suitable version of the Fundamental Estimate, which will be used to \MMM alter \EEE boundary conditions for the displacement fields $g$ in an energetically convenient way. 

\begin{lemma}[Fundamental Estimate\EEE]\label{lemma: Fundamental Estimate}
Let $\eta>0$ and let $O,O',O''\in\Op(\Omega)$  be \EEE with Lipschitz boundary and with $O'\subset \subset O$. Assume that $\cF$ satisfies conditions {\rm (H1), (H2),} and {\rm (H4)}.  Then, for every $u\in {\rm GBV}_\star(O;\Rk)$ and $v\in{\rm GBV}_\star(O'';\Rk)$, with $u\in L^1(O\setminus O';\Rk)$ and $v\in L^1((O\setminus O')\cap O'';\Rk)$, there exists a function $\varphi\in C^\infty(\Rd;[0,1])$ such that, setting $w\coloneq \varphi u+(1-\varphi)v$, we have $w\in {\rm GBV}_\star(O'\cup O'';\Rk)$, 
    \begin{equation}\label{thebc}
w=u \quad \text{ $\Lb^d$-a.e.\ in }O' \quad \text{ and }\quad w=v \quad \text{ $\Lb^d$-a.e.\ in } O''\setminus O,
\end{equation}
and
\begin{align}
\label{claim fundamental estimate} 
   \hspace{-0.3 cm}      \cF(w,W,O'\cup O'')&\leq (1+\eta)\big(\cF(u,W,O)+\cF(v,W,O'')\big) +M\|u-v\|_{L^1((O\setminus O')\cap O'')}+\eta \Lb^d(O'\cup O'') 
    \end{align}
    for every $W\in L^p(O\cup O'';\Rkd)$, where $M=M(O,O',O'',d,\MMM \alpha, \beta, \eta, \EEE k)$ is a positive constant independent of $u$ and $v$. Moreover, if $\rho>0$ and $x\in\Rd$ are such that $O_{\rho,x},O'_{\rho,x},\, O''_{\rho,x}\subset \Omega$ {\rm (}see \eqref{def blow up set}{\rm)}, then 
    \begin{equation*}
    M(O_{\rho,x},O'_{\rho,x},\, O''_{\rho,x},d,\alpha,\beta,\eta\EEE,k)=\frac{1}{\rho}M(O,O',O'',d, \MMM \alpha, \beta,  \eta, \EEE k).
    \end{equation*}
\end{lemma}
\begin{proof}
Let  us fix $N\in\N$ and consider open subsets $O_1,\dots, O_{N+1}$    of $\Rd$ with 
\begin{equation*}
O'\subset \subset O_1\subset \subset \dots\subset \subset O_{N+1}\subset \subset O.
\end{equation*}
For $i\in\{1,\dots, N\}$ we consider functions $\varphi^i\in C^\infty_c(O_{i+1};[0,1])$ with $\varphi^i=1$ in a neighbourhood of  $\overline{O_{i}}$  \EEE  and with $|\nabla\varphi^i|\leq 2N/\delta$, where $\delta\coloneq {\rm dist}(O',\Rd\setminus O)>0$. Thanks to Lemma~\ref{products with smooth functions}, for every $i\in\{1,\dots,N\}$ the function $w^i\coloneq \varphi^i u +(1-\varphi^i)v$ belongs to ${\rm GBV}_\star(O'\cup O'';\Rk)$, where we have extended $u$ and $v$ arbitrarily outside  \MMM of $O$ and $O''$, \EEE respectively.   

 We set $S_i\coloneq O''\cap (O_{i+1}\setminus \overline{O_i})$ and fix $i\in\{1,\dots,N\}$. By the measure property \ref{hyp:H1} and locality \ref{hyp:H2} we see that 
\begin{equation}\label{intial estimate fundamental}
    \cF(w^i,W,O'\cup O'')\leq \cF(u,W,O)+\cF(v,W,O'')+\cF(w^i,W,S_i)
\end{equation}
for every $W\in L^p(O\cup O'';\Rkd)$. 
To estimate the term $\cF(w^i, W, \EEE S_i),$ we exploit the upper bounds in \ref{hyp:H4} to get 
\begin{equation}\label{estimate on the strip}
    \cF(w^i,W,S_i)\leq \beta\sum_{j=1}^k\Big(\int_{S_i} \MMM | \EEE \nabla w^i_j|\,{\rm d}x+|D^cw^i_j|(S_i)+\int_{J_{w^i_j}\cap S_i}|[w^i_j]|\land 1\,{\rm d}\Hd\Big)+\beta\int_{S_i}(|W|^p+1)\,{\rm d}x.
\end{equation}
By Lemma~\ref{products with smooth functions} we have 
\begin{align*}
    &\nabla w^i=\nabla\varphi^i\otimes (u-v)+\varphi^i\nabla u+(1-\varphi^i)\nabla v\quad \text{ $\Lb^d$-a.e. in $S_i$},
    \\
    & D^cw^i=\varphi^i D^cu+(1-\varphi^i)D^cv\quad \text{ as Borel measures on $S_i$},
    \\
    &[w^i]=\varphi^i[u]+(1-\varphi^i)[v] \quad \text{ $\Hd$-a.e.\  on $S_i\cap J_{w^i}$},
\end{align*}
so that 
\begin{align*}
    \int_{S_i}|\nabla w^i_j|\,{\rm d}x & \leq \int_{S_i}(|\nabla u_j|+|\nabla v_j|)\,{\rm d}x+\|\nabla \varphi^i\|_{L^\infty(\Omega)}\int_{S_i}|u-v|\,{\rm d}x,\\
    |D^cw^i_j|(S_i) & \leq |D^cu_j|(S_i)+|D^cv_j|(S_i),\\
    \int_{S_i\cap \MMM J_{w^i_j} \EEE }|[w^i_j]|\land 1\,{\rm d}\Hd& \leq  \int_{S_i\cap J_{u_j}}|[u_j]|\land 1\,{\rm d}\Hd+ \int_{S_i\cap J_{v_j}}|[v_j]|\land 1\,{\rm d}\Hd
\end{align*}
for all $j\in\{1,\dots,k\}$.
Combining the previous inequalities with \eqref{estimate on the strip}, recalling that $|\nabla\varphi^i|\leq 2N/\delta$, and using the lower bounds in \ref{hyp:H4}, we get   
\begin{align*}
    \cF(w^i,W,S_i)&\leq \beta\sum_{j=1}^k\Big(\int_{S_i}(|\nabla u_j|+|\nabla v_j|)\,{\rm d}x+\frac{2N}{\delta}\int_{S_i}|u-v|\,{\rm d}x +|D^cu_j|(S_i)+|D^cv_j|(S_i)\\
    &\quad +\int_{S_i\cap J_{u_j}}|[u_j]|\land 1\,{\rm d}\Hd+ \int_{S_i\cap J_{v_j}}|[v_j]|\land 1\,{\rm d}\Hd\Big)+\beta\int_{S_i}|W|^p\,{\rm d}x+\beta\Lb^d(S_i)\\
    &\leq \frac{\beta}{\alpha}\Big(\cF(u,W,S_i)+\cF(v,W,S_i)\Big)+\frac{2\beta N k}{ \delta}\int_{S_i}|u-v|\,{\rm d}x+\beta\Lb^d(S_i).
\end{align*}
We set $S\coloneq (O\setminus O')\cap O''$,  sum the previous inequality over $i\in\{1,\dots,N\}$  and obtain   
\begin{equation*}
  \sum_{i=1}^N\cF(w^i,W,S_i)\leq \frac{\beta}{\alpha }\Big(\cF(u,W,S)+\cF(v,W,S)\Big)+\frac{2\beta Nk}{\delta}\int_{S}|u-v|\,{\rm d}x+\beta\Lb^d(S).
\end{equation*}
Therefore, there exists $i^*\in\{1,\dots,N\}$ such that  
\begin{align*}
   \cF(  w^{i^*}, \EEE W,S_{i^*})\leq \frac{\beta}{\alpha N }\Big(\cF(u,W,S)+\cF(v,W,S)\Big)+\frac{2\beta k}{\delta}\int_{S}|u-v|\,{\rm d}x+\frac{\beta}{N}\Lb^d(S).
\end{align*}
Plugging this inequality back in \eqref{intial estimate fundamental}, we deduce  
\begin{align*}
     \cF( w^{i^*}, \EEE W,O'\cup O'')&\leq \cF(u,W,O)+\cF(v,W,O'')\\&\quad +\frac{\beta}{\alpha N }\Big(\cF(u,W,S)+\cF(v,W,S)\Big)+\frac{2\beta k}{\delta}\int_{S}|u-v|\,{\rm d}x+\frac{\beta}{N}\Lb^d(S).
\end{align*}
We now let $N\in\N$ be so large that $\max\{\beta/(\alpha N),\beta/N\}\leq \eta$   and set $M\coloneq  2\beta k /\delta$, so that the previous inequality yields \eqref{claim fundamental estimate}. \MMM Note that by construction $w^{i^*}$ satisfies \eqref{thebc}. \EEE

Finally, the remark concerning the behaviour of the constant $M$ under rescaling can be proved by using as cut-off functions  $\hat{\varphi}^i(y)\coloneq \varphi^i(x+\frac{1}{\rho}(y-x))$ for all $y\in(O_{i+1})_{\rho,x}$. 
\end{proof}

 We now prove a lemma that allows to compare the value of $\mathcal{F}$ on two nested open sets. This observation will be used to modify, in an energetically convenient way, the matrix-valued fields $U$ so as to impose suitable integral constraints.  
 
 \begin{lemma}\label{lemma: smaller larger}
Let $(u, U) \in {\rm SD}^p_\star(\Omega;\Rk)$ and let $O_1, O_2\in\Op(\Omega)$ with $O_1\subset \subset O_2$. Assume that $\cF$ satisfies {\rm \ref{hyp:H1}} and {\rm \ref{hyp:H4}}. Then 
\begin{equation}\label{claim lemma easy}
\cF(u, U,O_2) \leq   \cF(u,U,O_1)  +\beta\Big(\mathcal{V}(u,O_2 \setminus O_1) + \|U\|_{L^p(O_2\setminus O_1)}^p+ \Lb^d(O_2 \setminus O_1)\Big).
\end{equation}
\end{lemma}
\begin{proof} It is enough to use properties \ref{hyp:H1} and \ref{hyp:H4}. \EEE
\end{proof}

With the previous lemma at our disposal, we can prove a result that allows to compare the value of the minimisation problems $\m$ defined \MMM in \EEE  \eqref{minimisation problem} on cubes of different size.
\begin{lemma}\label{easy}
 Assume that $\cF$ satisfies properties {\rm (H1)} and {\rm (H4)}. Then for every $(u, U)\in {\rm SD}^p_\star(\Omega;\Rk)$   the following two conditions hold:  
\begin{itemize}
	\item [{\rm (}a{\rm)}]if $p>1$ and  $Q_\nu(x, \rho) \subset \Omega$,  we have 
	\begin{equation*}
	\limsup_{\theta \to 0^+} \m(u, U,Q_\nu(x,(1-\theta) \rho))
	\leq \m(u, U,Q_\nu(x,\rho));
	\end{equation*}

	\item[{\rm (}b{\rm)}] if $p=1$ and  $O \in \mathcal O(\Omega)$ with Lipschitz boundary,  we have 
	\begin{equation*}
	\limsup_{\theta \to 0^+} \m(u, U,O_\theta)
	\leq \m(u, U, O),
	\end{equation*}
	where $ O_\theta=\{x \in O\colon\, {\rm dist}(x, \partial O) > \theta\}$.
	\end{itemize}
\end{lemma}

\begin{proof}  
It is enough to repeat the arguments of \cite[Lemma~3.2]{FriMatZap}, replacing their \MMM Remark~3.1 \EEE by our Lemma~\ref{lemma: smaller larger},  the total variation of the symmetric distributional gradients with the reference functional $\mathcal{V}$,   and the use of \cite[Proposition~2.5]{FonsecaToader} by our Lemma~\ref{lemma: Extension of GBV}.\EEE  
\end{proof}
 
To proceed with the Global Method, it is convenient to introduce the collection of cubes 
\begin{equation*}
\mathcal{Q}(\Omega)\coloneq  \big\{Q_\nu (x, \rho)\subset \Omega\colon \, x \in \Omega,\,\, \nu \in \Sd, \text{ and } \rho>0\big\}.
\end{equation*}
Given $(u,U) \in {\rm SD}^p_\star(\Omega;\Rk)$, we introduce the measure  
\begin{equation}\label{def mu}
\mu \coloneq  \Lb^d\mres \Omega + |D^cu|+\mathcal{V}(u,\cdot)\mres J_u,
\end{equation}
and consider the minimisation problem $\m^\delta(u,U,O)$ defined for $\delta>0$ and $O\in\Op(\Omega)$ by
\begin{align*}
\m^\delta(u,U,O) \coloneq  \inf &\Big\{\sum\nolimits_{i=1}^\infty \m(u, U, Q_i)\colon \, Q_i \in  \MMM \mathcal Q(\Omega), \EEE \,\,Q_i \subset  O,\,\,  \\&\text{ with }Q_i \cap Q_j= \emptyset \,\,\, {\rm if } \, i \neq j,\,\, 
  	 {\rm diam}(Q_i) < \delta, \text{ and }  \  \mu\Big(O \setminus \bigcup\nolimits_{i=1}^\infty Q_i\Big)=0\Big\}.
\end{align*}
Clearly, the function $\delta \mapsto \m^\delta(u,U,O)$ is non-increasing, so that the quantity
\begin{align*}
\widetilde{\m} (u, U,O)\coloneq  \sup_{\delta >0} \m^\delta (u, U,O)= \lim_{\delta \to 0^+} \m^\delta(u, U,O)
\end{align*}
is well-defined.

The first step of the Global Method  consists in proving that  $\cF(u,U,O)=\widetilde{\m}(u,U,O)$. This is accomplished in the following lemma, whose proof is   along the lines of   \cite[Lemma~4.4]{CriFriSolo} and \cite[Lemma~4.2]{FonsHafParo}.  
\begin{lemma}
	\label{Lemma delta=tilde} Assume that $\cF\in \F^p_\star$.  Then,
	\begin{equation*}
	\cF(u, U,O) = \widetilde{\m}(u, U, O)
	\end{equation*}
    for all $(u, U)\in {\rm SD}_\star^p(\Omega;\Rk)$  and all 
	$O \in \mathcal O(\Omega)$.
     \begin{proof}
       Let us fix $(u,U)\in {\rm SD}_\star^p(\Omega;\Rk)$ and $O\in\Op(\Omega)$. Since for each cube $Q\in \mathcal{Q}(\Omega)$ with $Q\subset O$, we have $\m(u,U,Q)\leq \cF(u,U,Q)$ by definition, the measure property \ref{hyp:H1} then implies that $\m^\delta(u,U,O)\leq \cF(u,U,O)$ for all $\delta>0$, and, as a consequence, $\widetilde{\m}(u,U,O)\leq \cF(u,U,O)$.

       We now prove the converse inequality. To this end, let us fix $\delta>0$ and a family of cubes $\{Q^\delta_i\}_i\subset \cR(\Omega)$ as in the definition of $\m^\delta(u,U,O)$ such that 
    \begin{equation}\label{cubi minimali}
\sum_{i=1}^{\infty}\m(u,U,Q^\delta_i)\leq \m^\delta(u,U,O)+\delta.
       \end{equation}
By definition of $\m(u,U,Q^\delta_i)$, for each   $i\in\N$ we find   $(v_i^\delta, V^\delta_i)\in \SD^p_\star(Q^\delta_i;\Rk)$ such that $v_i^\delta = u$ near $\partial Q_i^\delta$,  $\int_{Q^\delta_i}(V^\delta_i-U)\,{\rm d}x=0$,  \EEE and 
\begin{align}\label{Quasiminimo su cubo}
\mathcal{F}(v_i^\delta,V^\delta_i,Q_i^\delta) \leq \m(u,U,Q_i^\delta) + \delta \mathcal{L}^d(Q_i^\delta).
\end{align}
For every   $n\in\N$ we then define 
\begin{align}\label{def of vdelta}
 &v^{\delta,n}  \coloneq  \sum_{i=1}^n v^\delta_i \chi_{Q^\delta_i} +  u \chi_{N^{\delta,n}}  \quad \text{ and } \quad v^\delta \coloneq  \sum_{i=1}^{\infty} v^\delta_i \chi_{Q^\delta_i} + u \chi_{N^\delta}, 
 \\ \notag 
 &V^{\delta,n}  \coloneq  \sum_{i=1}^n V^\delta_i \chi_{Q^\delta_i} +  U\chi_{N^{\delta,n}}  \quad \text{ and } \quad V^\delta \coloneq  \sum_{i=1}^{\infty} V^\delta_i \chi_{Q^\delta_i} + U\chi_{N^\delta},
\end{align}
 where we have set $N^{\delta,n} \coloneq  \Omega \setminus \bigcup_{i=1}^n Q^\delta_i$   and $N^\delta \coloneq  \Omega \setminus \bigcup_{i=1}^{\infty} Q^\delta_i$.  
 Since $v^\delta_i=u$ near $\partial Q^\delta_i$, it follows that $ (v^{\delta,n},V^{\delta,n})\in \SD^p_\star(O;\Rk)$ \BBB by Lemma~\ref{lemma: Extension of GBV}. \EEE  Moreover, by \eqref{cubi minimali}, \eqref{Quasiminimo su cubo}, the measure property \ref{hyp:H1},  the \EEE locality property \ref{hyp:H2}, and the upper and lower bounds in \ref{hyp:H4} we have     
 \begin{align} \notag
     \sup_{n\in\N} \Big(\mathcal{V}(v^{\delta,n},\Omega)+\|V^{\delta,n}\|^p_{L^p(\Omega)}\Big)&\leq \frac{1}{\alpha}\sup_{n\in\N}\Big(\cF(u,U,N^{\delta,n})+\sum_{i=1}^n\cF(v^{\delta,n}_i,V^{\delta,n}_i,Q^\delta_i)\Big)   
    \\ \notag & \leq \frac{1}{\alpha}\sup_{n\in\N}\Big(\cF(u,U,\Omega)+\m^\delta(u,U,O)+\delta\big(1+\Lb^d(O)\big) \Big)
     \\ &\leq \frac{ 2 \EEE \beta}{\alpha}\sup_{n\in\N}\Big(\mathcal{V}(u,\Omega)+\|U\|^p_{L^p(\Omega)}+ 2 + \EEE \Lb^d(\Omega)\Big)=:S<+\infty. \label{estimate V+||}
 \end{align}
    As $ v^{\delta,n}  \to v^\delta$  and $ V^{\delta,n} \to V^{\delta}$ pointwise $\Lb^d$-a.e.\ in $\Omega$ as $n\to+\infty$, by \cite[Theorem~4.7]{DonatiGBV} and \eqref{estimate V+||} we deduce that $\{(v^{\delta,n},V^{\delta,n})\}_n$ converges to $(v^\delta,V^\delta)$ in ${\rm SD}^p_\star(\Omega;\Rk)$ as $n\to+\infty$  if $p>1$\EEE.  To obtain the same property for $p=1$, it is enough to observe that computations similar to those that led to \eqref{estimate V+||} give
    \begin{equation*}
        \int_{O}|V^{\delta,n}-V^\delta|\,{\rm d}x=\sum_{i=n+1}^{\infty}\int_{Q_i^\delta}|V^{\delta}_i|\,{\rm d}x+\int_{N^\delta}|U|\,{\rm d}x\leq \frac{2\beta}{\alpha}\Big(\mathcal{V}(u,N^{\delta,n+1})+\|U\|_{L^1(N^{\delta,n+1})}+ \EEE \Lb^d(N^{\delta,n+1})\Big)
    \end{equation*}
      for every $n\in\N$, where we have also used that $(u,U)$ is a competitor for the minimisation problems $\m(u,U,Q^\delta_i)$. As $ \lim_{n\to \infty} \EEE \mu(N^{\delta,n})=\mu(N^{\delta})=0$, this implies that $\{V^{\delta,n}\}_n$ converges  to $V^\delta$ strongly in $L^1(O;\Rkd)$.\EEE
    
    Therefore, by locality \ref{hyp:H2}, the  lower semicontinuity property \ref{hyp:H3}, \eqref{cubi minimali}, and \eqref{Quasiminimo su cubo}
  it follows that 
\begin{align}\notag 
\mathcal{F}(v^\delta,V^\delta,O) & \leq  \liminf_{n\to+\infty}\cF(v^{\delta,n},V^{\delta,n},O)
= \liminf_{n\to+\infty}\Big( \sum_{i=1}^{n}\mathcal{F}(v_i^\delta,V^\delta_i,Q_i^\delta)  + \mathcal{F}(u,U,N^{\delta,n} \cap O)\Big) \\&\leq \notag \sum_{i=1}^{\infty}\mathcal{F}(v_i^\delta,V^\delta_i,Q_i^\delta)  + \mathcal{F}(u,U,N^\delta \cap O) 
\leq  \sum_{i=1}^{\infty }\big(\m(u,U,Q_i^\delta) + \delta \mathcal{L}^d(Q_i^\delta)\big) \notag\\ 
\label{eq: vdelta bound}
&\leq \m^\delta(u,U,O) + \delta(1+\mathcal{L}^d(O)),    
\end{align}
 where in the second step we used  $\mu(\partial Q^\delta_i)=0$  for all $i \in \N$    and in the fourth step \EEE we have used  that $\mu(N^\delta \cap O)  =  \mathcal{F}(u,U,N^\delta \cap O) = 0$, which are a simple consequence of the definition of $\{Q^\delta_i\}_i$  and the upper bound in  \ref{hyp:H4}. 

To conclude the proof,  it \EEE is now enough to show that $\{(v^\delta,V^\delta)\}_\delta$ converges to $(u,U)$ in $\SD^p_\star(O;\Rk)$ as $\delta\to 0^+$, since passing to the liminf for $\delta\to 0^+$ in \eqref{eq: vdelta bound} and using the lower semicontinuity \ref{hyp:H3} we would then conclude $\cF(u,U,O)\leq \widetilde{\m}(u,U,O)$.

We observe that by Remark~\ref{remark: lowersemicontinuity} and   \eqref{estimate V+||} there exists a subsequence,  still denoted by $\delta$ for convenience, \EEE such that
\begin{equation}\label{estimate on the difference}
  \MMM   \lim_{\delta \to 0^+} \EEE \mathcal{V}(u-v^{\delta},O)\leq 2S,
\end{equation}
which, up to \MMM passing to a further subsequence, \EEE implies 
\begin{equation}\label{estimate jump set}
    \MMM   \sup_{\delta>0} \EEE\Hd(J^1_{u-v^{\delta}}\cap O)\leq 3S.
\end{equation}
 We now show  that \MMM $\{v^{\delta}\}_{\delta}$ \EEE  converges to  $u$ in $L^0(O;\Rk)$ as \MMM $\delta\to 0^+$. \EEE  To this end, \EEE we apply  Corollary~\ref{cor:small sets} and \EEE  Remark~\ref{remark: scaling p=1}  on each cube $Q^\delta_i$ to the function $u- v_i^{\delta}$, which equals zero near $\partial Q^{\delta}_i$. We get Borel sets $\omega^{\delta}_i \subset Q^{\delta}_i$ such that  
\begin{gather}\label{small sets u-vdelta}\big(\mathcal{L}^d(\omega^{\delta}_i)\big)^{(d-1)/d} \leq  C_{\rm Poin}    \mathcal{H}^{d-1}\big( J^1_{u- v^{\delta}}  \cap Q^{\delta}_i \big), \\
\label{u-vdelta}\| u - v^{\delta}_i \EEE \|_{L^{1}(Q^{\delta}_i \setminus \omega^{\delta}_i)}\leq C_{\rm Poin}  {\delta} \mathcal{V}(u- v^{\delta}, \EEE  Q^{\delta}_i \EEE \setminus J^1_{u-v^{\delta}}) 
\end{gather}
 for a constant $C_{\rm Poin}>0$, independent of $i$ and of $\delta$.  In light of \eqref{def of vdelta} and the fact that  $\Lb^d(N^\delta)=0$, we can estimate
\begin{align}\label{titire1}
\int_O|u-v^{\delta}|\land 1\, {\rm d}x &  = \sum_{i=1}^{\infty} \int_{Q^{\delta}_i} |u-v_i^{\delta}|\land 1\, {\rm d}x \leq \sum_{i=1}^{\infty} \Big( \| u -v^{\delta}_i \|_{L^{1}(Q^{\delta}_i \setminus \omega^{\delta}_i)} + \mathcal{L}^d(\omega^{\delta}_i)    \Big).
\end{align}
Recalling that the sets $Q^{\delta}_i$ are pairwise disjoint, we may use \eqref{u-vdelta}   and \eqref{estimate on the difference} to get
\begin{align}\label{titire2}
\sum_{i=1}^{\infty} \| u -v^{\delta}_i \|_{L^{1}(Q^{\delta}_i \setminus \omega^{\delta}_i)} \leq   C\delta
\end{align}
for a constant $C$ independent of $\delta$.
Finally, since  $\omega^{\delta}_i \subset Q^i_{\delta}$ and ${\rm diam}(Q^{\delta}_i) \leq \delta$, we may use \eqref{estimate jump set},  \eqref{small sets u-vdelta}, and H\"older's inequality to obtain
\begin{align}\label{titire3}
\sum_{i=1}^{\infty} \mathcal{L}^d(\omega^{\delta}_i) \le \delta\sum_{i=1}^{\infty} \big(\mathcal{L}^d(\omega^{\delta}_i)\big)^{(d-1)/d}\le   C_{\rm Poin}  \delta \mathcal{H}^{d-1}\big( J^1_{u-v^{\delta}} \EEE \cap O \big)\leq C\delta. 
\end{align}
Putting together \eqref{titire1}--\eqref{titire3}, we get
\begin{equation*}
   \MMM   \lim_{\delta \to 0^+} \EEE \int_{O}|u-v^{\delta}|\land 1\,{\rm d}x=0.
\end{equation*}
  Since the equality above is equivalent to the $L^0(O;\Rk)$-convergence of $\{v^{\delta}\}_\delta $ to $u$ and the limit function $u$ does not depend on the chosen   subsequence of  $\delta$,   by Urysohn's principle we get    that the entire sequence \EEE $\{v^\delta\}_\delta$ converges to $u$ in $L^0(O;\Rk)$ as $\delta\to 0^+$.
  
  To conclude, we are left with showing that $\{V^\delta\}_\delta$ converges to $U$ weakly in $L^p(O;\Rkd)$ if $p>1$ and weakly$^*$ in $\mathcal M_b(O;\Rkd)$ if $p=1$. To prove this,  exploiting $\int_{Q^\delta_i}(V^\delta_i-U)\,{\rm d}x=0$, \EEE it is enough to use \cite[Lemma~3.3]{FonsHafParo} (note that the arguments \MMM  work also for   $\Rkd$ in place of $\R^{d\times N\times N}$), \EEE which implies that  $\{V^\delta\}_\delta$ converges to $U$ weakly$^*$  in $\mathcal{M}_b(O;\Rkd)$. This gives the desired convergence if $p=1$. If $p>1$, it is enough to note that by \eqref{estimate V+||} the sequence $\{V^\delta\}_\delta$ is uniformly bounded in $L^p(O;\Rkd)$, and thus converges weakly in $L^p(O;\Rkd)$ as well. This concludes the proof of the lemma. 
     \end{proof}
\end{lemma}

The following result shows that for $\mu$-a.e.\ $x\in \Omega$, the Radon-Nikodým derivative of $\cF$ with respect to $\mu$ can be characterised by means of the minimum values of the problems $\m$ on small cubes.

\begin{theorem}\label{thm: minimum=F}
Assume that $\mathcal{F}\in\F^p_\star$.  Then, for every $\nu \in \Sd$ and   $(u,U) \in \SD^p_\star(\Omega;\Rk)$, 
	we have 
	\begin{equation*}
	\lim_{\rho \to 0^+}\frac{{\mathcal F}(u, U, Q_\nu(x, \rho))}{\mu( Q_\nu(x,\rho))}= \lim_{\rho \to 0^+} \frac{\m (u, U, Q_\nu(x, \rho))}{\mu (Q_\nu(x,\rho))}
	\end{equation*}
	for $\mu$-a.e.\ $x \in \Omega$, where $\mu$ is the measure introduced in \eqref{def mu}.
	\end{theorem}
    \begin{proof}
        The result can be proved arguing exactly as in \cite[Theorem~4.3]{FonsHafParo}, replacing \cite[Lemmas 4.1 and 4.2]{FonsHafParo}  by our Lemmas \ref{easy} and \ref{Lemma delta=tilde}, respectively.  Indeed, \EEE \MMM the proof in \cite{FonsHafParo} \EEE relies on measure theoretical arguments alone and  does not require the use of \EEE the specific structure of the underlying function space.
    \end{proof}

\subsection{Characterisation of bulk and surface densities} In the rest of the section we prove Theorem~\ref{thm: integral star}. We prove separately equalities \eqref{repr a star thm} and \eqref{repre j star thm}, addressing first the integral representation of  the absolutely continuous part of $\cF$.

\begin{proposition}\label{prop: repre AC}
    Assume that   $\cF\in\mathfrak{F}^p_\star$.   Then  for every $(u,U)\in \SD^p_\star(\Omega;\Rk)$ and $B\in\B(\Omega)$ we have 
    \begin{equation}\label{repr a star prop}
	\mathcal{F}^a(u,U,B)= \int_Bf_{\rm bulk}\big(x,u,   \nabla  u, U\big)\,{\rm d} x,
    \end{equation}
    where $f_{\rm bulk}$ is the function defined by \eqref{def f}.
\end{proposition}

\begin{proof} 
To prove \eqref{repr a star prop}, we will characterise the Radon-Nikodým derivative of $\cF$ with respect to the Lebesgue measure, by showing that
	\begin{align}\label{fproof}
	\frac{{\rm d} \mathcal F(u, U,\cdot)}{{\rm d} \mathcal L^d} (x)=	
	f_{\rm bulk}\big(x, \widetilde{u}(x),   \nabla   u(x), \widetilde{U}(x)\big)
	\end{align}
     for $\Lb^d$-a.e.\ $x\in\Omega$, \MMM where here and in the following $\sim$ refers to the approximate limit, see \eqref{aplim} below. \EEE  Let us fix a point $x\in\Omega$ of approximate continuity of $u$, \BBB satisfying   \begin{align}\label{Lebesgue gradients0}
        &\lim_{\rho\to 0^+}\frac{1}{\rho^{d}}\int_{Q(x,\rho)}|\nabla u(y)-\nabla u(x)|\,{\rm d}y=0,\\\label{non singular0}
       & \lim_{\rho\to 0^+}\frac{1}{\rho^d}\Hd(J_{u}\cap Q(x,\rho))=0 \quad \text{ and } \lim_{\rho\to 0^+}\frac{1}{\rho^d}|D^cu|(Q(x,\rho))=0,
    \end{align}     
    and 
\begin{gather}\label{Besicovitch}  \frac{{\rm d}\cF (u,U,\cdot)}{{\rm d} \Lb^d} (x) =
\lim_{\rho \to 0^+}\frac{\cF(u, U,Q(x,\rho))}{\rho^d} =
\lim_{\rho \to 0^+}\frac{\m(u, U,Q(x,\rho))}{\rho^d},\\
\label{Lebesgue G}
\lim_{\rho\to 0^+}\frac{1}{\rho^d}\int_{Q(x,\rho)}|U(y)-\widetilde{U}(x)|\,{\rm d}x=0,
\end{gather}
\BBB and such that the statement of Lemma~\ref{lemma blow up ac} holds. \EEE 
Note that these conditions are satisfied by  $\Lb^d$-a.e.\ point of $\Omega$ thanks to Proposition~\ref{Prop:fine prop GBV}, Besicovitch and Lebesgue Derivation Theorems, and Theorem~\ref{thm: minimum=F}, so that it will be sufficient to show \eqref{fproof} for such a point $x$.

 Recalling the definition \eqref{def f} of the bulk density \MMM  $f_{\rm bulk}$, \EEE equality \eqref{Besicovitch} implies that \eqref{fproof} is equivalent to 
\begin{align}\label{bulk term to bound}
	 \lim_{\rho \to 0^+}\frac{\m(u,U, Q (x, \rho))}{\rho^d}=\lim_{\rho \to 0^+}
	\frac{\m(\ell_{x},\widetilde{U}(x), Q (x, \rho))}
	{\rho^d},
\end{align}
 where, for brevity,  we write  $\ell_{x} \coloneq  \ell_{x,\widetilde{u}(x), \nabla   u(x)}= \widetilde{u}(x)+\nabla  u(x)(\cdot-x)$. 
 This equality  is proved in \EEE two steps, dealing with the two inequalities in \eqref{bulk term to bound} separately. Before doing so, we  introduce \EEE some further notation related to \EEE the statement of Lemma~\ref{lemma blow up ac}. We   let $u_\rho$ and $\overline{u}_\rho$ be the functions defined for $y\in Q(x,\rho)$ by $u_\rho(y)\coloneq ({u}(y)-\widetilde{u}(x))/\rho$ and $\overline{u}_\rho(y)\coloneq (u_\rho(y)\lor \tau'(u_\rho,Q(x,\rho))\land \tau''(u_\rho,Q(x,\rho))$. Thanks to the  lemma, we have    
    \begin{align}
    \label{blowup ac 2}
    & \lim_{\rho\to 0^+}\frac{1}{\rho^{d+1}}\Lb^d(\{u_\rho\neq \overline{u}_\rho\})=0,
 \\ \label{blowup ac 3 u}  &  \lim_{\rho\to 0^+}\frac{1}{\rho^{d}}\int_{Q(x,\rho)}|\overline{u}_\rho(y)-\nabla u(x)\frac{(y-x)}{\rho}|\,{\rm d}y=0,\\ \label{blowup ac 4 u}
    &\lim_{\rho\to 0^+}\frac{1}{\rho^d}\int_{Q(x,\rho)}|\rho \nabla \overline{u}_\rho(y)-\nabla u(x)|\,{\rm d}y=0,\\\label{blowup ac 5 u}
    &\lim_{\rho\to 0^+}\frac{1}{\rho^d}\Hd(J_{\overline{u}_\rho}\cap Q(x,\rho))=0 \quad \text{ and } \lim_{\rho\to 0^+}\frac{1}{\rho^{d-1}}|D^c\overline{u}_\rho|(Q(x,\rho))=0.
    \end{align}

\medskip

\noindent \emph{Step 1.} ($\leq$) \MMM In this step we prove that \EEE  
 \begin{equation}\label{claim leq bulk}
      \lim_{\rho \to 0^+}\frac{\m(u,U, Q (x, \rho))}{\rho^d}\leq \lim_{\rho \to 0^+}
	\frac{\m(\ell_{x},\widetilde{U}(x), Q (x, \rho))}
	{\rho^d}.
\end{equation}
To this end, let us fix $\theta\in (0,1/4)$, set $Q^{\theta,\rho}\coloneq Q(x,(1-3\theta)\rho)$, and consider  $(v_\rho,V_\rho)$ belonging to $\C^p_\star(\ell_x,\widetilde{U}(x),Q^{\theta,\rho})$ such that 
\begin{equation}\label{quasiminimiser 1-3theta}
    \cF\big(v_\rho,V_\rho,Q^{\theta,\rho}\big)\leq \m\big(\ell_x,\widetilde{U}(x),Q^{\theta,\rho}\big)+\rho^{d+1}.
\end{equation}
\MMM We extend the function by setting  $v_\rho=\ell_x$  on  $Q(x,\rho)\setminus Q^{\theta,\rho}$.   
In view of \eqref{blowup ac 2},  Fubini's Theorem   implies that for every $\rho>0$  there exists  $s_\rho\in ((1-\theta) \rho,\rho)$ such that 
\begin{equation}\label{good shell ac}
    \lim_{\rho\to 0^+}\frac{1}{\rho^d}\Hd(\partial Q(x,s_\rho)\cap \{\overline{u}_\rho\neq u_\rho\})=0.
\end{equation}
We introduce the functions $\{z_\rho\}_\rho$  defined by \EEE
\begin{equation}\label{blowup ac 1} 
    z_{\rho}(y)\coloneq 
    \begin{cases}  \overline{u}_\rho \EEE (y) &  \text{ if $y\in Q(x,s_\rho)$,}\\
    u_\rho(y) &\text{ if $y\in Q(x,\rho)\setminus Q(x,s_\rho)$,}
    \end{cases}
\end{equation}
which, thanks to Lemma~\ref{lemma: Extension of GBV}, belong to $ \GBVs(Q(x,\rho);\Rk)$.  \EEE\EEE 
We note that by \eqref{Lebesgue gradients0}--\eqref{non singular0},     \eqref{blowup ac 3 u}--\eqref{blowup ac 5 u}, and \eqref{good shell ac}  they satisfy 
\begin{align}
    \label{blowup ac 3}  &  \lim_{\rho\to 0^+}\frac{1}{\rho^{d}}\int_{Q(x,(1-\theta)\rho)}|{z}_\rho(y)-\nabla u(x)\frac{(y-x)}{\rho}|\,{\rm d}y=0,\\ \label{blowup ac 4}
    &\lim_{\rho\to 0^+}\frac{1}{\rho^d}\int_{Q(x,\rho)}|\rho \nabla {z
    }_\rho(y)-\nabla u(x)|\,{\rm d}y=0,\\\label{blowup ac 5}
    &\lim_{\rho\to 0^+}\frac{1}{\rho^d}\Hd(J_{{z}_\rho}\cap Q(x,\rho))=0 \quad \text{ and } \lim_{\rho\to 0^+}\frac{1}{\rho^{d-1}}|D^c{z}_\rho|(Q(x,\rho))=0.
\end{align}
 Indeed, property  \eqref{blowup ac 3} is an immediate consequence of  \eqref{blowup ac 3 u} and of \eqref{blowup ac 1},  and   \eqref{blowup ac 4} follows from \eqref{Lebesgue gradients0} and  \eqref{blowup ac 4 u}. \EEE
As for the remaining property \eqref{blowup ac 5}, we use  \eqref{non singular0}, \eqref{blowup ac 5 u},   \eqref{good shell ac}, and 
\begin{equation*}
     J_{z_\rho}\subset \big(J_{u}\cap Q(x,\rho)\big)\cup  \big(\partial Q(x,s_\rho)\cap \{\overline{u}_\rho\neq u_\rho\}\big).
\end{equation*}
We set   
 \begin{equation}\label{def sets O O' O''}
    O'\coloneq Q(x,{1-2\theta}),\quad O\coloneq Q(x,1-\theta),\quad \text{ and }O''\coloneq Q(x,1)\setminus \overline{ \MMM Q(x,{1-3\theta}) \EEE },
\end{equation}
 and  
\begin{equation}\label{def Wrho 1} W_\rho(y)\coloneq  
\begin{cases}
 V_\rho(y) &\text{ in }Q^{\theta,\rho},\\
\displaystyle \frac{1}{\mathcal L^d(Q(x,\rho)\setminus Q^{\theta,\rho})}\Big(\int_{Q(x,\rho)} U(y)\, {\rm d} y-  \mathcal{L}^d(Q^{\theta,\rho}) \widetilde{U}(x)\EEE \Big) &\text{ in }Q(x,\rho)\setminus Q^{\theta,\rho}.
\end{cases}
\end{equation}
 Since \EEE  $(v_\rho,V_\rho)\in\C^p_\star(\ell_x,\widetilde{U}(x),Q^{\theta,\rho})$, we have
\begin{equation*}
    \int_{Q^{\theta,\rho}}  V_\rho(y) \, {\rm d} y = \int_{Q^{\theta,\rho}} \widetilde{U}(x) \,{\rm d} y =  \mathcal{L}^d(Q^{\theta,\rho}) \widetilde{U}(x), \EEE
\end{equation*}    
so that 
\begin{equation}\label{wrho average}
    \int_{Q(x,\rho)}W_\rho(y) \, {\rm d}y=\int_{Q(x,\rho)}{U}(y)\,{\rm d}y.
\end{equation}

We now want to apply Lemma~\ref{lemma: Fundamental Estimate} to $v_\rho$  with \EEE $u$ as a boundary datum near $\partial Q(x,\rho)$. We begin by observing that $v_\rho\in L^1(O_{\rho,x}\setminus O'_{\rho,x};\Rk)$ since $v_\rho=\ell_x$  $\Lb^d$-a.e.\ in $Q(x,\rho)\setminus Q^{\theta,\rho}$. Moreover, from the definition \eqref{blowup ac 1} of $z_\rho$, the fact that $s_\rho\geq (1-\theta)\rho$,  and the fact that $\bar{u}_\rho$ is bounded \EEE it follows immediately that $\widetilde{u}(x)+\rho z_\rho\in L^1(O_{\rho,x}\setminus O'_{\rho,x};\Rk)$. Hence,
for fixed $\eta>0$, we may apply Lemma~\ref{lemma: Fundamental Estimate} to $v_\rho$ (in place of $u$) and to $\widetilde{u}(x)+\rho z_\rho$ (in place of $v$) and with $O',O'',O$ as in \eqref{def sets O O' O''}.  We \EEE obtain a constant $M>0$, depending on $\theta$ but independent of $\rho$, and $w_\rho\in {\rm GBV}_\star(Q(x,\rho);\Rk)$ such that 
\begin{align}\notag 
\cF(w_\rho,W_\rho,Q(x,\rho))&\leq(1+\eta)\big(\cF(v_\rho,W_\rho,O_{\rho,x})+\cF(\widetilde{u
}(x)+\rho z_\rho,W_\rho, O''_{\rho,x})\big)
   \\& \label{fundamental applied 1} \quad +\frac{M}{\rho}\|(\widetilde{u}(x)+\rho z_\rho)-v_\rho\|_{L^1(O_{\rho,x}\setminus O'_{\rho,x})}+\eta\rho^d.
\end{align}
  Moreover, the same lemma and  \eqref{blowup ac 1} ensure that  
 \begin{equation}\label{wrho uguale cosa}
   w_\rho=v_\rho \quad \text{ $\Lb^d$-a.e.\ in $O_{\rho,x}'$} \quad \text{ and }\quad w_\rho=\widetilde{u}(x)+\rho z_\rho=u \quad \text{ $\Lb^d$-a.e.\ in $Q(x,\rho)\setminus Q(x,s_\rho)$}.
\end{equation}
From the second equality in \eqref{wrho uguale cosa} and \eqref{wrho average} we deduce that $(w_\rho,W_\rho)\in \C^p_\star(u,U,Q(x,\rho))$. Therefore,  by  \eqref{fundamental applied 1} \MMM and the fact that $v_\rho=\ell_x$  $\Lb^d$-a.e.\ in $Q(x,\rho)\setminus Q^{\theta,\rho}$ \EEE we infer 
\begin{align}\notag 
\m(u,U,Q(x,\rho))&\leq(1+\eta)\big(\cF(v_\rho,W_\rho,O_{\rho,x})+\cF(\widetilde{u}(x)+\rho z_\rho,W_\rho, O''_{\rho,x})\big)\\\label{fundamental applied with minimum}
&\quad +\frac{M}{\rho}\|\rho z_\rho-\nabla u(x)(\cdot-x)\|_{L^1(O_{\rho,x}\setminus O'_{\rho,x})}+\eta\rho^d.
\end{align}
We now estimate each term  on \EEE the right-hand side of \eqref{fundamental applied with minimum}.  First, note that   \eqref{blowup ac 3} implies
\begin{equation}\label{L1 term}
    \lim_{\rho\to 0^+}\frac{M}{\rho^{d+1}}\|\rho z_\rho-\nabla u(x)(\cdot-x)\|_{L^1(O_{\rho,x}\setminus O'_{\rho,x})}=0.
\end{equation}
We  use Lemma~\ref{lemma: smaller larger}, with $O_1=Q^{\theta,\rho}$ and $O_2=O_{\rho,x}$, locality \ref{hyp:H2},   and \eqref{def Wrho 1} to get \begin{align} \notag
    \cF(v_\rho,W_\rho,O_{\rho,x})&\leq \cF(v_\rho,V_\rho,Q^{\theta,\rho})+\beta \mathcal{V}(v_\rho,Q(x,\rho)\setminus Q^{\theta,\rho})\\\label{estimate on Orhox bulk}
    & \quad +\beta\|W_\rho\|_{L^p(Q(x,\rho)\setminus Q^{\theta,\rho})}^p+\beta(1-(1-3\theta)^d)\rho^d.
\end{align}
Recalling \eqref{quasiminimiser 1-3theta} and that 
$v_\rho=\ell_x$ on $ Q(x,\rho)\setminus Q^{\theta,\rho}$, we then get
\begin{align}\label{altrastima}
&\cF(v_\rho,V_\rho,Q^{\theta,\rho})\leq \m(\ell_x,\widetilde{U}(x),Q^{\theta,\rho})+\rho^{d+1}  \le \beta k(|\nabla u(x)|+|\widetilde{U}(x)|^p+1)\rho^d +\rho^{d+1}, \EEE \\
\label{stima zrho on Vbk}
&\mathcal{V}(v_\rho,Q(x,\rho)\setminus Q^{\theta,\rho})\leq (1-(1-3\theta)^d)k|\nabla u(x)|\rho^d.
\end{align}
 Here, in the second inequality of \eqref{altrastima} we used that  $(\ell_x,\widetilde{U}(x))$ is a competitor for $\m(\ell_x,\widetilde{U}(x),Q^{\theta,\rho})$ and thus $  \m(\ell_x,\widetilde{U}(x),Q^{\theta,\rho})\leq \beta k(|\nabla u(x)|+|\widetilde{U}(x)|^p+1)\rho^d$ by the upper bound in   (H4).    \EEE
As for the remaining term in \eqref{estimate on Orhox bulk}, we observe that 
\begin{align*}
 &\int_{Q(x,\rho)\setminus Q^{\theta,\rho}} | W_\rho  |^p \,{\rm d} y \leq 
\frac{1}{\rho^{d(p-1)}(1-(1-3\theta)^d)^{p-1}}\Big|\int_{Q(x,\rho)} U(y) \,{\rm d}y-\int_{Q^{\theta,\rho}} \widetilde{U}(x) \,{\rm d}y\Big|^p  \nonumber\\
&\hspace{0,3cm}\leq\frac{C}{\rho ^{d(p-1)}(1-(1-3\theta)^d)^{p-1}}\Big( \Big|\int_{Q(x,\rho)\setminus Q^{\theta,\rho}} \hspace{-0.5 cm}U(y) \, {\rm d}y\Big|^p+ \Big|\int_{Q^{\theta,\rho}}(U(y)-\widetilde{U}(x)) \,{\rm d} y\Big|^p\Big) \nonumber \\\notag 
&\hspace{0,3cm}\leq\frac{C \rho^{dp}}{\rho ^{d(p-1)}(1-(1-3\theta)^d)^{p-1}}\Big( \Big|\fint_{Q(x,\rho)}\hspace{-0,1cm} U(y) \,{\rm d} y- (1-3\theta)^d\fint_{Q^{\theta,\rho}}\hspace{-0,1cm}U(y) \,{\rm d} y \Big|^p\\
& \hspace{5 cm}+ \Big|(1-3\theta)^d\fint_{Q^{\theta,\rho}}\hspace{-0,1cm}(U(y) -\widetilde{U}(x))\,{\rm d}y\Big|^p\Big)
\end{align*}
for a   constant $C>0$ depending only on $p$. 
In view of \eqref{Lebesgue G}, this implies
\begin{equation}\label{asymptotic Wrho}
    \limsup_{\rho\to 0^+}\frac{1}{\rho^d}\int_{Q(x,\rho)\setminus Q^{\theta,\rho}} | W_\rho  |^p \,{\rm d}y \leq {C}|(1-(1-3\theta)^d)|\widetilde{U}(x)|^p. 
\end{equation}
In a similar fashion, observing that $\Lb^d(O''_{\rho,x})=\MMM (1-(1-3\theta)^d)\rho^d \EEE $   and that $\mathcal{V}$ is invariant by adding constants, we may estimate 
\begin{equation} \label{stima rurho}\
    \cF(\widetilde{u}(x)+\rho z_\rho,W_\rho, O''_{\rho,x})\leq \beta \mathcal{V}(\rho z_\rho,O''_{\rho,x})+\beta\|W_\rho\|_{L^p(Q(x,\rho)\setminus Q^{\theta,\rho})}^p
     +  \beta(1-(1-3\theta)^d)\rho^d,
\end{equation}
and observe that
\begin{equation*}
   \mathcal{V}(\rho z
   _\rho,O''_{\rho,x})\leq k\Big(\int_{O''_{\rho,x}}|\rho\nabla z_\rho|\,{\rm d}x+\rho|D^cz_\rho|(O''_{\rho,x})+\Hd(J_{z_\rho}\cap O''_{\rho,x})\Big).
\end{equation*}
Note that by \eqref{blowup ac 4} and \eqref{blowup ac 5} we have 
\begin{equation}\label{ultima step 1 ac}
   \limsup_{\rho\to 0^+}   \frac{1}{\rho^d} \mathcal{V}(\rho z_\rho,O''_{\rho,x})\leq k(1-(1-3\theta)^d)|\nabla u(x)|.
\end{equation}

We now combine   \eqref{fundamental applied with minimum}--\eqref{stima zrho on Vbk} and \eqref{asymptotic Wrho}--\eqref{ultima step 1 ac}\EEE,  \EEE divide by $\rho^{d}$, pass to the limsup as $\rho\to 0^+$, and obtain
\begin{align*}
\limsup_{\rho\to 0^+} &\frac{\m(u,U,Q(x,\rho))}{\rho^d}\leq(1-3\theta)^d\limsup_{\rho '\to 0^+}\Big( \frac{\m(\ell_x,\widetilde{U}(x),Q(x,\rho'))}{(\rho')^d}\Big)+\eta\beta k(|\nabla u(x)|+|\widetilde{U}(x)|^p+1)\\ & \quad \quad +(1+\eta)\Big(\beta(1-(1-3\theta)^d)(2+2k|\nabla u(x)|)
 +\frac{2C((1-(1-3\theta)^d)^p}{(1-(1-3\theta)^d)^{p-1}} |\widetilde{U}(x)|^p\Big)+\eta,
\end{align*}
where we have set $\rho'\coloneq (1-3\theta)\rho$.
Letting $\theta \to 0^+$ and $\eta\to 0^+$, we obtain \eqref{claim leq bulk}. 

\medskip
\medskip

\noindent \emph{Step 2.} ($\geq$)  The proof of this inequality  follows along the same lines of Step 1 by interchanging the roles of $(u,U)$ and $(\ell_{x},\widetilde{U}(x))$. The essential difference is that the Fubini-type argument in \eqref{good shell ac} is performed for some $s_\rho\in ((1-4\theta)\rho,(1-3\theta)\rho)$, we choose a competitor $(v_\rho,V_\rho)\in \C^p_\star(u,U,\MMM Q(x,s_\rho)) \EEE $ satisfying 
\begin{equation*}
    \cF(v_\rho, V_\rho,Q(x,s_\rho))\leq \m(u,U,Q(x,s_\rho))+\rho^{d+1},
\end{equation*}
 \MMM  we extend    $v_\rho$ \EEE  to the whole $Q(x,\rho)$ by setting  $v_\rho=  \MMM \widetilde{u}(x) + \rho \EEE \overline{u}_\rho$  outside of $Q(x,s_\rho)$   and \BBB we define $W_\rho$ outside of $Q(x,s_\rho)$ by \EEE
$${\MMM W_\rho \EEE  = \frac{1}{\mathcal L^d(Q(x,\rho)\setminus Q(x,s_\rho))}\Big(\mathcal{L}^d(Q(x,\rho)) \widetilde{U}(x) -  \int_{Q(x,s_\rho)} U(y)\,{\rm d}y\Big)}$$ 
This ensures that the Fundamental Estimate can again be employed to adjust the boundary values. We omit the details.  
\end{proof}

We \MMM proceed by \EEE addressing the integral representation of the jump part of $\cF$, concluding the proof of Theorem~\ref{thm: integral star}.
\begin{proposition}\label{prop:repre surf}
 Assume that  $\cF\in\mathfrak{F}^p_\star$.    Then  for every $(u,U)\in \SD^p_\star(\Omega;\Rk)$ and $B\in\B(\Omega)$ we have 
    \begin{equation}\label{repr j star prop}
	\mathcal{F}^j(u, U, B)= \int_{J_u\cap  B} \!f_{\rm surf}\big(x,u^+,u^-,\nu_{u})\,{\rm d}\Hd.
    \end{equation}
    where $f_{\rm surf}$ is the function defined by \eqref{def Psi}.
\end{proposition}

\begin{proof}
To prove \eqref{repr j star prop}, we characterise the Radon-Nikodým derivative of $\cF$ with respect to the measure $\Hd\mres J_u$ by showing that
\begin{align}\label{psiproof}
	\frac{{\rm d} \mathcal F(u, U,\cdot)}{{\rm d} \Hd\mres J_u} (x)=	
	f_{\rm surf}\big(x, u^+(x), u^-(x), \nu_u(x)\big),
	\end{align}
     for $\Hd$-a.e.\ $x\in J_u$.   
Let us fix a point $x\in J_u$ \BBB satisfying 
\begin{align}\label{blow up jump 20}
   &\lim_{\rho\to 0^+}\frac{1}{\rho^{d-1}}\int_{Q_{\MMM \nu_u(x)}(x,\rho)}|\nabla u|\,{\rm d}y=0\quad \text{ and }\quad    \lim_{\rho\to 0^+}\frac{1}{\rho^{d-1}}|D^cu|(Q_{\MMM \nu_u(x)}(x,\rho))=0,\\
 \label{blow up jump 50}
 &\lim_{\rho\to 0^+}  \frac{1}{\rho^{d-1}} \Hd(J_u\cap  E_{\rho,x}  )= \Hd( \MMM \Pi^{\nu_u(x)}_x \EEE \cap E) \quad \text{for all Borel sets } E \subset Q_{\MMM \nu_u(x)}( \MMM x  \EEE ,1),
\end{align}
\MMM where     $\Pi^\nu_{x}\coloneq \{y \in \R^d\colon \, (y-x)\cdot\nu=0\}$ for all $\nu \in \mathbb{S}^{d-1}$, \EEE 
and  
\begin{gather}\label{Besicovitch jump}  \frac{{\rm d} \mathcal F(u, U,\cdot)}{{\rm d} \Hd\mres J_u} (x) =
\lim_{\rho \to 0^+}\frac{\mathcal F(u, U,Q_{\nu_u(x)}(x,\rho))}{\rho^{d-1}} =
\lim_{\rho \to 0^+}\frac{\m(u, U,Q_{\nu_u(x)}(x,\rho))}{\rho^{d-1}},\\
\label{Lebesgue G jump}
\lim_{\rho\to 0^+}\frac{1}{\rho^{d-1}}\int_{Q_{\nu_u(x)}(x,\rho)}|U(y)|^p\,{\rm d}y=0,
\end{gather}
\BBB and such that the statement of  Lemma~\ref{lemma: blow up jump points} holds. \EEE
Note that these conditions are satisfied by  $\Hd$-a.e.\ point of $J_u$ thanks to Proposition~\ref{Prop:fine prop GBV},  \BBB  \cite[Theorem~2.56]{AFP}, \EEE Besicovitch and Lebesgue Derivation Theorems, and Theorem~\ref{thm: minimum=F}, \BBB where for \eqref{blow up jump 50} we can argue as in \eqref{blow up jump 4}--\eqref{blow up jump 5} below. \EEE Thus, it is sufficient to show \eqref{psiproof} for such a point $x$.

 Recalling the definition \eqref{def Psi} of the surface density $f_{\rm surf}$, equality \eqref{Besicovitch jump} implies that \eqref{psiproof} is equivalent to 
\begin{align*}
	 \lim_{\rho \to 0^+}\frac{\m(u,U, Q_\nu(x, \rho))}{\mathcal \rho^{d-1}}=\lim_{\rho \to 0^+}
	\frac{\m(v_{x},0, Q _\nu(x, \rho))}
	{\mathcal \rho^{d-1}},
\end{align*}
 where, for brevity,  we set  $v_{x}\coloneq v_{x,u^+(x),u^-(x),\nu_u(x)}$ and $\nu=\nu_u(x)$. 
 This equality is proved in two steps, dealing with the two inequalities separately.   To this end, we apply \EEE  Lemma~\ref{lemma: blow up jump points} to  get  $u_\rho\in {\rm  BV}(Q_\nu(x,\rho);\Rk)$   which  satisfy  
\begin{align}\label{blow up jump 1 applied}
 & \lim_{\rho\to 0^+}\frac{1}{\rho^{d}}\Lb^d(\{u_\rho\neq u\}\cap Q_\nu(x,\rho))=0,  
 \\ \label{blow up jump 2 applied}&  \lim_{\rho \to 0^+}   \frac{1}{\rho^{d}} \int_{Q_\nu(x,\rho)}  |u_\rho - v_{x} |\, {\rm d}y = 0,\\\label{blow up jump 3 applied}
 & \lim_{\rho \to 0^+} \   \frac{1}{\rho^{d-1}}\,\mathcal{H}^{d-1}(J_{u_\rho}\cap E_{\rho,x})\leq  \Hd(\Pi_x^\nu \cap E) \quad \text{for all Borel sets } E \subset Q_\nu(x,1),\\\label{blow up jump 4 applied}
 &    \lim_{\rho \to 0^+} \ \frac{1}{\rho^{d-1}} \int_{Q_\nu(x,\rho)}  | \nabla u_\rho | \, \mathrm{d}y = 0\quad \text{and}\quad \lim_{\rho \to 0^+} \ \frac{1}{\rho^{d-1}}|D^cu_\rho|(Q_\nu(x,\rho))=0.
\end{align}
 \EEE

\medskip

\noindent \emph{Step 1.} ("$\leq$")
\MMM In this step we prove that \EEE  
\begin{equation}\label{claim step 1 jump}
    \lim_{\rho \to 0^+}\frac{\m(u,U, Q_\nu(x, \rho))}{\mathcal \rho^{d-1}}\leq \lim_{\rho \to 0^+}
	\frac{\m(v_{x},0, Q _\nu(x, \rho))}
	{\mathcal \rho^{d-1}}.
\end{equation}
To this end, we fix $\theta\in (0,1/4)$, set  $Q^{\theta,\rho}_\nu\coloneq Q_\nu(x,(1-3\theta)\rho)$, and  consider $(v_\rho,V_\rho)\in \C^p_\star(v_x,0,Q^{\theta,\rho}_\nu)$ such that 
\begin{equation}\label{quasi minimiser step 1 jump}
\cF( v_\rho, V_\rho, Q^{\theta,\rho}_\nu)\leq  \m(v_x,0, Q^{\theta,\rho}_\nu)+\rho^{d}.  
\end{equation}
\MMM We extend $v_\rho$   to the whole cube $Q_\nu(x,\rho)$  by setting $v_\rho=v_x$ on $Q_\nu(x,\rho)\setminus Q^{\theta,\rho}_\nu$.  Moreover, we set \EEE   
\begin{equation}\label{def Wrho 1 jump }\hspace{-0.4 cm} W_\rho(y)\coloneq  
\begin{cases}
 V_\rho(y) &\text{ in }Q^{\theta,\rho}_\nu,\\
\displaystyle \frac{1}{\mathcal L^d(Q_\nu(x,\rho)\setminus Q^{\theta,\rho}_\nu)}\int_{Q_\nu(x,\rho)} U(y)\, {\rm d} y &\text{ in }Q_\nu(x,\rho)\setminus Q^{\theta,\rho}_\nu.
\end{cases}
\end{equation}
Since $(v_\rho,V_\rho)\in\C^p_\star(v_x,0,Q^{\theta,\rho}_\nu)$, we have
\begin{equation}\label{wrho average jump}
\int_{Q_\nu(x,\rho)}W_\rho(y)\,{\rm d}y=\int_{Q_\nu(x,\rho)}U(y)\,{\rm d}y.
\end{equation}
In light of \eqref{blow up jump 1 applied}, Fubini's Theorem implies that for every $\rho>0$ there exists  $s_\rho\in ((1-\theta) \rho,\rho)$ such that 
\begin{equation}\label{good shell}
    \lim_{\rho\to 0^+}\frac{1}{\rho^{d-1}}\Hd(\partial Q_\nu(x,s_\rho)\cap \{{u}_\rho\neq u\})=0.
\end{equation}
We also introduce   $z_\rho$  defined by    
\begin{equation}\label{blowup vrho jump} 
    z_{\rho}(y)\coloneq 
    \begin{cases}u_\rho(y) &  \text{ if $y\in Q_\nu(x,s_\rho)$,}\\
    u(y) &\text{ if $y\in Q_\nu(x,\rho)\setminus Q_\nu(x,s_\rho)$,}
    \end{cases}
\end{equation}
 which by Lemma~\ref{lemma: Extension of GBV} belong to $\GBVs(Q_\nu(x,\rho);\Rk)$. Note \EEE
 that \BBB by \eqref{blow up jump 20}--\eqref{blow up jump 50},  \EEE  \eqref{blow up jump 2 applied}--\eqref{blow up jump 4 applied}, and \eqref{good shell} it holds that 
\begin{align}
    \label{blowup jump 3 applied}  &  \lim_{\rho\to 0^+}\frac{1}{\rho^{d}}\int_{Q_\nu(x,(1-\theta)\rho)}|{z}_\rho(y)-v_x|\,{\rm d}y=0,\\ \label{blowup jump 4 applied}
    &\lim_{\rho\to 0^+}\frac{1}{\rho^{d-1}}\mathcal{V}(z_\rho,  E_{\rho, x} \EEE \cap J_{z_\rho})\leq \MMM k \EEE \Hd(\Pi^\nu_x\cap E) \quad \text{ for all Borel sets $E\subset Q_\nu(x,1)$},\\\label{blowup jump 5 applied}
    & \lim_{\rho\to 0^+}\frac{1}{\rho^{d-1}} 
    \int_{Q_\nu(x,\rho)}| \nabla {z
    }_\rho(y)|\,{\rm d}y=0 \quad \text{ and } \lim_{\rho\to 0^+}\frac{1}{\rho^{d-1}}|D^c{z}_\rho|(Q_\nu(x,\rho))=0.
\end{align}
We set 
\begin{equation}\label{def sets O O' O'' jumps}
    O'\coloneq Q_\nu(x,{1-2\theta}),\quad O\coloneq Q_\nu(x,1-\theta),\quad \text{ and }O''\coloneq Q_\nu(x,1)\setminus \overline{ \MMM Q_\nu(x,1-3\theta)\EEE }.
\end{equation}
Observe that $v_\rho\in L^1(O_{\rho,x}\setminus O'_{\rho,x};\Rk)$ since $v_\rho=v_x$  $\Lb^d$-a.e.\ in $Q_\nu(x,\rho)\setminus Q^{\theta,\rho}_\nu$. Moreover, from definition \eqref{blowup vrho jump} and the fact that $s_\rho\geq (1-\theta)\rho$, it follows that $z_\rho\in L^1(O_{\rho,x}\setminus O'_{\rho,x};\Rk)$. Hence,
for fixed $\eta>0$, we may apply Lemma~\ref{lemma: Fundamental Estimate} to $v_\rho$ (in place of $u$) and to $z_\rho$ in place of $v$ and with $O',O'',O$  \MMM as in \eqref{def sets O O' O'' jumps}. \EEE We obtain a constant $M>0$ depending on $\theta$ but independent of $\rho$,  and functions $w_\rho\in {\rm GBV}_\star(Q_{ \nu\EEE}(x,\rho);\Rk)$ such that 
\begin{align}\label{fundamental applied 1 jump}
 \cF(w_\rho,W_\rho,Q(x,\rho))&\leq(1+\eta)\big(\cF(v_\rho,W_\rho,O_{\rho,x})+\cF(z_\rho,W_\rho, O''_{\rho,x})\big)
     +\frac{M}{\rho}\| z_\rho-v_\rho\|_{L^1(O_{\rho,x}\setminus O'_{\rho,x})}+\eta\rho^{d},\notag\\&=(1+\eta)\big(\cF(v_\rho,W_\rho,O_{\rho,x})+\cF(z_\rho,W_\rho, O''_{\rho,x})\big)     +\frac{M}{\rho}\|z_\rho-  v_x\|_{L^1(O_{\rho,x}\setminus O'_{\rho,x})}+\eta\rho^{d},
\end{align}
where in the equality we have used that $v_\rho=v_x$ $\Lb^d$-a.e.\ in $Q_\nu(x,\rho)\setminus Q^{\theta,\rho}_\nu$.
Moreover, the same lemma and \eqref{blowup vrho jump} gives 
\begin{equation}\label{wrho uguale cosa jump}
   w_\rho=v_\rho \quad \text{ $\Lb^d$-a.e.\ in $O_{\rho,x}'$} \quad \text{ and }\quad w_\rho=u \quad \text{ $\Lb^d$-a.e.\ in $Q_\nu(x,\rho)\setminus Q_\nu(x,s_\rho)$}.
\end{equation}
From the second equality in \eqref{wrho uguale cosa jump} and \eqref{wrho average jump} we deduce that $(w_\rho,W_\rho)\in \C^p_\star(u,U,Q_\nu(x,\rho))$. Therefore,  by  \eqref{fundamental applied 1 jump} we infer 
\begin{align}\label{fundamental applied with minimum jump} 
\hspace{-0.15 cm} \m(u,U,Q_\nu(x,\rho))&\leq(1+\eta)\big(\cF(v_\rho,W_\rho,O_{\rho,x})+\cF(z_\rho,W_\rho, O''_{\rho,x})\big)  +\frac{M}{\rho}\| z_\rho- v_x\|_{L^1(O_{\rho,x}\setminus O'_{\rho,x})}+\eta\rho^{d}.
\end{align}

We estimate each term \MMM on \EEE  the right-hand side of \eqref{fundamental applied with minimum jump}.   First, we note that \EEE  \eqref{blowup jump 3 applied} implies
\begin{equation}\label{L1 term jump}
    \lim_{\rho\to 0^+}\frac{M}{\rho^d}\| z_\rho- v_x\|_{L^1(O_{\rho,x}\setminus O'_{\rho,x})}=0.
\end{equation}
We  use Lemma~\ref{lemma: smaller larger}, with $O_1=Q^{\theta,\rho}_\nu$ and $O_2=O_{\rho,x}$, locality (H2),  and \EEE \eqref{def Wrho 1 jump } to get 
\begin{align} \notag
\cF(v_\rho,W_\rho,O_{\rho,x})&\leq \cF(v_\rho,\MMM V_\rho, \EEE Q^{\theta,\rho}_\nu)+ \beta \mathcal{V}(v_\rho,Q_\nu(x,\rho)\setminus Q^{\theta,\rho}_\nu)\\\label{estimate on Orhox}
    & \quad +\beta\|W_\rho\|_{L^p(Q_\nu(x,\rho)\setminus Q^{\theta,\rho}_\nu)}^p+\beta(1-(1-3\theta)^d)\rho^d.
\end{align}
Recalling \eqref{quasi minimiser step 1 jump} and that 
$v_\rho=v_x$ on $ Q_\nu(x,\rho)\setminus Q^{\theta,\rho}_\nu$, we have 
\begin{gather}
\cF( v_\rho, V_\rho, Q^{\theta,\rho}_\nu)\leq  \m(v_x,0, Q^{\theta,\rho}_\nu)+\rho^{d}  \leq \beta\rho^d+\beta k\rho^{d-1} +\rho^{d}, \EEE
\label{stima zrho on VXXX} \\
\label{stima zrho on V}
\mathcal{V}(v_\rho,Q_\nu(x,\rho)\setminus Q^{\theta,\rho}_\nu)\leq (1-(1-3\theta)^{d-1})k\rho^{d-1}.
\end{gather}
 Here, in the second inequality of \eqref{stima zrho on VXXX} we used that  $(v_x,0)$ is a competitor for $\m(v_x,0,Q^{\theta,\rho}_\nu)$ and thus $ \m(v_x,0,Q^{\theta,\rho}_\nu)\leq \beta\rho^d+\beta k\rho^{d-1}$ by the upper bound in   (H4).    \EEE 
As for the remaining term in \eqref{estimate on Orhox}, using  H\"older's inequality, \MMM we get \EEE 
\begin{align}\notag
\int_{Q_\nu(x,\rho)\setminus Q^{\theta,\rho}_\nu} 
| W_\rho  |^p \, {\rm d} y &\leq 
\frac{1}{\rho ^{d(p-1)}(1-(1-3\theta)^d)^{p-1}}\Big|\int_{Q_\nu(x,\rho)} U(y)\, {\rm d}y \Big|^p \\
&\leq \frac{\rho^{d(p-1)}}{\rho ^{d(p-1)}(1-(1-3\theta)^d)^{p-1}} \|U\|^p_{L^p(Q_\nu(x,\rho))} \le  \MMM C \EEE \|U\|^p_{L^p(Q_\nu(x,\rho))} \label{mest4}
\end{align}
\MMM for some universal $C>0$, where we used $\theta <1/4$. \EEE 
Similarly, we may estimate 
\begin{equation} \label{stima rurho jump}
    \cF(z_\rho,W_\rho, O''_{\rho,x})\leq  \beta \mathcal{V}(z_\rho,O''_{\rho,x}) +\beta\|W_\rho\|_{L^p(Q_\nu(x,\rho)\setminus Q^{\theta,\rho}_\nu)}^p+\beta(1-(1-3\theta)^d)\rho^d
\end{equation}
and observe that 
\begin{equation}\label{ultima step 1 jump}
    \mathcal{V}( z_\rho,O''_{\rho,x})\leq k\Big(\int_{O''_{\rho,x}}|\nabla z_\rho|\,{\rm d}x+|D^cz_\rho|(O''_{\rho,x})+\Hd(J_{z_\rho}\cap O''_{\rho,x})\Big).
\end{equation}
Hence, using \eqref{blowup jump 4 applied}    with $E =O'' $ and \eqref{blowup jump 5 applied}, we get
\begin{equation}\label{Vrd-1 strip}
     \limsup_{\rho \to 0^+} \frac{1}{\rho^{d-1}}\mathcal{V}(
    z_\rho,O''_{\rho,x})\leq (1-(1-3\theta)^{d-1})k.
\end{equation}
Putting together   \eqref{fundamental applied with minimum jump}--\eqref{Vrd-1 strip}, dividing by $\rho^{d-1}$, letting $\rho\to 0^+$, and using \eqref{Lebesgue G jump}  we get 
\begin{align*}
\limsup_{\rho\to 0^+} \frac{\m(u,U,Q_\nu(x,\rho))}{\rho^{d-1}}&\leq (1-3\theta)^{d-1}\limsup_{\rho' \to 0^+}\Big( \frac{\m(v_x,0,Q_\nu(x,\rho'))}{(\rho')^{d-1}}\Big)+\eta\beta k\\ &\quad+ (1+\eta)\Big(\beta(1-(1-3\theta)^{d-1})2k\Big),
\end{align*}
where we have set $\rho'\coloneq (1-3\theta)\rho$.
Letting $\theta \to 0^+$ and $\eta\to 0^+$, we obtain \eqref{claim step 1 jump}.

\medskip

\noindent \emph{Step 2.} ($\geq$)  As in the previous  proof, the other inequality can be proven along similar lines. We therefore omit the full proof and indicate  only the essential changes. We follow Step 1 by interchanging the roles of $(u,U)$ and $(v_{x},0)$. The   Fubini-type argument in \eqref{good shell} is performed for some $s_\rho\in ((1-4\theta)\rho,(1-3\theta)\rho)$, we choose a competitor $(v_\rho,V_\rho)\in \C^p_\star(u,U,\MMM Q_{\nu\EEE}(x,s_\rho)) \EEE $ satisfying 
\begin{equation*}
    \cF(v_\rho, V_\rho,Q_{\nu\EEE}(x,s_\rho))\leq \m(u,U,Q_{\nu\EEE}(x,s_\rho))+\rho^{d},
\end{equation*}
\MMM and we extend    $v_\rho$ to the whole $Q_{\nu\EEE}(x,\rho)$ by setting  $v_\rho= u_\rho$  outside of $Q_{\nu\EEE}(x,s_\rho)$, as well as \EEE  
$$ \MMM W_\rho \EEE  =  - \frac{1}{\mathcal L^d(Q_{\nu\EEE}(x,\rho)\setminus Q_{\nu\EEE}(x,s_\rho))}   \int_{Q_{\nu\EEE}(x,s_\rho)} U(y)\,{\rm d}y$$ outside of $Q_{\nu\EEE}(x,s_\rho)$. 
\end{proof}

\MMM
 \begin{proof}[Proof of Theorem \ref{thm: integral star}]
The proof follows by combining Propositions \ref{prop: repre AC} and \ref{prop:repre surf}.
\end{proof}

\EEE

\begin{remark}
In the recent papers \cite{DalToaConvex}  and \cite{dal2025homogenization}, the integral representation of a family of functionals related to $\F^p_\star$ is addressed. In particular, in the latter  work, the authors study a class  $\E_{\rm sc}$ of functionals $E\colon L^0(\Omega;\Rk)\times\mathcal{B}(\Omega)\to [0,+\infty]$ whose domain is ${\rm GBV}_\star(\Omega;\Rk)$ and that satisfy, among others, the following property: there exists a constant $C_{\rm Tr}>0$ such that  for every $N\in\N$, $u\in L^0(\Omega;\Rk)$, $B\in\mathcal{B}(\Omega)$, $w\in W^{1,1}(\Omega;\Rk)$, and $R>0$ it holds (see \eqref{def PhiR})
    \begin{align*}
        \nonumber \frac{1}{N}\sum_{i=1}^NE(w+\Phi_{R_i}\circ (u-w),B)& \leq E(u,B) + C_{\rm Tr}\int_{B^R_{u,w}}|\nabla w|\, {\rm d}x+C_{\rm Tr}\Lb^d(B^R_{u,w}) \\
      & \quad +\frac{C_{\rm Tr}}{N}\Big(E(u,B) + \Lb^d(B) + \int_B|\nabla w|\, {\rm d}x \Big),
        \end{align*}
       where  $R_i\coloneq \tau^{i-1}R$ \BBB (see \eqref{Properties of Phi}, for some $\tau$ large enough)   and $B^R_{u,w}\coloneq \{x\in B\colon |u(x)-w(x)|\geq R\}$. \EEE
  Relying on truncation arguments based on the previous property\EEE, it is  shown \EEE  in  \cite[Theorem~5.16]{dal2025homogenization} that every functional $E\in\E_{\rm sc}$ satisfies a property similar to the one described  in \EEE  Theorem~\ref{thm: integral star} for functionals in $\F^p_\star$, in the sense that, for every bounded open set $O\subset \Omega$,  the bulk and surface parts of  $E(\cdot,O)$ can be represented as integral functionals on ${\rm GBV}_\star(\Omega;\Rk)$. 

\BBB 
The methods used in this section show that  the result \cite[Theorem~5.16]{dal2025homogenization} can actually be derived \emph{without} the above truncation property. 
\end{remark}

\section{Relaxation of bulk and surface energies:  Proof of Theorems \ref{thm: relaxation partial}, \ref{thm:cell formulas},   \ref{thm:relaxation full}}\label{section Relaxation}

This section is devoted to the proof of the Relaxation Theorems \ref{thm: relaxation partial}, \ref{thm:cell formulas}, and  \ref{thm:relaxation full}, concerning the integral representation of the relaxed functional $I^p_\star$ defined by \eqref{def Istar localised} and \eqref{def Borel relax}.

\subsection{Proof of Theorem~\ref{thm: relaxation partial}}
We  start \EEE  the section \MMM by \EEE presenting a useful consequence of the Approximation Theorem~\ref{thm:approximation SDstar}, which will be exploited in the proof of Theorem~\ref{thm: relaxation partial}. Recall the structural constants  $c_\Psi$, $C_\Psi$, $c_\gamma$, and $C_\gamma$ in the assumptions on $\Psi$ and $\gamma$, as well as  $\overline{C}_\Psi$ in \eqref{bound above W}.

\begin{lemma}\label{lemma: upper bound}
    Let $(g,G)\in {\rm SD}^p_\star(\Omega;\Rk)$ and \MMM let $O\in\Op(\Omega)$. \EEE  Assume that $\Psi$ and $\gamma$ satisfy {\rm \ref{(W1)_p}--\ref{W4}}, {\rm\ref{(gamma1)}}, and {\rm \ref{(gamma2)}}. Then 
    \begin{equation*}
    \alpha\big(\mathcal{V}(g,O)+\|G\|^p_{L^p(O)}-\Lb^{d}(O)\big)\leq  I^p_{\star}(g,G,O)\leq \beta\big(\mathcal{V}(g,O)+\|G\|^p_{L^p(O)}+\Lb^d(O)\big)
    \end{equation*}
    for two suitable positive constants $\alpha,\beta$,  depending only on $c_\Psi$, $C_\Psi$, $\overline{C}_\Psi$, $c_\gamma$,  $C_\gamma$, and $p$.  
\end{lemma}
\begin{proof} 
Using  \ref{W4} we find  $  \Psi(x,A) \ge  \frac{c_\Psi}{2} (|A|^{p} + |A| ) - C_\Psi'$  for $\Lb^d$-a.e.\ $x \in \Omega$ and all $A \in \Rkd$ 
 for some $C_\Psi'>0$ sufficiently large.  Then, \EEE the lower bound follows by  \ref{(gamma2)}, together with the lower semicontinuity of $\mathcal{V}$ (see Remark~\ref{remark: lowersemicontinuity}) and of $\|G\|^p_{L^p(O)}$ under weak convergence.   
    
    To prove the upper bound, we use the Approximation Theorem~\ref{thm:approximation SDstar} to obtain a sequence $\{u_n\}_{n}\subset {\rm SBV}(O;\Rk)$ with $\nabla u_n=G$ $\Lb^d$-a.e.\ in $O$ and such that \begin{equation*}
        \limsup_{n\to+\infty}\mathcal{V}(u_n,O)\leq \MMM C\big( \mathcal{V}(g,O) + \Vert G \Vert_{L^1(O)} \big). \EEE
    \end{equation*}
    Thanks to Remark~\ref{remark: Improved upper bound},  the previous inequality,  and  \ref{(gamma2)}, \EEE  we then have 
    \begin{align*}
        I^p_\star(g,G,O)\leq \liminf_{n\to+\infty} \mathcal{E}(u_n,O)&\leq \beta \big(\liminf_{n\to+\infty}\mathcal{V}(u_n,O)+\|G\|^p_{L^p(O)}+\Lb^d(O)\big)\\
         &\leq \MMM C \beta \big(\mathcal{V}(g,O)+ \|G\|_{L^1\EEE(O)} + \|G\|^p_{L^p(O)}+\Lb^d(O)\big). \EEE
    \end{align*}
\MMM    Using that $|A| \le 1 +|A|^p$ for all $A \in \R^{k\times d}$, the proof is concluded. \EEE  
\end{proof}

We are now ready to prove Theorem~\ref{thm: relaxation partial}.
\medskip
 
\noindent\emph{Proof of Theorem \ref{thm: relaxation partial}.} 
   Thanks to Theorem~\ref{thm: integral star}, to conclude the proof, it is enough to show  that \EEE    $I^p_\star\in\F^p_\star$.

Let us fix $(g,G)\in {\rm SD}^p_\star(\Omega;\Rk)$. 
Given \MMM $O_1,O_2,O_3 \in \mathcal{O}(\Omega)\EEE$   \EEE with 
$O_1 \subset\subset  O_2 \subset O_3\subset \Omega$,  we claim that 
\begin{equation}\label{weaksub}I^p_\star(g,G,O_3) \leq I^p_\star(g,G,O_2) + I^p_\star(g,G,O_3\setminus \overline {O_1}).
\end{equation}
This inequality can be obtained with an argument similar to, \MMM for instance,   \cite[Theorem~2.8]{FriMatZap}, \EEE
but for the sake of \MMM completeness \EEE we present here \MMM the \EEE  full proof.
 Without loss of generality, we prove this inequality only in the case $p=1$ as the case $p>1$ follows by similar, but easier, arguments. Let $\lbrace u_n \rbrace_n \subset  {\rm SBV}(O_2;\Rk)$ and 
$\lbrace v_n  \rbrace_n \subset {\rm SBV}(O_3\setminus \overline{O_1};\Rk)$ be two sequences satisfying the following properties: $u_n\SDtostar (g,G)$ in $O_2$, $v_n\SDtostar(g,G)$ \EEE in $O_3\setminus \overline{O_1}$,  and 
\begin{equation}\label{recoveries weak}
I^p_\star(g,G,O_2)=\lim_{n\to +\infty}\mathcal{E}(u_n,O_2)\quad \text{ and }\quad I^p_\star(g,G,O_3\setminus \overline{O_1}) = \lim_{n\to +\infty}\mathcal{E}(v_n,O_3\setminus \overline{O_1}).
\end{equation}
Note also that  
\begin{equation}  \label{measure convergence}
 u_n - v_n\to \EEE 0 \text{ in }\; 
L^0( O_2 \setminus \overline{O_1}; {\mathbb{R}}^k) \text{ as $n\to+\infty$}.
\end{equation}
Let us fix $\delta>0$ and consider $O_{\delta}\coloneq  \{ x \in O_2 \colon   \, \mbox{dist}(x, O_1) < \delta\}$.  We let ${\rm d}(x)\coloneq  \mbox{dist}(x, O_1)$ be the distance function from $O_1$, which we recall   is Lipschitz continuous. Therefore,  by the Area Formula  \cite[(2.47)]{AFP} \EEE  we have  
\begin{equation*}
\int_{O_{\delta}\setminus \overline{O_1}} |u_n(x) - v_n(x)|\land 1 \,  J{\rm d}(x)  \, {\rm d} x =
\int_{0}^{\delta}\Big( \int_{{\rm d}^{-1}(y)} |u_n(x) - v_n(x)|\land 1 \, 
{\rm d}\Hd(x)\Big)\, {\rm d} y,
\end{equation*}
 where we let $J{\rm d}$ denote the Jacobian of  ${\rm d}$. In light of \eqref{measure convergence} and of the fact that $J{\rm d}$ is bounded, we may apply Fatou's  Lemma \EEE to obtain,  for $\Lb^1$-a.e.\  $\rho \in [0, \delta]$,   
\begin{equation}  \label{shellconvergence}
\liminf_{n\rightarrow +\infty} 
\int_{{\rm d}^{-1}(\rho)} |u_n - v_n|\land 1\, {\rm d}\mathcal{H}^{d-1}(x) 
= \liminf_{n\rightarrow +\infty} \int_{\partial O_{\rho}}
|u_n - v_n |\land 1\, {\rm d}\mathcal{H}^{d-1}(x) = 0.
\end{equation}
Let us fix $\delta_0 \in [0, \delta]$  such that \eqref{shellconvergence} holds   and $\Hd(\partial O_{\delta_0})<+\infty$.  The latter can  be assumed \EEE  by the Coarea Formula. By passing to a non-relabelled subsequence of $\{u_n-v_n\}_n$, we may also assume that the liminf in \eqref{shellconvergence} is actually a limit. Since $O_{\delta_0}$ is a set with
locally Lipschitz boundary (being the level set of a Lipschitz function),   we can consider  $u_n, v_n$ 
on $\partial O_{ \delta_0}$ in the sense of traces of ${\rm GBV_\star}(\Omega;\Rk)$-functions, and define 
\begin{align*}
w_n\coloneq  
\begin{cases}
u_n & \text{ in}\; \overline{O_{\delta_0}}, \\ 
v_n & \text{ in}\; O_3\setminus \overline{O_{\delta_0}}.
\end{cases}
\end{align*}
It follows from Lemma~\ref{lemma: Extension of GBV}  that $\{w_n\}\subset {\rm GBV}_\star(O_3;\Rk)$, 
 $|[w_n]|=|u_n-v_n|$ $\Hd$-a.e.\ on $\partial O_{\delta_0}$, and 
  $w_n\SDtostar (g,G)$ \EEE in $O_3$. Hence, by  {\rm \ref{(gamma2)}},
 \eqref{shellconvergence}, and \eqref{recoveries weak} we get  
\begin{align*}
I^p_\star(g,G,O_3) &\leq  \liminf_{n\to +\infty} 
\Big(\int_{O_3}\Psi(x, \nabla w_n) \, {\rm d} x +
\int_{J_{w_{n}}\cap O_3} \gamma(x,[w_n],\nu_{w_n})\, {\rm d}\Hd\Big) \\
&\leq  \liminf_{n\to +\infty} 
\Big(_{O_2}\Psi(x, \nabla  u_n) \, {\rm d} x +
\int_{J_{u_{n}}\cap O_2} \gamma(x,[u_n],\nu_{u_n})\, {\rm d}\Hd \\
&  \hspace{1cm} + \int_{O_3\setminus \overline{O_1}}\Psi(x, \nabla  v_n) \, {\rm d}x  + \int_{J_{v_{n}}\cap (O_3 \setminus \overline{O_1})}
\gamma(x,[v_n],\nu_{v_n})\, {\rm d}\Hd \\
&   \hspace{1cm} + \int_{J_{w_n} \cap \partial O_{\delta_0}} 
 C_\gamma    |u_n- v_n| \land 1\EEE \, {\rm d}\Hd \Big) \\
& =  I^p_\star(g,G,O_2) + I^p_\star(g,G,O_3 \setminus \overline{O_1}).
\end{align*}
 This concludes the proof of \eqref{weaksub}.  
 
 Next,   one can check that \EEE $O\mapsto I^p_\star(g,G,O)$ coincides with the restriction to $\Op(\Omega)$ of a bounded Radon measure on $\Omega$.  We omit the exact argument which can be found \EEE in \cite[Proposition~2.22,~Step 1]{ChoksiFonseca}, using Lemma~\ref{lemma: upper bound}  and \eqref{weaksub} \EEE  \MMM  in place of \EEE \cite[Lemma~2.18,  Lemma~2.21]{ChoksiFonseca}.     This yields (H1)  
 and proves that definition \eqref{def Borel relax} is well-posed (see the comment after \eqref{def Borel relax})\EEE.

 Moreover, by inner regularity on open sets, we have that $I^p_\star$  satisfies the upper bound \eqref{upper bound} with $B$ replaced by $O$, for every $O\in\Op(\Omega)$. Combining this observation with \eqref{def Borel relax} \MMM and Lemma~\ref{lemma: upper bound} \EEE we get that $I^p_\star$ satisfies \ref{hyp:H4} as well  (up to a  shift\EEE).  

To prove that $I^p_\star$ satisfies the lower semicontinuity property \ref{hyp:H3}, it is enough to repeat \MMM exactly the \EEE  arguments used in \cite[Proposition~5.1]{ChoksiFonseca}. Finally, locality property \ref{hyp:H2} is a simple consequence of lower semicontinuity property \ref{hyp:H3} as observed in \cite[Remark~2.1]{BFM1998}.   This concludes the proof of    $I^p_\star\in \F^p_\star$ and thus of the theorem.
\qed
\medskip

\medskip

\subsection{Proof of Theorem \ref{thm:cell formulas}}  
 Throughout the rest of the section, we let $\Phi\in C^\infty_c(\Rk;\Rk)$ be a function satisfying \eqref{Properties of Phi}, for a fixed $\tau\geq \max \{3,\beta_\gamma+1\}$, where $\beta_\gamma$ is the  constant appearing in \ref{(gamma5)}. For $R>0$, we let $\Phi_R$ be the function associated to $\Phi$ via \eqref{def PhiR}.    \EEE
The next truncation result will be   useful  for \EEE proving Theorem~\ref{thm:cell formulas} as    it  allows  to truncate    competitors appearing in suitable minimisation problems. \EEE    It will also be crucial in the proof of Lemma~\ref{lemma:I truncated} below,    a key ingredient in the proof of  Theorem~\ref{thm:relaxation full}.\EEE
 
\begin{lemma}\label{lemma:CagnettiDetFree W}
Assume that $\Psi$ and $\gamma$ satisfy conditions  {\rm \ref{(W1)_p}--\ref{W4}}, { \rm    \ref{(gamma1)}, \ref{(gamma2)}\EEE}, and {\rm \ref{(gamma5)}}. 
Let $O\in\Op(\Omega)$. Then there exists a constant $C>0$, depending only on the structural constants   $c_\Psi,\bar{C}_\Psi, c_\gamma, C_\gamma,$   and $\beta_\gamma$, such that for every  $N\in\N$, $R>0$,   $u\in {\rm SBV}(O;\Rk)$, and $B\in \B(O)$ it holds
\begin{align}\label{claim truncation} 
     \frac{1}{N}\sum_{i=1}^N\mathcal{E}(\Phi_{R_i}\circ u, B)\leq \Big(1+\frac{C}{N}\Big)\mathcal{E}(u,B)+ \Lb^d(\MMM \{|u|\geq R\} \EEE )+\frac{C}{N}\Lb^d(B),
\end{align}
 where we have set $R_i\coloneq \tau^{i-1}R$.   If, in addition, $\gamma$ satisfies {\rm \ref{(gamma6)}} as well, \BBB then \EEE for every $x\in\Omega$ and $\rho>0$ the same property holds with the same constant $C$, for the functional  $\mathcal{E}^{x}_\rho$  defined  by 
    \begin{align}
\label{def Erhox}
    \mathcal{E}^x_\rho(u,B)&\coloneq \int_{B}\Psi(x+\rho y,\nabla u(y))\,{\rm d}y+\frac{1}{\rho}\int_{J_u\cap B}\gamma(x+\rho y,\rho[u](y),\nu_u(y))\,{\rm d}\Hd(y)
    \end{align}
    for every $u\in{\rm SBV}(O;\Rk)$ and $B\in \B(O)$.\EEE
\end{lemma}
\begin{proof}
 The proof of inequality \eqref{claim truncation} for $\mathcal{E}$ and $\mathcal{E}^x_\rho$ can be obtained arguing exactly as in the section of the proof of \cite[Proposition~3.14]{dal2025homogenization}   devoted to property (g) of \cite[Definition~3.8]{dal2025homogenization}. 
 \end{proof}

Later, given a sequence $\{u_n\}_n\subset {\rm SBV}(\Omega;\Rk)$ converging in $\SD^p_\star$, it \BBB will be essential \EEE to determine the limit of the truncated sequence $\Phi_R\circ u_n$ in a related topology. This is taken care of in the next lemma. We remark that the proof below cannot be reproduced in the case $p=1$, unless we replace in Definition \ref{def: convergences} weak$^*$-convergence  in $\mathcal{M}_b(\Omega;\Rkd)$ with weak $L^1$-convergence.  
\begin{lemma}\label{lemma:convergence truncations}
   Let $p>1$, let $O\in\Op(\Omega)$, let $(g,G)\in \SD^p_\star(\Omega;\Rk)$, let $ \{u_n\}_n\EEE\subset  {\rm SBV}(O;\Rk)$ with  $u_n \to_{SD_\star}  (g,G)$ \EEE in $O$, and let $R>0$.  Consider the functions  $g^R\coloneq \Phi_R\circ g$,  $u^R_n\coloneq \Phi_R\circ u_n$, and $G^R\coloneq \nabla \Phi_R(g)G$. Then   $\{u_n^R\}_n$ converges to $g^R$ strongly in $L^1(O;\Rk)$ and $\{\nabla u^R_n\}_n$ converges weakly in $L^p(O;\Rkd)$ to $G^R$.  In particular, if $g\in L^\infty(O;\Rk)$ and  $R>2\|g\|_{L^\infty(O;\Rk)}$, then $g^R=g$ and $G^R=G$. \EEE
\end{lemma}
\begin{proof}
   By the Chain Rule for ${\rm BV}$-functions \MMM  (see \cite[Proposition~3.96]{AFP}) \EEE we have \begin{equation}\label{chain rule}
    \nabla u^R_n =\nabla \Phi_{R}(u_n)\nabla u_n
\end{equation} 
for every $n\in\N$. 
We pass to a subsequence $\{u_{n_\ell}\}_\ell$  that converges pointwise a.e.\ to $g$ as $\ell\to+\infty$. As $\Phi_{R}$ is a $C^1$-function, from the pointwise convergence of $\{u_{n_\ell}\}_\ell$ and the Dominated Convergence Theorem it follows that $\{u^R_{n_\ell}\}_\ell$ converges to $g^R$ strongly in $L^1(O;\Rk)$ and 
$\{\nabla \Phi_{R}(u_{n_\ell})\}_\ell$ converges to $\nabla \Phi_{R}(g)$ strongly in $L^q(\Omega; \MMM \R^{k\times k}\EEE )$ for every $1\leq q<+\infty$ and pointwise a.e.\ in $O$.

 We now prove that $\{\nabla \Phi_{R}(u_{n_\ell})\nabla u_{n_\ell}\}_\ell$ converges to $G^R$ weakly in $L^p(O;\Rkd)$. To this end, let us fix   $\varphi\in L^{q}(O;\Rkd)$, where $q\coloneq p/(p-1)$ \EEE denotes the conjugate exponent of $p$. We have
\begin{align*}
    \int_{O}\varphi:(\nabla \Phi_{R}(u_{n_\ell})\nabla u_{n_\ell}-\nabla \Phi_{R}(g)G)\,{\rm d}x&= \int_{O}\varphi:(\nabla \Phi_{R}(g)\nabla u_{n_\ell}-\nabla \Phi_{R}(g)G)\,{\rm d}x\\&\quad +\int_{O}\varphi:(\nabla \Phi_{R}(u_{n_\ell})\nabla u_{n_\ell}-\nabla \Phi_{R}(g)\nabla u_{n_\ell}\big)\,{\rm d}x,
\end{align*}
where for $A,L\in\Rkd$ the symbol $A:L$ denotes the Frobenius scalar product between $A$ and $L$. By weak convergence in $L^p(O;\Rkd)$ of $\{\nabla u_{n_\ell}\}_\ell$ to $G$, the first term converges to $0$ as $\ell\to+\infty$. For the second term, we note instead that 
\begin{align*}
    \Big|\int_{O}\varphi:(\nabla \Phi_{R}(u_{n_\ell})\nabla u_{n_\ell}&-\nabla \Phi_{R}(g)\nabla u_{n_\ell}\big)\,{\rm d}x\Big| \leq\|\nabla u_{n_\ell}\|_{L^p(O)}\Big( \int_O|\varphi|^{q}|\nabla \Phi_{R}(u_{n_\ell})-\nabla \Phi_{R}(g)|^q\,{\rm d}x\Big)^{1/q}.
    \end{align*} \EEE
Since $\sup_\ell\|\nabla u_{n_\ell}\|_{L^p(O)}<+\infty$ by weak convergence,  $|\nabla \Phi_{R}|\leq 1$ by \eqref{Properties of PhiR}\EEE, and $\{\nabla \Phi_R(u_{n_\ell})\}_\ell$ converges to $\nabla\Phi_R(g)$ pointwise a.e.\ in $O$, an application of the Dominated Convergence Theorem shows  that \EEE 
$\{\nabla \Phi_{R}(u_{n_\ell})\nabla u_{n_\ell}\}_\ell$ converges to $G^R$ weakly in $L^p(O;\Rkd)$ as $\ell\to +\infty$. As the limit does not depend on the chosen subsequence $\{n_\ell\}_\ell$, by \eqref{chain rule} the whole sequences $\{u^R_n\}_n$ and $\{\nabla u^R_n\}_n$   converge to $g^R$  and $G^R$, respectively. \MMM This concludes the proof. \EEE  
\end{proof}

 We are now ready to prove Theorem \ref{thm:cell formulas}. 
\medskip

\noindent \emph{Proof of Theorem \ref{thm:cell formulas}.}  Throughout the proof, given $x\in \Omega$, we let $\mathcal{E}^x_0$ be the functional defined by 
\begin{equation*} 
\mathcal{E}^x_0(u,B)\coloneq \int_{B}\Psi(x,\nabla u(y))\,{\rm d}y+\int_{J_u\cap B}\BBB \gamma^0 \EEE(x,[u](y),\nu_u(y))\,{\rm d}\Hd(y)
\end{equation*}
for every $u\in{\rm SBV}(Q;\Rk)$ and $B\in \B(Q)$. We  split \EEE the proof  into \EEE two \BBB parts, \EEE addressing first   the statement concerning the bulk energy density.  We divide its proof into several steps.  

\EEE

\medskip

 \noindent \emph{Step 1.} (Bulk: Auxiliary cell formula)  \EEE
For every $x\in\Omega$ and $A,L\in\Rkd$ we set
\begin{align}\notag
\widetilde{H}_p(x,A,L)\coloneq \inf &\Big\{\liminf_{n\to+\infty}\mathcal{E}^x_0(u_n,Q)\colon \,
\{u_n\}_n\subset{\rm SBV}(Q;\Rk),\\ &  \label{def Htilde}\,\text{ }u_n\to \ell_{0,0,A} \text{ strongly in $L^1(Q;\Rk)$},\,\,  
\nabla u_n \rightharpoonup  L\text{ \MMM weakly \EEE  in $L^p(Q;\Rkd)$}\Big \}.
\end{align}
 \EEE
 We will prove that for $\Lb^d$-a.e.\ $x\in\Rd$  and for all $A,L\in \Rkd$ it holds 
\begin{equation}\label{step 1 cell}
  \widetilde{H}_p(x,A,L)\leq f_{\rm bulk}(x,A,L) \leq H_p(x,A,L).\EEE
\end{equation}
 Since by \ref{(W2)_p}, \eqref{bounds gammazero}, and Remark~\ref{remark:gamma zero uniform} the densities $\Psi(x,\cdot)$ and $\BBB \gamma^0 \EEE(x,\cdot)$ satisfy assumptions ${\rm (\mathscr{H}1)}_p$, ${\rm (\mathscr{H}2)}$, and ${\rm (\mathscr{H}4)}$  of \cite{ChoksiFonseca},  it follows from \cite[Proposition~3.1]{ChoksiFonseca} that $\widetilde{H}_p(x,A,L)={H}\EEE_p(x,A,L)$ for every $x\in\Omega$ and every $A,L\in\Rkd$.  Therefore, proving
 \eqref{step 1 cell} is sufficient to show that in \eqref{repr bulk relax} we can write $H_p$ in place of $f_{\rm bulk}$.\EEE

By definition of $f_{\rm bulk}$ \MMM (see also Remark~\ref{remark: invariance}) \EEE and by Theorem~\ref{thm: minimum=F} applied with $\cF=I^p_\star$ we have  
\begin{equation}\label{equalities Ipstar}
    f_{\rm bulk}(x,A,L)=\limsup_{\rho\to 0^+}\frac{\m(\ell_{x,0,A},L, Q(x,\rho))}{\rho^d}=\limsup_{\rho\to 0^+}\frac{I^p_\star(\ell_{x,0,A},L, Q(x,\rho))}{\rho^d}  <+\infty \EEE
\end{equation}
for $\Lb^d$-a.e.\ $x\in\Omega$ and for all $A,L\in\Rkd$.  Therefore, to prove
 \eqref{step 1 cell} it suffices to show that 
\begin{equation}\label{fgeq Hp}
 \widetilde{H}_p(x,A,L)\leq   f_{\rm bulk}(x,A,L)  \leq    {H}_p\EEE(x,A,L)
\end{equation}
for all $x\in\Omega$ such that  \eqref{equalities Ipstar} holds and for  all $A,L\in\Rkd$.

  \medskip 

\noindent \emph{Step 2.} (Bulk: $\widetilde{H}_p\leq f_{\rm bulk}$)  \EEE We let $x\in\Omega$ be such that \eqref{equalities Ipstar} holds, and let  $A,L\in\Rkd$.
For any $\rho>0$ we consider a sequence $\{u^\rho_n\}_n\subset {\rm SBV}(\Omega;\Rk)$ with   $u^\rho_n \to_{SD_\star}  (\ell_{x,0,A},L)$ \EEE in $Q(x,\rho)$, $\{u^\rho_n\}_n$ converging to \MMM $\ell_{x,0,A}$ \EEE   pointwise a.e.\ in $Q(x,\rho)$, 
and 
\begin{equation}\label{quasiminimal Hp}
\lim_{n\to+\infty} \mathcal{E}(u^\rho_n,Q(x,\rho))=  I^p_\star(\ell_{x,0,A},L,Q(x,\rho)).
\end{equation}
We introduce the functions $v^\rho_n\in {\rm SBV}(Q;\Rk)$ defined by $v_n^\rho \MMM (y) \EEE \coloneq \frac{1}{\rho}u^\rho_n(x+\rho y)$ \MMM for $y \in Q$, \EEE which  satisfy  
\begin{gather}
\label{changes 0}\text{$v_n^\rho\SDtostar (\ell_{0,0,A}, L)$ in $Q$ } \text{ and }  v_n^\rho\to \ell_{0,0,A} \text{ pointwise a.e.\   as $n\to+\infty$},\\
\label{change 1}
    \nabla v_n^\rho(y)=\nabla { u_n^\rho\EEE}(x+\rho y) \quad \text{for $\Lb^d$-a.e.\ $y\in Q$},\\\label{change 2}
    [v_n^\rho](y)=\frac{1}{\rho}[u_n^\rho](x+\rho y)\quad \text{for $\Hd$-a.e.\ $y\in J_{v^n_\rho}=(J_{u^\rho_n}-x)/\rho$}.
\end{gather}
 Using \eqref{change 1} and \eqref{change 2}, a change of variables then shows that for all  $n\in\N$ and $\rho>0$ we have
\begin{equation}\label{change variables bulk}
    \mathcal{E}(u^\rho_n,Q(x,\rho))=\rho^d\mathcal{E}_\rho^x(v_n^\rho,Q),
\end{equation}
where $\mathcal E^{\rho}_x$ is the functional given by \eqref{def Erhox}.

We set $R\coloneq 2\|\ell_{0,0,A}\|_{L^\infty(Q;\Rk)}\leq 2\sqrt{d}|A|$ and $R_i\coloneq \tau^{i-1}R$ for all $i\in\N$. Given $N\in\N$, we may apply Lemma~\ref{lemma:CagnettiDetFree W} to the functional $\mathcal{E}_x^\rho$  and to $v^\rho_n$. We find $i_n=i(n,N,\rho)\in \{1,\dots,N\}$ such that 
\begin{equation}\label{truncation Hp}
    \mathcal{E}_\rho^x(\Phi_{R_{i_n}}\circ v_n^\rho, Q)\leq \Big(1+
     \frac{C}{N}\Big)\mathcal{E}_\rho^x(v_n^\rho,Q)+\Lb^d(\{|v^\rho_n|\geq R\})+\frac{C}{N}\Lb^d(Q)
\end{equation}
for   $C>0$ independent of $n$, $\rho$, $N$, $A,$ and $L$. We remark that by \eqref{changes 0} we have 
$\limsup_{n\to+\infty}\chi_{\{|v_n^\rho|\geq R\}}(y)$ $\leq \chi_{\{|\ell_{0,0,A}|\geq R\}}(y)$ 
for $\Lb^d$-a.e.\ $y\in Q$, so that by Fatou's lemma
\begin{equation}\label{vanishing super level sets}
    \limsup_{n\to+\infty}\Lb^d\big(\{|v^\rho_n|\geq R\}\big)\leq \Lb^d\big(\{|\ell_{0,0,A}|\geq R\} \BBB \cap Q \EEE \big)=0 \quad \text{ for all $\rho>0$}.
\end{equation}
We set $w^\rho_n\coloneq \Phi_{R_{i_n}}\circ v_n^\rho\in {\rm SBV}(Q;\Rk)$. Then, \BBB in view of \EEE \eqref{changes 0}, Lemma~\ref{lemma:convergence truncations} gives 
\begin{gather*}
\text{$\{w^\rho_n\}_n$ converges to $\ell_{0,0,A}$ strongly in $L^1(Q;\Rk)$ as $n\to+\infty$},
\\
    \text{$\{\nabla w^\rho_n\}_n$ converges to $L$ weakly in $L^p(Q;\Rkd)$ as $n\to+\infty$,}
\end{gather*}
so that, for each $\rho>0$, the sequence $\{w^\rho_n\}_n$ is a competitor for the minimisation problem in \eqref{def Htilde}. Therefore,
\begin{equation}\label{ultimo step htilde}
    \widetilde{H}_p(x,A,L)\leq \liminf_{n\to+\infty} \mathcal{E}^x_0(w^\rho_n,Q)\quad \text{ for all $\rho>0$}.
\end{equation}

In order to use \eqref{change variables bulk} we still need to compare \EEE $\MMM \mathcal{E}^x_0 \EEE (w^\rho_n,Q)$ and $\mathcal{E}_\rho^x(w^\rho_n,Q)$. To this end, we begin by observing that, \BBB by the fact that \EEE $\vartheta$ is nondecreasing,  \ref{(gamma6)}, \MMM \eqref{def Erhox}, \EEE \eqref{truncation Hp},  and the fact that $\|w^\rho_n\|_{L^\infty(Q)}\leq \tau^NR$,  we obtain 
\begin{align}\notag 
    \Big|\int_{ J_{w^\rho_n}\cap Q}&\gamma^0\big(x+\rho y, [w^\rho_n](y),\nu_{w^\rho_n}(y)\big)\,{\rm d}\Hd(y)-\int_{\ J_{w^\rho_n}\cap Q}\frac{1}{\rho}\gamma\big(x+\rho y,\rho [w^\rho_n](y),\nu_{w^\rho_n}(y)\big)\,{\rm d}\Hd(y)\Big|\\\notag
    &\leq \vartheta\big(\tau^N2R\rho \big)\int_{ J_{w^\rho_n}\cap Q}\frac{1}{\rho}\gamma\big(x+\rho y,\rho [w^\rho_n](y),\nu_{w^\rho_n}(y)\big)\,{\rm d}\Hd(y)\\ 
    &\leq \vartheta\big(\tau^N2R\rho\big)\Big(\Big(1+\frac{C}{N}\Big)\mathcal{E}_\rho^x(v_n^\rho,Q)+\Lb^d(\{|v^\rho_n|\geq R\})+\frac{C}{N}\Lb^d(Q)\Big)
    \label{|-|leq}.
\end{align}
In light of \eqref{quasiminimal Hp}, \MMM \eqref{change variables bulk}, \EEE  and \eqref{vanishing super level sets},  we have 
\begin{align}\notag 
\limsup_{\rho\to 0^+}&\limsup_{n\to+\infty}\,\,\vartheta\big(\tau^N2R\rho\big)\Big(\Big(1+\frac{C}{N}\Big) \MMM \mathcal{E}_\rho^x(v_n^\rho,Q) \EEE +\Lb^d(\{|v^\rho_n|\geq R\})+\frac{C}{N}\Lb^d(Q)\Big)\\\label{chacanery}
&\leq \limsup_{\rho \to 0^+}\vartheta\big(\tau^N2R\rho\big)\Big(\Big(1+\frac{C}{N}\Big)\frac{I^{p}_\star(\ell_{x,0,A},L,Q(x,\rho))}{\rho^d}+\frac{C}{N}\Lb^d(Q)\Big)=0,
\end{align}
\EEE
where in the last equality we have used  that $\vartheta$ is a continuous function with $\vartheta(0)=0$, and that $x$ is such that \eqref{equalities Ipstar} holds. 
   Exploiting \MMM \eqref{uniform gamma zero}, \eqref{equalities Ipstar}, the fact that $\|w^\rho_n\|_{L^\infty(Q)}\leq \tau^NR$, and \ref{(gamma2)} \EEE    with a similar argument  we can prove that 
\begin{align}\notag
    \limsup_{\rho\to 0^+}\limsup_{n\to+\infty}\Big| \int_{ J_{w^\rho_n}\cap Q}\gamma^0\big(x+\rho y, &[w^\rho_n](y),\nu_{w^\rho_n}(y)\big)\,{\rm d}\Hd(y)\\
    &-\int_{J_{w^\rho_n}\cap Q }\gamma^0\big(x, [w^\rho_n](y),\nu_{w^\rho_n}(y)\big)\,{\rm d}\Hd(y)\Big|=0.
\end{align}
In a similar fashion, using \ref{Wcontinuous} we can \BBB also \EEE show that  
\begin{equation}\label{uniform x Hp}
 \limsup_{\rho\to 0^+}\limsup_{n\to+\infty}\Big| \int_{Q}\Psi\big(x+\rho y, \nabla w^\rho_n\big)\,{\rm d}y-\int_{Q}\Psi\big(x,\nabla w^\rho_n\big)\,{\rm d}y\Big|=0.
\end{equation}

We are now ready to conclude the proof of  \MMM Step 2. \EEE
Combining \eqref{equalities Ipstar}, \eqref{quasiminimal Hp}, and \eqref{change variables bulk}--\eqref{uniform x Hp}, we obtain 
\begin{align*}
   \widetilde{H}_p(x,A,L)&\leq  \limsup_{\rho\to 0^+}\limsup_{n\to+\infty}\mathcal{E}^x_0(w_n^\rho,Q)=\limsup_{\rho\to 0^+}\limsup_{n\to+\infty}\mathcal{E}^x_\rho(w_n^\rho,Q) \\
   &\leq\limsup_{\rho\to 0^+}\limsup_{n\to+\infty}  \Big(1+
     \frac{C}{N}\Big)\mathcal{E}_\rho^x(v_n^\rho,Q)+\frac{C}{N}\Lb^d(Q)\\
     &\leq \limsup_{\rho\to 0^+} \Big(1+
     \frac{C}{N}\Big)\frac{I^p_\star(\ell_{x,0,A},L,Q(x,\rho))}{\rho^d}+\frac{C}{N}\Lb^d(Q) =\Big(1+\frac{C}{N}\Big)\BBB f_{\rm bulk} \EEE (x,A,L)+\frac{C}{N}\Lb^d(Q).
\end{align*}
Letting $N\to +\infty$, we obtain the first inequality in  \eqref{fgeq Hp}.

 \medskip \noindent \emph{Step 3.} (Bulk: $f_{\rm bulk}\leq {H}_p$)    The second inequality \EEE in \eqref{fgeq Hp} follows by an  argument similar to the one presented in \emph{Step 2}.  Let $\eta>0$ and let $u\in \C^{\rm bulk}_{p,\star}(A,L)  \EEE $ (see \eqref{bulk competitors})    be such that 
\begin{equation}\label{quasiminimiser cell}
    \int_{Q}\Psi\big(x,\nabla u(y)\big)\,{\rm d}y+ \int_{J_u\cap Q}\gamma^0\big(x,[u](y),\nu_u(y)\big)\,{\rm d}\Hd(y)\leq  {H}_p\EEE(x,A,L)+\eta.
\end{equation}
For every $\rho>0$, we consider the function  $v^\rho \coloneq \rho u(\frac{\MMM \cdot \EEE -x}{\rho})\in {\rm SBV}(Q(x,\rho);\Rk)$ and observe that 
\begin{gather}\label{change 1 bis}
    \nabla v^\rho(z)=\nabla u(\tfrac{z-x}{\rho}) \quad \text{for $\Lb^d$-a.e.\ $z\in Q(x,\rho)$},\\\label{change 2 bis}
    [v^\rho](z)=\rho[u](\tfrac{z-x}{\rho})\quad \text{for $\Hd$-a.e.\ $z\in J_{v^\rho}$}.
\end{gather}
In particular, from \eqref{change 1 bis} and the fact that $u\in C^{\rm bulk}_{p,\star}(A,L)$ we deduce that 
$v^\rho=\ell_{x,0,A}$ close to $\partial  Q \EEE  (x,\rho)$ and, by a change of    variables, \EEE 
\begin{equation*}
    \int_{Q(x,\rho)}\nabla v^\rho\,{\rm d}y=\int_{Q(x,\rho)}L\,{\rm d}y,
\end{equation*}
so that $v^\rho$ is a competitor for the minimisation problem $\m(\ell_{x,0,A},L,Q(x,\rho))$. As such,
\begin{equation}\label{inequality sad}
  \m(\ell_{x,0,A},L,Q(x,\rho))\leq   \mathcal{E}(v^\rho,Q(x,\rho)).
\end{equation}
Using \eqref{change 1 bis} and  \eqref{change 2 bis},   a  change of  variables   gives 
\begin{equation*}
    \mathcal{E}(v^\rho,Q(x,\rho))=   \rho^{d}\int_{Q}\Psi\big(x+\rho y,\nabla u(y)\big)\,{\rm d}y+ \rho^{d-1}\int_{J_u\cap Q}\gamma\big(x+\rho y,\rho[u](y),\nu_u(y)\big)\,{\rm d}\Hd(y).
\end{equation*}
From \ref{Wcontinuous}, \ref{(gamma4)}, and \eqref{bound above W} it follows that 
\begin{align}\notag 
  \Big|\int_{Q}\Psi\big(x+\rho y,\nabla  u(y)\big)\,{\rm d}y- &\frac{1}{\rho}\int_{J_u\cap Q}\gamma\big(x +\rho y,\rho[u](y),\nu_u(y)\big)\,{\rm d}\Hd(y)\\\notag &-\int_{Q}\Psi\big(x,\nabla u(y)\big)\,{\rm d}y+ \frac{1}{\rho}\int_{J_u\cap Q}\gamma\big(x,\rho[u](y),\nu_u(y)\big)\,{\rm d}\Hd(y)\Big|\\
  &\quad \leq\omega_\Psi(\rho)\int_{Q}2\overline{C}_\Psi(|\nabla u|^p+1)\,{\rm d}y+ \omega_\gamma(\rho)\int_{J_u\cap Q}|[u](y)|\,{\rm d}\Hd(y).\label{i thought we finished}
\end{align}\EEE
 By inequality \ref{(gamma2)} we  also get
\begin{align}\notag 
\frac{1}{\rho}\gamma\big(x,\rho [u](y),\nu_{u}(y)\big)\leq C_\gamma|[u](y)|
\end{align} for $\Hd$-a.e.\ $y\in J_u\cap Q$. \EEE
 Hence,  combining  \eqref{inequality sad} with  \eqref{i thought we finished}, dividing by $\rho^d$, letting $\rho\to 0^+$, and using \eqref{existence gamma zero}, the Dominated Convergence Theorem, and the fact that  $\omega_\Psi$ and $\omega_\gamma$ are continuous functions with $\omega_\Psi(0)=0$ and $\omega_\gamma(0)=0$, we  get
\begin{equation*}
    f_{\rm bulk}(x,A,L)\leq  \int_{Q}\Psi\big(x,\nabla u(y)\big)\,{\rm d}y+\int_{ J_{u}\cap Q}\gamma^0\big(x, [u](y),\nu_{u}(y)\big)\,{\rm d}\Hd(y).
\end{equation*}\EEE
   \EEE    Finally\EEE, recalling \eqref{quasiminimiser cell}, we conclude $
    f_{\rm bulk}(x,A,L)\leq {H}_p(x,A,L)+\eta$. 
Letting $\eta\to 0^+$ we conclude the proof of the second inequality in \eqref{fgeq Hp}  and thus of the part of the statement concerning the bulk energy density.\EEE

\medskip

\noindent \emph{Step 4.}  (Surface: Auxiliary   cell \EEE formula)   We are left with proving the part of the statement concerning the surface energy density. The proof can be obtained arguing similarly to \cite[Theorem~4.5]{ChoksiFonseca},   taking into account some modifications due to the different growth conditions.

The first step is  to define \EEE an auxiliary cell formula. For every $x\in\Omega$, $\zeta\in\Rk$, and $\nu\in\Sd$ \EEE we set 
\begin{align*}
    \widetilde{h}_p(x,\zeta,\nu)\coloneq \inf&\Big\{\liminf_{n\to+\infty}\int_{J_{ u_n\EEE}\cap Q_\nu}\gamma(x,[u_n](y),\nu_{u_n}(y))\,{\rm d}\Hd(y)\colon \, \{u_n\}_n\subset {\rm SBV}(Q_\nu;\Rk)\\&
    \quad u_n\to v_{\zeta,\nu} \text{ in } L^1(Q_\nu;\Rk),\,\, \nabla u_n\to 0 \text{ strongly in $L^p(Q_\nu;\Rk)$}
    \Big\},
\end{align*}
where we have set $v_{\zeta,\nu}\coloneq  \BBB v_{0,\zeta,0,\nu} \EEE $ (see \eqref{stepfun}).
Using arguments very similar to \cite[Proposition~4.2]{ChoksiFonseca}, properly taking into account also the $\nu$-dependence (see   \cite[Remark~3.5, Point 5]{MorandottiBook}\EEE),   one can \EEE show that
\begin{equation*}
\widetilde{h}_p(x,\zeta,\nu)=h_p(x,\zeta,\nu)\quad \text{ for all $x\in\Omega$, $\zeta\in\Rk$, and $\nu\in\Sd$.}   
\end{equation*}
Hence, in view of Theorem~\ref{thm: minimum=F} \MMM and the definition of $f_{\rm surf}$ (see also Remark~\ref{remark: invariance}), \EEE  to conclude, it is sufficient to show that 
\begin{equation}\label{to prove surf cell}
    \widetilde{h}_p(x,\zeta,\nu) \leq f_{\rm surf}(x,\zeta,\nu)\leq h_p(x,\zeta,\nu)\quad \
\end{equation}
 for every $x\in\Omega$, $\zeta\in\Rk$, and $\nu\in\Sd$ such that
\begin{equation}\label{fsur at point}
     f_{\rm surf}(x,\zeta,\nu)=\limsup_{\rho\to 0^+}\frac{\m(v_{x,\zeta,0,\nu},0, Q_\nu(x,\rho))}{\rho^{d-1}}=\limsup_{\rho\to 0^+}\frac{I^p_\star(v_{x,\zeta,0,\nu},0, Q_\nu(x,\rho))}{\rho^{d-1}}<+\infty.
\end{equation}

\medskip
\noindent \emph{Step 5.}  (Surface: $\widetilde{h}_p \leq f_{\rm surf}\leq h_p$) Let us fix $x\in\Omega$, $\zeta\in\Rk$, and $\nu\in\Sd$ as above. Since for any  $u\in \C^{\rm surf}_{p,\star}(\zeta,\nu)$ (see \eqref{surf competitors}) the function defined by $v(y)\coloneq u((y-x)/\rho)\EEE$ \MMM  for \EEE  $y\in Q_{\nu\EEE}(x,\rho)$ is a competitor for $\m(v_{x,\zeta,0,\nu},0, Q_\nu(x,\rho))$, the first equality in \eqref{fsur at point} and a change of variables prove the second inequality in \eqref{to prove surf cell}.

To prove $\widetilde{h}_p \leq f_{\rm surf}$, we argue as follows. For every $\rho>0$ we let  $\{u^\rho_n\}_n\subset {\rm SBV}(Q_\nu(x,\rho);\Rk)$ be \MMM such that \EEE $u_n^\rho\SDtostar (v_{x,\zeta,0,\nu},0)$ in $Q_\nu(x,\rho)$ and 
\begin{equation}\label{eq:truncation surf}
    \lim_{n\to+\infty}\mathcal{E}(u_n^\rho,Q_\nu(x,\rho))=I^p_\star(v_{x,\zeta,0,\nu},0,Q_\nu(x,\rho)).
\end{equation}
We let $R\coloneq 2|\zeta|$,   $N\in\N$, and \MMM  apply Lemma~\ref{lemma:CagnettiDetFree W} \EEE  to obtain $i_n=i(n,N,\rho)$ such that, setting $v_n^\rho\coloneq \Phi_{\MMM R_{i_n}}\circ u_n\in {\rm SBV}(Q_\nu(x,\rho);\Rk)$, we have  
\begin{equation}\label{truncated energy}
    \mathcal{E}(v_n^\rho,Q_\nu(x,\rho))\leq \Big(1+\frac{C}{N}\Big)\mathcal{E}(u_n^\rho,Q_\nu(x,\rho))+C\rho^{d}
\end{equation}
and $\|v^\rho_n\|_{L^\infty(Q_\nu(x,\rho))}\leq \tau^N R$.   Thanks to Lemma~\ref{lemma:convergence truncations},   the sequence \EEE $\{v_n^\rho\}_n$ converges to $v_{x,\zeta,0,\nu}$ strongly in $L^1(Q_\nu(x,\rho);\Rk)$ and $\{\nabla v^\rho_n\}_n$ converges to $0$ weakly in $L^p(Q_\nu(x,\rho))$. Moreover, from \eqref{fsur at point}--\eqref{truncated energy} we deduce that  
\begin{equation*}
     \limsup_{\rho\to 0^+} \limsup_{n\to+\infty}\frac{1}{\rho^{d-1}} \MMM \int_{J_{v^\rho_n}\cap Q_\nu(x,\rho)} \EEE \gamma(y,[v^\rho_n](y),\nu_{v^\rho_n}(y))\,{\rm d}\Hd(y)\leq \Big(1+\frac{C}{N}\Big)f_{\rm surf}(x,\zeta,\nu),
\end{equation*}
which by  \ref{(gamma2)}, \ref{(gamma4)},  and \eqref{fsur at point} implies
\begin{equation}\label{pezzo superficie}
     \limsup_{\rho\to 0^+} \limsup_{n\to+\infty}\frac{1}{\rho^{d-1}}\MMM \int_{J_{v^\rho_n}\cap Q_\nu(x,\rho)} \EEE \gamma(x,[v^\rho_n](y),\nu_{v^\rho_n}(y))\,{\rm d}\Hd(y)\leq \Big(1+\frac{C}{N}\Big)f_{\rm surf}(x,\zeta,\nu).
\end{equation}

We now consider the functions $\{w^\rho_n\}_n\subset {\rm SBV}(Q_\nu;\Rk)$ defined by $w^\rho_n(y)\coloneq  v^\rho_n(x+\rho y)$ \MMM  for \EEE $y\in Q_\nu$, such  that 
\begin{gather*}
    \nabla w^\rho_n(y)=\rho\nabla v^\rho_n(x+\rho y)\quad \text{ for $\Lb^d$-a.e. $y\in Q_\nu$},\\
    [w^\rho_n](y)=[v^\rho_n\EEE](x+\rho y)\quad \text{ for $\Hd$-a.e.\ $y\in Q_\nu\cap J_{w^\rho_n}={( \MMM J_{v^\rho_n}} \EEE -x)/\rho$}.
\end{gather*}
\BBB As $\{v_n^\rho\}_n$ converges to $v_{x,\zeta,0,\nu}$, \EEE  for every $\rho>0$ the sequence $\{w_n^\rho\}_n$ converges to $v_{\zeta,\nu}$ strongly in $L^1(Q_\nu;\Rk)$.
A change of variables, \BBB \ref{W4}, and \eqref{fsur at point}--\eqref{truncated energy} then show that 
\begin{align*}
    \limsup_{\rho\to 0^+}\limsup_{n\to+\infty}\int_{Q_\nu}|\nabla w^\rho_n|^p\,{\rm d}y&=\limsup_{\rho\to 0^+}\limsup_{n\to+\infty} \MMM \frac{\rho^{p}}{\rho^{d}} \EEE \int_{Q_\nu(x,\rho)}|\nabla v^\rho_n|^p\,{\rm d}y\\
    &\leq \limsup_{\rho\to 0^+}\limsup_{n\to+\infty}\frac{\rho^{p-1}}{c_W}\Big(\frac{\mathcal{E}(v^\rho_n,Q_\nu(x,\rho))}{\rho^{d-1}}+\BBB \frac{\rho}{c_W} \EEE \Big)\\
    &\leq \Big(1+\frac{C}{N}\Big) \limsup_{\rho\to 0^+}\frac{\rho^{p-1}}{c_W}f_{\rm surf}(x,\zeta,\nu)=0.
\end{align*}
Finally, with a diagonal argument we can find sequences $\{\rho_\ell\}_\ell$ and $\{n_\ell\}_\ell$, with $\rho_\ell\to 0^+$ and $n_\ell\to +\infty$ as $\ell\to+\infty$, such  that the functions defined by $z_\ell\coloneq w^{\rho_\ell}_{n_\ell}$ are competitors for the minimisation problem defining $\widetilde{h}_p(x,\zeta,\nu)$. 
Therefore, \eqref{pezzo superficie} and a change of variables give
\begin{equation*}
    \widetilde{h}_p(x,\zeta,\nu)\leq \Big(1+\frac{C}{N}\Big)f_{\rm surf}(x,\zeta,\nu). 
\end{equation*}
Letting $N\to+\infty$, we obtain the first inequality in \eqref{to prove surf cell}. \MMM This concludes the proof.  \EEE
\qed
\medskip

We now prove Proposition~\ref{prop:integrands}.
\medskip

\noindent\emph{Proof of Proposition \ref{prop:integrands}}.
The properties concerning the bulk energy density $H_p$ are a consequence of  \cite[Theorem~2.10]{J.Elasticity}.
It is immediate to check that $h_p$ \EEE satisfies \ref{(gamma1)}--\ref{(gamma3)}, while \ref{(gamma4)} can be obtained arguing exactly as in \cite[Theorem~2.10]{J.Elasticity}.

We now check that $h_p$ satisfies \ref{(gamma5)}.  To this end,  we let $x\in \Omega$, $\zeta_1,\zeta_2\in\Rk$, with $\beta_\gamma|\zeta_1|\leq |\zeta_2|$ and $\zeta_2\neq 0$, and $\nu\in\Sd$. We consider $u_2\in \C^{\rm surf}_{p,\star}(\zeta_2,\nu)$ (recall \eqref{surf competitors}), set $\lambda\coloneq |\zeta_1|/|\zeta_2|$, and let $R\in {\rm SO}(k)$ be such that $R(\lambda\zeta_2)=\zeta_1$.  \MMM Then, \EEE $u_1\coloneq \lambda R(u_2)\in \C^{\rm surf}_{p,\star}(\zeta_1,\nu)$, and clearly $\beta_\gamma|[u_1]|\leq \beta_\gamma \lambda |R[u_2]|\leq |[u_2]|$, where we have used that $\beta_\gamma\lambda\leq 1$. Since $\gamma$ satisfies \ref{(gamma5)}, this implies 
\begin{equation*}
  h_{p}(x,\zeta_1,\nu)\leq   \MMM \int_{J_{u_1}\cap Q_\nu} \EEE \gamma(x,[u_1](y),\nu_{u_1}(y))\,{\rm d}\Hd(y)\leq \MMM \int_{J_{u_2}\cap Q_\nu} \EEE \gamma(x,[u_2](y),\nu_{u_2}(y))\,{\rm d}\Hd(y).
\end{equation*}
From the arbitrariness of $u_2\in \C^{\rm surf}_{p,\star}(\zeta_2,\nu)$, we conclude that $h_p$ satisfies \ref{(gamma5)}.
\EEE
 \qed
\EEE

\medskip

\subsection{Proof of Theorem \ref{thm:relaxation full}}
The rest of the section is devoted to the proof of Theorem \ref{thm:relaxation full}.   We have already shown in Theorems \ref{thm: relaxation partial} and \ref{thm:cell formulas} that the bulk and surface parts $(I^p_\star)^a$ and $(I^p_\star)^j$ admit an integral representation, so that we  are  left with proving that $(I^p_\star)^c$ can be represented as an integral functional, as well (see Definition \ref{decomposition}). This is done in three steps. We first use Lemma~\ref{lemma: Lipschitz} below to deduce  that $(I^p_\star)^c(g,G,\cdot)=(I^p_\star)^c(g,0,\cdot)$ for all $(g,G)\in \SD^p_\star(\Omega;\Rk)$. Then, for $\delta>0$ we will consider a perturbed functional $ I^{p,\delta}_{\star}(g,B)$  defined for $g\in{\rm BV}(\Omega;\Rk)$ and $B\in\B(\Omega)$ which satisfies all the hypotheses of \cite[Theorem~3.12]{BFM1998}, and thus
admits a complete integral representation, including its Cantor part. Letting $\delta\to 0^+$ and using Lemma~\ref{lemma: recession}, we recover an integral representation for $(I^p_{\star})^c(g,0,B)$ for all $g\in {\rm BV}(\Omega;\Rk)$, which can then be extended to all  $g\in {\rm GBV}_\star(\Omega;\Rk)$ by a truncation argument based on Lemma~\ref{lemma:I truncated}  below.

We now present Lemma~\ref{lemma: Lipschitz}, which states that $I^p_\star$ is \MMM locally Lipschitz. \EEE  Note that both in this lemma and in Lemma~\ref{lemma:I truncated} below we do not assume that $\Psi$ and $\gamma$ are independent of the $x$-variable, a property that will only be used  to apply \cite[Theorem~3.12]{BFM1998} in the proof of Theorem~\ref{thm:relaxation full}.

\begin{lemma}\label{lemma: Lipschitz}
 Assume that $\Psi$ and $\gamma$ satisfy  {\rm \ref{(W1)_p},\ref{(W2)_p}, \ref{W4}} \EEE
 \EEE and {\rm \ref{(gamma1)}--\ref{(gamma3)}}.      Let $(g_1,G_1),(g_2,G_2)\in {\rm SD}^p_\star(\Omega;\Rk)$.
    Then for every $B\in\B(\Omega)$ it holds 
    \begin{align} 
\big|I^p_\star(g_1,G_1,B)-I^p_\star(g_2,G_2,B)\big| \leq   C_\gamma \EEE \mathcal{V}(g_1-g_2,B) \label{claim Lipschitz} +\MMM  C\|G_1-G_2\|^p_{L^p(B)}+\widehat{C}C\|G_1-G_2\|_{L^p(B)} \EEE, \EEE 
    \end{align}
  where $\widehat{C}\geq 0$ is given by    \begin{equation}\label{constant chat}
\widehat{C}\coloneq  I^{p}_\star(g_1,G_1,B)^{(p-1)/p}+I^{p}_\star(g_2,G_2,B)^{(p-1)/p}+(\Lb^d(B))^{(p-1)/p} 
    \end{equation}
 \MMM  and  $C>0$ is a constant \EEE depending only on   $c_\Psi$, ${C}_\Psi$, $C_\gamma$, $p$,  and $d$.
\end{lemma}
\begin{proof}
 As $I^p_\star\in \F^p_\star$, \BBB in view of \eqref{def Borel relax}, \EEE  it is enough to prove \eqref{claim Lipschitz} for a fixed $O\in\Op(\Omega)$. 
Arguing by symmetry, it is enough to show that
\begin{align}\label{claim lipschitz}
  \notag \hspace{-0.3 cm}I^p_\star(g_1,G_1,O)&\leq I^p_\star(g_2,G_2,O)+  C_\gamma \EEE \mathcal{V}(g_1-g_2,O)  \\ &
 \quad  \MMM  +C \|G_1-G_2\|^p_{L^p(O)} + C\Big(I^{p}_\star(g_2,G_2,O)^{(p-1)/p}+(\Lb^d(O))^{(p-1)/p} \Big)\|G_1-G_2\|_{L^p(O)} \EEE
\end{align}
  for   $C>0$ depending only on $c_\Psi$, ${C}_\Psi$, $C_\gamma$, $p$,  and $d$. 
  
We consider a sequence of functions $\{u_n\}_n\subset {\rm SBV}(O;\Rk)$ with $u_n\SDtostar (g_2,G_2)$ in $O$ and 
\begin{equation}\label{Lipschitz recovery}
    \lim_{n\to+\infty}\mathcal{E}(u_n,O)= I^p_\star(g_2,G_2,O).
\end{equation}
 We  use  Alberti's Theorem~\ref{Alberti's Theorem} to \MMM find \EEE  $h\in {\rm SBV}(O;\Rk)$ such that $\nabla h=G_1-G_2$ $\Lb^d$-a.e.\ in $O$.  We then apply Lemma~\ref{lemma piecewise constants} (see Remark~\ref{re:lastonestanding}) \EEE to $g_1-g_2-h$  and obtain a sequence of piecewise constant functions $\{h_n\}_n\subset {\rm SBV}(O;\Rk)$ converging to $g_1-g_2-h$ in $L^0(O;\Rk)$ and such that  
\begin{equation}\label{V Lipschitz}
\lim_{n\to+\infty}\mathcal{V}(h_n,O)= \mathcal{V}(g_1-g_2-h,O)\leq \mathcal{V}(g_1-g_2,O)+ ( C(d) + 1) \|G_1-G_2\|_{L^1(O \EEE )},
\end{equation}
where $C(d)$ is the constant of Theorem~\ref{Alberti's Theorem}.  
We set $v_n\coloneq u_n+h+h_n$ and observe that  $v_n \to g_1$ in $L^0(O;\R^k)$ and   $\nabla v_n \rightharpoonup G_1$ weakly in $ {L}^p(O;\Rkd)$ \EEE as $n\to +\infty$, so that $v_n\SDtostar(g_1,G_1)$ in $O$. In particular,
\begin{equation}\label{O Lipschitz}
    I^p_\star(g_1,G_1,O)\leq \liminf_{n\to+\infty} \mathcal{E}(v_n,O).
\end{equation}
 As   $J_{v_n}\subset J_{u_n}\cup J_{h_n}\cup J_h\EEE$ for every $n\in\N$,  we \EEE  can choose an orientation for the normals  
such that,  for every $n\in\N$, 
 \begin{align*}
     [v_n]=[u_n]+[h_n]+[h]\EEE \quad \text{ $\Hd$-a.e. on $ (J_{u_n}\cup J_{h_n}\cup J_h\EEE)\cap O$}.
 \end{align*}
By the previous observations, using also {\rm \ref{(W2)_p},  \ref{(gamma2)},  \ref{(gamma3)}},   the fact that $\nabla h=G_1-G_2$ $\Lb^d$-a.e.\ in $O$, and H\"older's inequality  we get 
   \begin{align}\notag 
       \mathcal{E}(v_n,O)&=\int_{O}\Psi(x,\nabla u_n+\nabla h)\,{\rm d}x+\int_{J_{v_n}\cap O}\gamma(x,[v_n],\nu_{v_n})\,{\rm d}\Hd\\\notag
       &\leq \int_{O}\Psi(x,\nabla u_n)\,{\rm d}x+  C_\Psi\int_{O}|\nabla h|(1+|\nabla u_n + \nabla h|^{p-1} + |\nabla u_n |^{p-1})\,{\rm d}x\\&\notag \quad + \int_{J_{u_n}\cap O}\gamma(x,[u_n],\nu_{v_n})\,{\rm d}\Hd+\int_{J_{h_n}\cap O}\gamma(x,[h_n],\nu_{h_n})\,{\rm d}\Hd +\int_{J_{h}\cap O}\gamma(x,[h],\nu_{h})\,{\rm d}\Hd 
       \\& \notag
       \leq \int_{O}\Psi(x,\nabla u_n)\,{\rm d}x+\int_{J_{u_n}\cap O}\gamma(x,[u_n],\nu_{u_n})\,{\rm d}\Hd +  C_\gamma\mathcal{V}(h_n,O)\\ 
       &\quad  +C\big(\|G_1-G_2\|^p_{L^p(O)}+\|G_1-G_2\|_{L^1(O)}\big)\EEE\label{zampe}  +C\|\nabla u_n\|^{p-1}_{L^p(O)}\|G_1-G_2\|_{L^p(O)}
   \end{align}
 for   $C>0$ depending only on $c_\Psi$, $C_\Psi$, $C_\gamma$, and  $p$. \EEE In light of \ref{W4} and \eqref{Lipschitz recovery}, we have 
   \begin{align*}
       \|\nabla u_n\|^{p-1}_{L^p( O\EEE )}&\leq \Big(\frac{1}{ c_\Psi\EEE}\int_{ O\EEE}\Big(\Psi(x,\nabla u_n)+\MMM \frac{1}{ c_\Psi\EEE} \EEE \Big)\,{\rm d}x\Big)^{(p-1)/p} \\
       &\leq  \MMM 2\Big(\frac{1}{ c_\Psi\EEE}\Big)^{(p-1)/p} I^{p}_\star(g_2,G_2,O)^{(p-1)/p}+ \Big(\frac{1}{ c_\Psi^2\EEE}\Big)^{(p-1)/p}(\Lb^d(O))^{(p-1)/p} \EEE
   \end{align*}
   for every $n\in\N$ large enough.  
    Hence, in view of \eqref{Lipschitz recovery}, \eqref{V Lipschitz},  \eqref{zampe}, and H\"older's inequality we deduce  
   \begin{align*}
\liminf_{n\to+\infty}\mathcal{E}(v_n,O)&\leq I^p_\star(g_2,G_2,O) +  C_\gamma \EEE \mathcal{V}(g_1-g_2,O) \\&\quad \MMM  +C \|G_1-G_2\|^p_{L^p(O)} + C\Big(I^{p}_\star(g_2,G_2,O)^{(p-1)/p}+(\Lb^d(O))^{(p-1)/p} \Big)\|G_1-G_2\|_{L^p(O)} \EEE 
   \end{align*}
  for a different constant $C>0$, depending only on   $c_\Psi$, ${C}_\Psi$, $C_\gamma$, $p$,  and $d$. \EEE
   By  \eqref{O Lipschitz}, this \BBB leads to \eqref{claim lipschitz}, which \EEE   concludes the proof. 
\end{proof}

We now study truncation properties of $I^p_\star(\cdot,0,\cdot)$.
 \begin{lemma}\label{lemma:I truncated}
 Let $p> 1$   and assume that $\Psi$ and $\gamma$ satisfy  {\rm \ref{(W1)_p}--\ref{W4}} and {\rm \ref{(gamma1)},\ref{(gamma2)}}, and {\rm \ref{(gamma5)}}. Then for every $N\in\N$, $g\in {\rm GBV}_\star(\Omega;\Rk)$,    $B\in\B(\Omega)$, and $R>0$ it holds
 \begin{align}
  \hspace{-0.3 cm}  \frac{1}{N}\sum_{i=1}^NI^p_\star \big(\Phi_{R_i}\circ g,0,B\big)&\leq \Big(1+\frac{C}{N}\Big)I^p_\star(g,0,B)\label{claim truncation estimate relax} + 
  C\Lb^d\big(\MMM \{x\in B\colon \, |{g}(x)|\geq R\} \EEE \big)
  +\frac{C}{N}\Lb^d(B)
 \end{align}
 for   $C>0$ depending only on   $c_\Psi,\overline{C}_\Psi, c_\gamma, C_\gamma,$ and $p$, where we have set $R_i\coloneq \tau^{i-1}R$.
 
Moreover,   for every $g\in {\rm GBV}_\star(\Omega;\Rk)$,   $B\in\B(\Omega)$, and \BBB every \EEE  sequence $\{R_N\}_N\subset (0,+\infty)$, with $R_N\to +\infty$ as $N\to+\infty$, it holds
 \begin{equation}\label{cantor truncation}
     \lim_{N \to+\infty} \frac{1}{N}\sum_{i=1}^N(I^p_{\star})^c(\Phi_{\tau^{i-1}R_N}\circ g,0,B)=(I^p_\star)^c(g,0,B).
 \end{equation} 
\end{lemma}
\begin{proof}
\DDD As $I^p_\star\in \F^p_\star$, in view of \eqref{def Borel relax}, \EEE  it is enough to prove \eqref{claim truncation estimate relax} and \eqref{cantor truncation} for all $O\in\Op (\Omega)$. 
Let us fix $N\in\N$, $g\in {\rm GBV}_\star(\Omega;\Rk)$, $R>0$, and $O\in\Op (\Omega)$.
   We let $\{u_n\}_{n}\subset {\rm SBV}(O;\Rk)$ be a sequence with $u_n\SDtostar (g,0)$ in $O$ such that
   \begin{equation}\label{recovery for I}
\lim_{n\to+\infty}\mathcal{E}(u_n,O)= I^p_\star(g,0,O).
   \end{equation}
We set $v^i_n\coloneq \Phi_{R_i}\circ u_n$ and $v^i\coloneq \Phi_{R_i}\circ g$.  An application of Lemma~\ref{lemma:convergence truncations}  yields  
\begin{equation}\label{SDstar convergence lemma}
  v^{i}_n \SDtostar (v^{i},0) \quad \text{in $O$ as $n\to+\infty$}
 \end{equation} 
 for all $i\in\{1,\dots, N\}$. \EEE 
  Thanks to Lemma~\ref{lemma:CagnettiDetFree W},  there exists $C>0$, depending only on $c_\Psi,\overline{C}_\Psi, c_\gamma, C_\gamma,$ such that  
\begin{align*}
     \frac{1}{N} \sum_{i=1}^N\mathcal{E}(v^i_n,O)\leq \Big(1+\frac{C}{N}\Big)\mathcal{E}(u_n,O)+C\Lb^d(\{x\in O \colon \,  |{u_n}(x)|\geq R\})+\frac{C}{N}\Lb^d(O)
\end{align*}
     for every $n\in\N$. In light of \eqref{recovery for I}  and of the convergence $u_n$ to $g$ in $L^0(O;\Rk)$\EEE, this implies   
\begin{align*} 
        \frac{1}{N}\sum_{i=1}^N \liminf_{n\to+\infty} \mathcal{E}(v^i_n,O)& \leq\Big(1+\frac{C}{N}\Big) I^p_{\star}(g,0,O) +C\Lb^d(\{x\in O \colon \,  |{g}(x)|\geq R\}\EEE)+\frac{C}{N}\Lb^d(O).
      \end{align*}
Combining this inequality with \eqref{SDstar convergence lemma}, by definition \eqref{def Istar localised} of $I^p_\star$ we infer
 \eqref{claim truncation estimate relax}. The proof of \eqref{cantor truncation} can be obtained arguing exactly as in \cite[Lemma~5.2]{dal2025homogenization}.  
\end{proof}

We are finally ready to prove Theorem~\ref{thm:relaxation full}. \
\medskip
 
\noindent \emph{Proof of Theorem \ref{thm:relaxation full}.\ }  As we have already characterised the bulk and surface parts of $I^p_\star$ in Theorem~\ref{thm:cell formulas}, we are left with proving that (recall Definition \ref{decomposition}) 
\begin{equation}\label{claim cantor repr}
(I^p_\star)^c(g,G,B)=\int_BH_p^\infty\Big(\frac{{\rm d}D^cg}{{\rm d}|D^cg|},0\Big)\,{\rm d}|D^cg|
\end{equation}
for all $(g,G)\in {\rm SD}^p_\star(\Omega;\Rk)$ and  $B\in\B(\Omega)$. 

\medskip

\noindent \emph{Step 1.\ } (Replacement of $G$ by $0$) 
We first show that 
\begin{equation}\label{Cantor=Cantor0}
(I^p_\star)^c(g,G,B)= (I^p_\star)^c(g,0,B)
\end{equation}
for all $(g,G)\in {\rm SD}^p_\star(\Omega;\Rk)$ and  $B\in\B(\Omega)$. Equivalently, we prove that
 for \EEE all $(g,G)\in {\rm SD}^p_\star(\Omega;\Rk)$  it holds \EEE
\begin{equation}\label{wof wof}
    \frac{{\rm d}I^p_\star(g,G,\cdot)}{{\rm d}|D^cg|}(x)= \frac{{\rm d}I^p_\star(g,0,\cdot)}{{\rm d}|D^cg|}(x)\quad \text{for $|D^cg|$-a.e.\ $x\in\Omega$.}
\end{equation}
To prove this, for every $x\in\Omega$ and $\rho>0$ we use Lemma~\ref{lemma: Lipschitz} with $g_1=g_2=g$ and $ B=B(x,\rho)\EEE$ to obtain 
\begin{align*}
\frac{I^p_\star(g,0,B(x,\rho))}{|D^cg|(B(x,\rho))}\leq \frac{I^p_\star(g,G,B(x,\rho))}{|D^cg|(B(x,\rho))}+ \MMM \frac{C}{|D^cg|(B(x,\rho))}\Big(\|G\|^p_{L^p(B(x,\rho))}+\widehat{C}\|G\|_{L^p(B(x,\rho))}\Big),
\end{align*}
\EEE where $\widehat{C}$ is given by \eqref{constant chat}. \EEE 
By the Besicovitch Derivation Theorem this proves that   
\begin{equation*}
    \frac{{\rm d}I^p_\star(g,0,\cdot)}{{\rm d}|D^cg|}(x) \le   \frac{{\rm d}I^p_\star(g,G,\cdot)}{{\rm d}|D^cg|}(x) \quad \text{for $|D^cg|$-a.e.\ $x\in\Omega$.}
\end{equation*}
The reverse inequality is obtained in the exact same way. This concludes the proof of \eqref{wof wof} and thus of \eqref{Cantor=Cantor0}.

\medskip

\noindent  \noindent \emph{Step 2.\ } (Perturbative argument) For $\delta>0$, consider the functional $I^{p,\delta}_\star\colon {\rm BV}(\Omega;\Rk)\times \B(\Omega)\to [0,+\infty)$ defined by 
\begin{equation*}
I^{p,\delta}_\star(g,B)\coloneq I^{p}_\star(g,0,B)+\delta \int_{J_g\cap B}|[g]|\,{\rm d}\Hd \quad \text{ for all }g\in {\rm BV}(\Omega;\Rk) \text{ and }B\in\B(\Omega).
\end{equation*}
Since $I^p_\star\in \F^p_\star$ and $I^p_\star$ satisfies  the invariance property \eqref{independence day}\EEE, the functional $I^{p,\delta}_\star$ satisfies all hypotheses of \cite[Theorem~3.12]{BFM1998}.   Hence, there exist functions $f_{\rm bulk}^\delta\colon\Rkd\to [0,+\infty)$ and $f_{\rm surf}^\delta\colon \Rk\times \Sd\to [0,+\infty)$ such that 
\begin{equation}\label{delta representation}
\MMM I^{p,\delta}_\star(g,  B) \EEE =\int_{B}f_{\rm bulk}^\delta(\nabla g)\,{\rm d}x+\int_{B}f^{\delta,\infty}_{\rm bulk}\Big(\frac{{\rm d}D^cg}{{\rm d}|D^cg|}\Big)\,{\rm d}|D^cg|+\int_{ J_g\cap \EEE B}f_{\rm surf}^\delta([g],\nu_g)\,{\rm d}\Hd
\end{equation}
for all $g\in {\rm BV}(\Omega;\Rk)$ and $B\in\B(\Omega)$. It is \BBB elementary \EEE  to see that $f^{\delta}_{\rm bulk} = H_p$ \MMM by Theorem~\ref{thm:cell formulas}, \EEE and thus, in view of Proposition~\ref{prop:integrands}, we have $f^{\delta}_{\rm bulk}(A)\leq  C_H (|A|+1)$ for all $A\in\Rkd$. Similarly, as $h_p$ satisfies \ref{(gamma2)}, it holds that $f^{\delta}_{\rm surf}(\zeta,\nu)\leq(  C_\gamma\EEE+\delta)|\zeta|$ for all $\zeta\in\Rk$ and $\nu\in\Sd$. \   Moreover, one can check that
  $f^{\delta}_{\rm surf}$ and $f^\delta_{\rm bulk}$ are nondecreasing  in $\delta$, and that the same property holds for $A\mapsto f^{\delta,\infty}_{\rm bulk}(A)$ when $A$ is a rank-one matrix with \MMM unit \EEE  norm.   Thus,  we may define
\begin{gather*}
    \widehat{f}_{\rm bulk}(A)=\lim_{\delta\to 0^+}f^\delta_{\rm bulk}(A)=\inf_{\delta>0}f^\delta_{\rm bulk}(A) \quad \text{ for all $A\in\Rkd$,}\\
\widehat{f}_{\rm Cantor}(A)\coloneq \lim_{\delta\to0^+}f^{\delta,\infty}_{\rm bulk}(A)=\inf_{\delta>0}f^{\delta,\infty}_{\rm bulk}(A) \quad \text{ for all $A\in\Rkd$ with rank $1$ and $|A|=1$,}\\
\widehat{f}_{\rm surf} \MMM (\zeta,\nu) \EEE \coloneq \lim_{\delta\to 0^+}f_{\rm surf}^\delta(\zeta,\nu)=\inf_{\delta>0}f_{\rm surf}^\delta(\zeta,\nu)\quad \text{ for all $\zeta\in\Rk$ and $\nu\in\Sd$}.
\end{gather*}
The function \MMM  $\widehat{f}_{\rm Cantor}$ \EEE is extended by $1$-homogeneity to all rank-one matrices.
Together with \eqref{delta representation}, the Monotone Convergence Theorem then implies that 
\begin{equation*}
I^{p}_\star(g,0,B)=\int_{B}\widehat{f}_{\rm bulk}(\nabla g)\,{\rm d}x+\int_{B}\widehat{f}_{\rm Cantor}\Big(\frac{{\rm d}D^cg}{{\rm d}|D^cg|}\Big)\,{\rm d}|D^cg|+\int_{ J_g\cap \EEE B}\widehat{f}_{\rm surf}([g],\nu_g)\,{\rm d}\Hd
\end{equation*}
for all $g\in {\rm BV}(\Omega;\Rk)$ and $B\in\B(\Omega)$. \MMM Thus, by \EEE 
Theorem~\ref{thm:cell formulas}, we obtain
\begin{equation*}
    \widehat{f}_{\rm bulk}(A)=H_p(A,0) \quad\text{and }\quad \widehat{f}_{\rm surf}(\zeta,\nu)=h_p(\zeta,\nu)
\end{equation*}
for all $A\in\Rkd$, $\zeta\in\Rk$, and $\nu\in\Sd$.

Then, we observe that the corresponding bounds for $f^\delta_{\rm bulk}$ and $f^\delta_{\rm surf}$ also give that $\widehat{f}_{\rm bulk}(A)\leq C_H(|A|+1)$ for all $A\in\Rkd$ and that $\widehat{f}_{\rm surf}(\zeta,\nu)\leq C_\gamma|\zeta|$ for all $\zeta\in\Rk$ and $\nu\in\Sd$.  Since the functional $g\mapsto I^p_\star(g,0,O)$ is $L^1(\Omega;\Rk)$-lower semicontinuous on ${\rm BV}(\Omega;\Rk)$ for all $O\in\Op(\Omega)$, by Lemma~\ref{lemma: recession} we  then deduce that  
\begin{equation*}
H_p^\infty(A,0) =\MMM \widehat{f}_{\rm Cantor} \EEE (A)
\end{equation*}
for all rank-one matrices $A\in\Rkd$ and 
\begin{equation*}
I^{p}_\star(g,0,B)=\int_{B}H_p(\nabla g, \BBB 0 \EEE )\,{\rm d}x+\int_{B}H_p^\infty\Big(\frac{{\rm d}D^cg}{{\rm d}|D^cg|},0\Big)\,{\rm d}|D^cg|+\int_{ J_g\cap \EEE B} h_p([g],\nu_g)\,{\rm d}\Hd
\end{equation*}
for all $g\in {\rm BV}(\Omega;\Rk)$ and $B\in\B(\Omega)$. In light of \eqref{Cantor=Cantor0}, this shows that \eqref{claim cantor repr} holds whenever $(g,G)\in \SD^p_\star(\Omega;\Rk)$, with $g\in {\rm BV}(\Omega;\Rk) $.   

\medskip

\noindent \emph{Step 3.\ } (Conclusion) We now deal with the general case $g\in {\rm GBV}_\star(\Omega;\Rk)$. To this end, for every $R>0$ we consider the set $\widetilde{\Omega}^R_{g}\coloneq \{x\in \Omega\colon\,\widetilde{g}(x) \text{  exists and }|\widetilde{g}(x)|\leq R\}$ . We claim that it is enough to prove \eqref{claim cantor repr} for every $B\in\mathcal{B}(\Omega)$ for which there exists $R>0$ such that $B\subset \widetilde{\Omega}^R_{g}$. Indeed, by Proposition~\ref{prop:Properties of derivatives of GBVstar} we have that $ \widetilde{\Omega}^R_{g}\nearrow \Omega_{\rm reg}\coloneq \{x\in \Omega\colon \widetilde{g}(x) \,\,\text{exists}\}$ and $|D^cg|(\Omega\setminus \Omega_{\rm reg})=0$. Hence, every $B\in\mathcal{B}(\Omega)$ can be written, up to a $|D^cg|$-negligible set, as the increasing union of Borel sets each contained in $ \widetilde{\Omega}^{R_n}_{g}$, for some sequence $R_n \to + \infty$.    This proves the claim.

     Let us fix $R>0$, a Borel set $B\subset  \widetilde{\Omega}^R_{g}$, and   a sequence $\{R_N\}_N$ with $R_N\to+\infty$ and $R_N>R$. Thanks to Lemma~\ref{lemma:I truncated}, we have  
     \begin{equation}\label{pronto}
         \lim_{N\to+\infty}\frac{1}{N}\sum_{i=1}^N(I^p_\star)^c(v^i_N,0,B)=(I^p_\star)^c(g,0,B),
     \end{equation}
     where we have set $v^i_N\coloneq  \BBB \Phi_{\tau^{i-1}R_N\circ g}$.   Since $v^i_N\in {\rm BV}(\Omega;\Rk)$ for all $i\in\{1,\dots,N\}$ and $N\in\N$, 
      the integral representation \eqref{claim cantor repr}  in ${\rm BV}(\Omega;\Rk)$  can  be applied, and thus 
     \begin{equation}\label{spid}
    (I^p_\star)^c(v^i_N, \MMM 0, \EEE B)=\int_B H_p^\infty\Big (\frac{{\rm d}D^c v^i_N}{{\rm d}|D^cv^i_N|},0\Big)\, d|D^cv^i_N|
     \end{equation}
     for every $i\in\{1,\dots,N\}$ and $N\in\N$. As we assumed that  $B\subset  \widetilde{\Omega}^R_{g}$ and \BBB $R_N>R$, \EEE we have $B\subset  \widetilde{\Omega}^{R_N}_{g}$. Thus,  from (ii)   of Proposition~\ref{prop:Properties of derivatives of GBVstar} we get 
     \begin{eqnarray*}
        &\displaystyle \frac{{\rm d}D^c v^i_N}{{\rm d}|D^cv^i_N|}=\frac{{\rm d}D^cg}{{\rm d}|D^cg|} \quad |D^cg|\text{-a.e. in }  B,\\
        &\displaystyle |D^cv^i_N|=|D^cg| \quad \text{as Borel measures on $ B$},
     \end{eqnarray*}
     for every $i\in\{1,\dots,N\}$ and $N\in\N$. Together with \MMM \eqref{Cantor=Cantor0}, \EEE \eqref{pronto}, and \eqref{spid},  these equalities give \eqref{claim cantor repr}. \MMM This \EEE  concludes  the proof. 
\medskip

\smallskip

\noindent\textbf{Acknowledgements.}
The authors wish to thank Gianni Dal Maso, with whom D.D.\  devised the core arguments of Theorem \ref{thm: Goffmann-Serrin} and  of Remark \ref{remark: vVec}   in the context of a different research project.  
This work was carried out as D.D.\ was visiting Johannes Kepler Universit\"at of Linz, as part of the Erasmus+ program. He wishes to thank this institution and, in particular the Institut für
Analysis,  for the warm hospitality he received during this visit. He also wishes to thank SISSA for the financial support.  D.D.\ is a member of GNAMPA of INdAM. \MMM This research was funded by the Deutsche Forschungsgemeinschaft (DFG, German Research Foundation) - 377472739/GRK 2423/2-2023.  M.F.\  is \EEE \MMM very grateful for this support. \EEE

\appendix

\section{Generalised functions of bounded variation}\label{sec: GBV}

 In the sequel, $\Omega$ stands for an open, bounded subset of $\R^d$. We do not assume any regularity on $\Omega$ unless specifically stated.

In this section, we collect some further properties of generalised functions of bounded variation. Before doing so, we recall the definition of the approximate gradient and of the jump set of an $\Lb^d$-measurable function. 

\medskip 

\EEE
\noindent{\textbf{Approximate gradient.}} 
    Let $E$ be an  $\mathcal{L}^d$-measurable subset of $\R^d$ and let $x\in \Rd$ be a point with positive $d$-dimensional density, i.e., 
    \begin{equation}\label{positive density}
        \limsup_{\rho \to 0^+}\frac{\Lb^d(E\cap B_\rho(x))}{\rho^d}>0.
    \end{equation}
    A function $u\in L^0(E;\R^k)$  admits an {\it approximate limit} $\widetilde{u}(x)\in\R^k$ at $x$
    if, for every $\e>0$, we have
    \begin{equation*}
         \lim_{\rho\to 0^+}\frac{\Lb^d(\{y\in E\colon \,  |u(y)-\widetilde{u}(x)|>\e\}\cap B_\rho(x))}{\rho^d}=0; 
    \end{equation*}
    in this case, we write 
    \begin{equation}\label{aplim}
\text{ap}\!\lim_{\substack{y\to x\\y\in E}}u(y)=\widetilde{u}(x).
    \end{equation}
 Throughout the paper, the symbol $\widetilde{u}(x)$  always denotes the approximate limit of $u$ defined by \eqref{aplim}, which,  thanks to \eqref{positive density}, is uniquely defined. We also introduce the set $S_u$ of {\it singular points} of $u$ as the complement in $E$ of the set of points where the approximate limit exists. We recall that for every $u\in L^0(\Rd;\R^k)$ it holds $\Lb^d(S_u)=0$. 

 If $O\in \Op(\Omega)$ and $u\in L^0(O;\Rk)$, a function $\nabla u\in L^0(O;\Rkd)$ is said to be  the approximate gradient of $u$ if, for $\Lb^d$-a.e.\ $x\in O\EEE$, it holds 
\begin{equation}\label{approximate grad}
      {\rm ap}\lim_{\substack{y\to x\\y\in O}}\frac{u(y)-\widetilde{u}(x)-\nabla u(x)(y-x)}{|y-x|}=0.
\end{equation}

\medskip

\noindent{\textbf{Jump set.}}
We recall the definition of jump set of a measurable function. \MMM Let \EEE  $O \in \mathcal{O}(\R^d)$ \EEE and $u\in L^0(O;\R^k)$. The   \textit{jump set} $J_u$ is the Borel set consisting of all points $x\in O$ for which there exists a triple  $(u^+(x),u^-(x),\nu_u(x))\in \R^k\times \R^k\times \mathbb{S}^{d-1}$, with $u^+(x)\neq u^-(x)$, such that 
\begin{align}\label{def Jump set}
       {\rm ap}\!\!\!\!\!\!\!\!\!\lim_{\substack{y\to x\\y\in H^{\pm}(x)\cap O}}u(y)=u^\pm(x),
     \end{align}
where we have set $H^{\pm}(x)\coloneq \{y\in\Rd:\pm(y-x)\cdot\nu_u(x)>0\}$.  Note that the triple $(u^+(x),u^-(x),\nu_u(x))$ is uniquely defined up to swapping the roles of $u^+(x)$ and $u^-(x)$  and changing the sign of $\nu_u(x)$. For brevity of notation, we set $[u](x)\coloneq u^{+}(x)-u^{-}(x)$ for all $x\in J_u$.

 We fix a smooth map $\Phi\in C^\infty_c(\Rk;\Rk)$ satisfying the following conditions (see, for an example, \cite[Section~4]{CagnettiDet}):  
 \begin{equation} \label{Properties of Phi}
 \begin{cases}
    \Phi(y)=y \quad \text{for }y\in \{|y|\leq 1\},\\
    \Phi(y)=0 \quad \text{for }y\in \{|y|\geq  \tau\},\\
    {\rm Lip}(\Phi)\leq 1,
    \end{cases}
 \end{equation}
 for a fixed constant $\tau>2$. \BBB (In our work, we  choose  $\tau\geq \max \{3,\beta_\gamma+1\}$,   see before Lemma~\ref{lemma:CagnettiDetFree W}.) \EEE 
 For every $R>0$ we set 
 \begin{equation}\label{def PhiR}
     \Phi_R(y)\coloneq R\Phi\big(\frac{y}{R}\big)\quad \text{for all $y\in\R^k$},
 \end{equation}  
that satisfies
 \begin{equation} \label{Properties of PhiR}
 \begin{cases}
    \Phi_R(y)=y \quad \text{for }y\in \{|y|\leq R \},\\
    \Phi_R(y)=0 \quad \text{for }y\in \{|y|\geq \tau\EEE R\},\\
    {\rm Lip}(\Phi_R)\leq 1.
    \end{cases}
 \end{equation}
The following proposition collects the main properties of functions in ${\rm GBV}(\Omega;\Rk)$. \MMM For this, observe that $\Phi_R \circ u$ lies in $\BV_{\rm loc}$ and thus admits the usual representation of the distributional derivative $D(\Phi_R \circ u)$. \EEE
\begin{proposition}\label{Prop:fine prop GBV}
    Let $u\in \GBV(\Omega;\Rk)$. The following hold:
    \begin{itemize}
        \item [{\rm (a)}]  $\Hd(S_u\setminus J_u)=0$;
        \item [{\rm(b)}] $u$ admits an approximate gradient  $\nabla u\in L^0(\Omega;\Rkd)$;
moreover, $\nabla u(x)=\nabla (\Phi_R\circ u)(x)$ for $\Lb^d$-a.e.\ $x\in\{|u|\leq R\}$;
    \item [{\rm (c)}] for every $0<R_1< R_2$ we have $J_{\Phi_{{R_1}\circ u}}\subset J_{\Phi_{{R_2}\circ u}}$  up to a set of $\Hd$-measure zero and \begin{equation*}\Hd\big(J_u\setminus \bigcup\nolimits_{R>0}J_{\Phi_R\circ u}\big)=0;
        \end{equation*} moreover, for $\Hd$-a.e.\ $x\in J_u$ we have  $(\Phi_{R}\circ u)^+\EEE(x)=u^+(x)$ and $(\Phi_{R}\circ u)\EEE^-(x)=u^-(x)$, whenever $R>\max \{|u^+(x)|,|u^-(x)|\}$.
        \end{itemize}
       Moreover, if $u\in {\rm GBV}(\Omega)^k$, then
        \begin{itemize}
        \item [{\rm (d)}] there exists a measure  $|D^cu|\in\mathcal{M}^+(\Omega)$, possibly unbounded, such that $|D^cu|(B)=|D^c(\Phi_R\circ u)|(B)$ for every Borel set $B\subset \{|\widetilde{u}|\leq R\}$ and $R>0$, and $|D^cu|(B)=0$ whenever $B\subset \Omega$ is a Borel set $\sigma$-finite with respect to $\Hd$; in addition, we have 
        \begin{equation*}
             |D^cu|(B)=\lim_{R\to +\infty}|D^c(\Phi_R\circ u)|(B) = \sup_{R>0}|D^c(\Phi_R\circ u)|(B) \EEE\in [0,\infty]
        \end{equation*}
     for all Borel sets $B\subset \Omega$.  
    \end{itemize}
\end{proposition}

\begin{remark}\label{remark:monotonic}
 
If $k=1$, the space $\GBV(\Omega)$ \MMM may \EEE also be characterised by a different type of truncation. Indeed, it can be shown (see, for instance, \cite{Pallara1990}) that $\GBV(\Omega)$ is the space of all 
$u\in L^0(\Omega)$ such that $u^{(n)}\in \BV_{\rm loc}(\Omega)$ for all $n>0$, where we have set  $u^{(n)}\coloneq (u\lor (-n))\land n$. Moreover, properties (b) and (d) of Proposition~\ref{Prop:fine prop GBV} can be strengthened in this case. Indeed, given $n_1,n_2\in\N$ with $n_1<n_2$, we also have 
    \begin{gather*}
        |\nabla u^{(n_1)}|\leq |\nabla u^{(n_2)}| \quad \text{ $\Lb^d$-a.e.\ in $\Omega$ }\quad \text{ and }\quad  |D^cu^{(n_1)}|(B)\leq |D^cu^{(n_2)}|(B),\\
          J_{u^{(n_1)}}\subset J_{u^{(n_2)}} \quad \text{ and } \quad  J^1_{u^{(n_1)}}\subset J^1_{u^{(n_2)}}\quad \text{up to $\Hd$-negligible sets}, \\   \quad |[u^{(n_1)}]|\leq |[u^{(n_2)}]| \quad \text{$\Hd$-a.e.\ on $J_{u^{(n_1)}}$}\EEE, 
    \end{gather*}
    for all Borel sets $B\subset \Omega$.
\end{remark}
The following proposition, \DDD  see  \cite[Proposition~3.9 and its proof]{DonatiGBV}, characterises ${\rm GBV}_\star(O;\Rk)$ by means of the truncation functions $\Phi_R$.
\begin{proposition}\label{prop:truncation}
   Let $O\in\Op(\Omega)$ and let $u\in L^0(\Omega;\Rk)$. Then $u$ belongs to ${\rm GBV}_\star(O;\Rk)$ if and only if  $\Phi_R\circ u\in {\rm BV}(O;\Rk)$ for every $R>0$ and 
\begin{equation*}
 S\coloneq \sup_{R>0}\mathcal{V}(\Phi_R\circ u,O)<+\infty.
\end{equation*}
In this case, $\mathcal{V}(u,O)=S$. 
\end{proposition}  

\begin{remark}\label{remark: troncature}
The space ${\rm GBV}_\star(\Omega;\Rk)$  can also be characterised by truncations as those considered in Remark~\ref{remark:monotonic}. More precisely, a function $u\in L^0(\Omega;\Rk)$ belongs to ${\rm GBV}_\star( \Omega\EEE;\Rk)$ if and only if  $u^{(m)}\in {\rm BV}(\Omega;\Rk)$ for all $m>0$ and     \begin{equation*}
\sup_{m>0}\mathcal{V}( u^{(m)},\Omega)<+\infty,
\end{equation*}
where  we have set $u^{(m)}\coloneq \big((u_1\lor (-m))\land m,\dots,(u_k\lor (-m))\land m\big)$.

Moreover, given   $a,b\in\Rk$ with $a_i\leq b_i$ for all $i\in\{1,\dots,k\}$ and $u\in {\rm GBV}_\star(\Omega;\Rk)$, we  have $v\coloneq (u\lor a)\land b\in {\rm GBV}_\star(\Omega;\Rk)$, $\mathcal{V}(v,O)\leq \mathcal{V}(u,O)$ for all $O\in\Op(\Omega)$, as well as $|D^c v|(O) \le |D^c u|(O)  $, 
   \begin{gather*}
        |Dv|(O\setminus J^r_u)\leq (r \vee 1 ) \mathcal{V}(u,O\setminus J^r_u) \quad \text{ and }|Dv|( J^r_u \cap O)\leq k|b-a|\Hd(J^r_u \cap O),
   \end{gather*}
  and $J_v^r\subset J^r_u$, up to  an \EEE $\Hd$-negligible set,  for all $r\geq 0$. 
\end{remark}

It is possible to associate to every function $u\in \GBVs(O;\Rk)$ a {\it  matrix-valued} measure $D^cu\in\mathcal{M}_b( O\EEE;\Rkd)$ that generalises the Cantor part of the distributional derivative for functions of bounded variation. Note that its total variation measure coincides with the measure introduced in (d) of Proposition~\ref{Prop:fine prop GBV}. For a proof of this result, we refer the reader to \cite[Lemma~2.10]{AlicandroFocardi}   and to  \cite[Theorem~2.7 and Proposition~2.9]{DalToa22}.  
\begin{proposition}\label{Prop:Cantor GBV}
    Let $u\in {\rm GBV}_{\star}(O;\Rk)$. Then there exists a measure $D^cu\in \mathcal{M}_b(O;\Rkd)$ such that $D^cu(B)=D^c(\Phi_R\circ u)(B)$ for every Borel set $B\subset \{|\widetilde{u}|\leq R\}$, and $D^cu(B)=0$ whenever $B\subset O$ is a Borel set $\sigma$-finite with respect to $\Hd$; in addition, we have
        \begin{gather*}
D^cu(B)=\lim_{R\to+\infty}D^c(\Phi_R\circ u)(B)\quad \text{ for all Borel sets }B\subset O,\\
             |D^cu|(B)=\lim_{R\to+\infty}|D^c(\Phi_R\circ u)|(B)= \sup_{R>0}|D^c(\Phi_R\circ u)|(B)\EEE \in [0,\infty) \quad \text{ for all Borel sets }B\subset O.
        \end{gather*}
    \end{proposition}

 We also state a result concerning the Cantor part of  the \EEE composition of ${\rm GBV}_\star$-functions with smooth functions, whose proof can be found in \cite[Proposition~3.13]{DonatiGBV}.
\begin{proposition}\label{prop:Properties of derivatives of GBVstar}
   Let $u\in {\rm GBV}_\star(\Omega;\Rk)$ and $R>0$. Then, 
    \begin{itemize}
        
        \item[{\rm (i)}]  $D^c(\Phi_R \circ u)=\nabla \Phi_R(\widetilde{u})D^cu$ as Radon measures on $\Omega$;
        \item [{\rm (ii)}]  we have 
        \begin{equation*}
            \frac{{\rm d}D^c(\Phi_R\circ u)}{{\rm d}|D^c(\Phi_R\circ u)|}=\frac{{\rm d}D^cu}{{\rm d}|D^cu|} \quad \text{$|D^cu|$-a.e.\ in $\ \widetilde{\Omega}^R_u$},
        \end{equation*} 
      where $\MMM \widetilde{\Omega}_u^R \EEE \coloneq \{x\in \Omega\colon \widetilde{u}(x)\text{ exists and } |\widetilde{u}(x)|\leq R\}$.
        As a consequence, we have
        \begin{equation*}
           \lim_{R\to+\infty} \frac{{\rm d}D^c(\Phi_R\circ u)}{{\rm d}|D^c(\Phi_R\circ u)|}=\frac{{\rm d}D^cu}{{\rm d}|D^cu|} \quad\text{ $|D^cu|$-a.e.\ in $\Omega$}.
        \end{equation*}
    \end{itemize}
\end{proposition}

 If $\Omega$ has Lipschitz boundary and $u\in {\rm GBV}(\Omega;\Rk)$, we let $u_{\partial \Omega}$ denote the trace of $u$ on $\partial \Omega$, which is defined for $\Hd$-a.e.\ $x\in\partial\Omega$ in the sense of approximate limits. We conclude the section by proving that it is  always possible to glue functions in ${\rm GBV}_\star(\Omega;\Rk)$ across a Lipschitz boundary. 
 
\begin{lemma}\label{lemma: Extension of GBV} 

Let $O\in\Op(\Omega)$ with $O\subset \subset \Omega$ and with Lipschitz boundary. Let $u\in {\rm GBV}_\star(O;\Rk)$, $v\in {\rm GBV}_\star(\Omega\setminus \overline{O};\Rk)$, and consider the function defined for $x\in \Omega$ by
\begin{equation*}
    w(x)\coloneq \begin{cases}
        u(x)&\text{if $x\in O$},\\
        v(x)& \text{if $x\in\Omega\setminus \overline{O}$}.
    \end{cases}
\end{equation*}
Then $w\in {\rm GBV}_{\star}(\Omega;\Rk)$. Moreover, $|[w]|\coloneq |u_{\partial O}-v_{\partial O}|$  $\Hd$-a.e.\ on $\partial O$.
\EEE
\end{lemma}
\begin{proof}
     It is immediate to check that $w\in {\rm GBV}(\Omega;\Rk)$ and that $\mathcal{V}(w,\Omega\setminus \overline{O})+ \MMM \mathcal{V}(w,O) \EEE <+\infty$. Since 
    \begin{equation*}
      \BBB  \int_{J_w\cap \partial O}|[w]|\land 1\,{\rm d}\Hd \EEE \leq \Hd(\partial O)<+\infty,
    \end{equation*}
  we get  $\mathcal{V}(w,\Omega)<+\infty$, and hence $w\in {\rm GBV}_\star(\Omega;\Rk)$.
    \EEE
\end{proof}

\section{Slicing properties of generalised functions of bounded variation}\label{sec:slicing}

In this section, we present some new results concerning fine properties of  ${\rm GBV}_\star(\Omega;\Rk)$-functions and use
these properties to provide a  crucial result  that is needed \EEE in the proof of the Fundamental Estimate of Lemma~\ref{lemma: Fundamental Estimate} (see Lemma~\ref{products with smooth functions}). In particular, we show that the space $\GBVs(\Omega;\Rk)$ can  be characterised by means of one-dimensional slicing. This characterisation shows that the space ${\rm GBV}_\star(\Omega;\Rk)$ is the analogue to Dal Maso's ${\rm GBD}(\Omega)$ (see \cite{DMJems}) in the context of functions of bounded variation.
 
In order to present the characterisation mentioned above,
we  set some notation and recall some properties of one-dimensional slices of ${\rm GBV}$-functions. Given $\xi\in\Sd$, we denote the hyperplane orthogonal to $\xi$ and passing through the origin by $\Pi^\xi\coloneq \{y\in\Rd\colon\,y\cdot\xi=0\}$, and let $\pi^\xi\colon\Omega\to \Pi^\xi$ be the orthogonal projection onto $\Pi^\xi$. Given a set $E\subset \Omega$, its one-dimensional slice in direction $\xi$ with base point $y\in\Pi^\xi$ is defined as $E^\xi_y\coloneq \{t\in\R\colon\,y+t\xi\in E\}.$ For $k\in\N$,  $u\in L^0(\Omega;\R^k)$, $\xi\in\Sd$, and $y\in\pi^\xi(\Omega)$, we let $u^\xi_y  \colon \Omega^\xi_y \to \R^k \EEE$ be the function defined  by $u^\xi_y(t)\coloneq u(y+t\xi)$ \MMM for \EEE $t\in\Omega^\xi_y$. 
We recall that, if $u\in{\rm GBV}(\Omega)^k$, then for every $\xi\in\Sd$ and for $\Hd$-a.e.\ $y\in\pi^\xi(\Omega)$    we have \EEE $u^\xi_y\in {\rm GBV}(\Omega^\xi_y)^k$ and   (recall \eqref{def Jru}) \EEE
\begin{gather}\label{slice GBV nabla}
    \nabla u^\xi_y=\nabla u(y+t\xi)\xi \quad \text{$\Lb^1$-a.e.\ in $\Omega^\xi_y$,}\\\label{slice GBV jump 3}
      J_{u^\xi_y}=(J_u)^\xi_y\quad \text{ and }\quad J^r_{u^\xi_y}= (J^r_u)^\xi_y\quad \text{ for all $r\geq 0$,}\\\label{slice GBV jumps}
     (u^{\pm})^\xi_y(t)=(u^{\xi}_y)^\pm(t) \quad \text{ for all $t\in (J_u)^\xi_y=J_{u^\xi_y}$},
\end{gather}
where the normals at $J_u$ and at $J_{u^\xi_y}$ are oriented in such a way that $\nu_u\cdot\xi\geq 0$ and $\nu_{u^\xi_y}=1$. For a proof of these properties, we refer the reader to \cite[Proposition~4.35]{AFP}.

The following lemma  shows that, for a function $u\in{\rm GBV}(\Omega)$, it is always possible to \MMM  bound  the integral  of  $\mathcal{V}(u^\xi_y,\Omega^\xi_y)$  with respect to $y\in \pi^\xi(\Omega) $ from above by   $\mathcal{V}(u,\Omega)$.  \EEE
\begin{lemma}\label{slice V}
    Let $u\in {\rm GBV}(\Omega)$ and $\xi\in\Sd$. Then 
    \begin{equation}\label{claim bound slice}
\int_{\pi^\xi(\Omega)} \big(|Du^\xi_y|(\Omega^\xi_y\setminus J^1_{u^\xi_y})+\mathcal{H}^0(\Omega^\xi_y\cap J^1_{u^\xi_y})\big)\,{\rm d}\Hd(y)\leq \mathcal{V}(u,\Omega).
    \end{equation}
Moreover, if $u\in {\rm GBV}_\star(\Omega)$, we have  
\begin{align}\label{slice bulk}
      \int_{B}(\nabla u) \xi\,{\rm d}x&=\int_{\pi^\xi( B)}\int_{B^\xi_y}\nabla u^\xi_y(t)\,{\rm d}t\,{\rm d}\Hd(y),\\\label{slice cantor}
        D^cu(B)\xi&=\int_{\pi^\xi(B)}D^cu^\xi_y(B^\xi_y)\,{\rm d}\Hd(y),\\
        \label{slice jumps}
        \int_{B\cap (J_u\setminus J^1_u)}[u] (\nu_u\cdot\xi)\,{\rm d}\Hd&=\int_{\pi^\xi(B)}Du^\xi_{y}\big((B\cap (J_u\setminus J^1_u))^\xi_y\big)\,{\rm d}\Hd(y)
\end{align}
 for every Borel set  $B\subset \Omega$ and $\xi\in\Sd$.
\end{lemma}
\begin{proof}
If $u\notin  {\rm GBV}_\star(\Omega)$  by the definition \eqref{Def V BV} of $\mathcal{V}$ we have $\mathcal{V}(u,\Omega)=+\infty$, \EEE so that we may assume that $u\in {\rm GBV}_\star(\Omega)$.   Let $m\in\N$ and consider the function $u^{(m)}\in {\rm BV}(\Omega)$   (see Remark~\ref{remark:monotonic})\EEE.   By the slicing properties of functions of bounded variation (see \cite[Theorems 3.107 and 3.108]{AFP}) we have  
    \begin{align*}
    \int_{\Omega}|\nabla u^{(m)}|\,{\rm d}x&\geq\int_{\pi^\xi(\Omega)}\Big(\int_{\Omega^\xi_y}|\nabla(u^{(m)})^\xi_y|\,{\rm d}t\Big){\rm d}\Hd(y),
       \\ |D^cu^{(m)}|(\Omega)&\geq \int_{\pi^\xi(\Omega)}|D^c(u^{(m)})^\xi_y|(\Omega^\xi_y)\,{\rm d}\Hd(y),
    \end{align*} 
    which, thanks to Remark~\ref{remark:monotonic},  yields   
    \begin{align}\label{modulo 1}
    \int_{\Omega}|\nabla u|\,{\rm d}x&\geq\int_{\pi^\xi(\Omega)}\Big(\int_{\Omega^\xi_y}|\nabla(u^{(m)})^\xi_y|\,{\rm d}t\Big){\rm d}\Hd(y),
       \\ 
       \label{modulo 2}
       |D^cu|(\Omega)&\geq \int_{\pi^\xi(\Omega)}|D^c(u^{(m)})^\xi_y|(\Omega^\xi_y)\,{\rm d}\Hd(y).
    \end{align} 
    As $u\in {\rm GBV}(\Omega)$, for $\Hd$-a.e. $y\in \pi^\xi(\Omega)$ we have $u^\xi_y\in{\rm GBV}(\Omega^\xi_y)$.  Thus, \EEE applying Remark~\ref{remark:monotonic} to the slices $u^\xi_y$, we obtain 
    \begin{align*}
\int_{\Omega^\xi_y}|\nabla(u^{(m)})^\xi_y|\,{\rm d}t&\to \int_{\Omega^\xi_y}|\nabla u^\xi_y|\,{\rm d}t  \quad \text{ for $\Hd$-a.e.\ $y\in \pi^\xi(\Omega)$ as $m\to+\infty$}, \\
|D^c(u^{(m)})^\xi_y|(B^\xi_y)&\to |D^cu^\xi_y|(B^\xi_y)\quad \text{ for \EEE $\Hd$-a.e.\ $  y\in\EEE\pi^\xi(\Omega)$ as $m\to+\infty$},
    \end{align*}
    with both convergence being monotone. Therefore, by the Monotone Convergence Theorem and \eqref{modulo 1}--\eqref{modulo 2} we infer
    \begin{align}\label{colla}
        \int_{\Omega}|\nabla u|\,{\rm d}x&\geq\int_{\pi^\xi(\Omega)}\Big(\int_{\Omega^\xi_y}|\nabla u^\xi_y|\,{\rm d}t\Big){\rm d}\Hd(y),
       \\ \label{colla 2}
       |D^cu|(\Omega)&\geq \int_{\pi^\xi(\Omega)}|D^cu^\xi_y|(\Omega^\xi_y)\,{\rm d}\Hd(y).
    \end{align}
    As for the jump part, we use the Area Formula (see  \cite[(2.47)]{AFP}\EEE ), \eqref{slice GBV jumps}, and \eqref{slice GBV jump 3} to get
\begin{align}\notag 
\int_{J_{u^{(m)}}}|[u^{(m)}]|\land 1
\, {\rm d}\Hd
&\geq \int_{\pi^{\xi}(\Omega)}\Big(\int_{J_{(u^{(m)})^\xi_y}}|[(u^{(m)})^{\xi}_y](t)|\land 1\,{\rm d}\mathcal{H}^0(t)\Big){\rm d}\Hd(y)\\\notag
&=\int_{\pi^{\xi}(\Omega)}\Big(|D^j(u^{(m)})^\xi_y|(\Omega^\xi_y\setminus J^1_{(u^{(m)})^\xi_y})+\mathcal{H}^0(\Omega^\xi_y\cap J^1_{(u^{(m)})^\xi_y}) \Big){\rm d}\Hd(y).
\end{align}
 Using again Remark \ref{remark:monotonic} and the Monotone Convergence Theorem, we obtain 
\begin{equation}\label{bound jump}
    \int_{J_u}|[u]|\land 1
\, {\rm d}\Hd\geq \int_{\pi^{\xi}(\Omega)}\Big(|D^ju^\xi_y|(\Omega^\xi_y\setminus J^1_{u^\xi_y})+\mathcal{H}^0(\Omega^\xi_y\cap J^1_{u^\xi_y}) \Big)\, {\rm d}\Hd(y).
\end{equation}
\EEE 
Putting together \eqref{colla}--\eqref{bound jump}, we  obtain \eqref{claim bound slice}.

 To prove \eqref{slice bulk}--\eqref{slice jumps}, we use  \cite[Theorems 3.107 and 3.108]{AFP} to get
\begin{align}\label{slice bulk proof}
      \int_{B}(\nabla u^{(m)}) \xi\,{\rm d}x&=\int_{\pi^\xi( B)}\Big(\int_{B^\xi_y}\nabla(u^{(m)} )^\xi_y(t)\,{\rm d}t\Big){\rm d}\Hd(y),\\\label{slice cantor proof}
        D^cu^{(m)}(B)\xi&=\int_{\pi^\xi(B)}D^c(u^{(m)})^\xi_y(B^\xi_y)\,{\rm d}\Hd(y).
\end{align}
Thanks to \eqref{claim bound slice}, for $\Hd$-a.e.\ $y\in \pi^\xi(\Omega)$ the function $u^\xi_y$ belongs to ${\rm BV}_{\rm loc}(\Omega^\xi_y)$, so that   
\begin{align*}
   \int_{B^\xi_y}\nabla (u^{(m)})\EEE^\xi_y(t)\,{\rm d}t &\to  \int_{B^\xi_y}\nabla u^\xi_y(t)\,{\rm d}t\quad \text{ for $\Hd$-a.e.\ $y\in \pi^\xi(\Omega)$ as $m\to+\infty$},\\
 D^c(u^{(m)})\EEE^\xi_y(B^\xi_y)&\to D^cu^\xi_y(B^\xi_y)\quad \text{for  $\Hd$-a.e.\  $y\in\pi^\xi(\Omega)$ as $m\to+\infty$.}
\end{align*}
 Thus, \EEE  we may let $m\to +\infty$ in \eqref{slice bulk proof} and \eqref{slice cantor proof}. Then,  by    Propositions \ref{Prop:fine prop GBV} and \ref{Prop:Cantor GBV},  and  the Dominated Convergence Theorem we obtain \eqref{slice bulk} and \eqref{slice cantor}.

As for \eqref{slice jumps}, it is enough to apply the Area Formula  and \MMM to \EEE  use \eqref{slice GBV jump 3} and \eqref{slice GBV jumps}.
\end{proof}

The following proposition characterises ${\rm GBV}_\star(\Omega;\Rk)$ by means of one-dimensional slicing.
\begin{proposition}\label{prop:characterisation by slicing}
    Let $u\in L^0(\Omega;\Rk)$. Then $u\in{\rm GBV}_\star(\Omega;\Rk)$ if and only if for every $\xi\in\Sd$ the  following two conditions hold:
    \begin{align}\label{slice is BV characterisation} 
        & u^\xi_y\in{\rm BV}_{\rm loc}(\Omega^\xi_y;\Rk) \quad \text{ for $\Hd$-a.e.\ $y\in\pi^\xi(\Omega)$},\\
\label{boundedness slice}&\int_{\pi^\xi(\Omega)} \big(|Du^\xi_y|(\Omega^\xi_y\setminus J^1_{u^\xi_y})+\mathcal{H}^0(\Omega^\xi_y\cap J^1_{u^\xi_y})\big)\,{\rm d}\Hd(y)<+\infty.
    \end{align}
\end{proposition}
\begin{proof}
It is enough to prove the statement in the case $k=1$.

 The fact \EEE that $u\in {\rm GBV}_\star(\Omega)$ implies \eqref{slice is BV characterisation} and \eqref{boundedness slice} is a simple consequence of Lemma~\ref{slice V} and the fact that in dimension $1$ the space ${\rm GBV}_\star$  is contained in \EEE  ${\rm BV}_{\rm loc}$ (see \cite[Remark~3.2]{DalToa22}).   

We now prove that \eqref{slice is BV characterisation} and \eqref{boundedness slice} imply  $u\in {\rm GBV}_\star(\Omega)$.
 To this end, let us fix  $\xi\in\Sd$. By \eqref{slice is BV characterisation}, there exists a Borel set $N^1_\xi\subset \pi^\xi(\Omega)$, with $\Hd(N^1_\xi)=0$, such that $u^\xi_y\in {\rm BV}_{\rm loc}(\Omega^\xi_y)$ for all $y\in\pi^\xi(\Omega)\setminus N^1_\xi$. Let us fix $m>0$.  For every $y\in\pi^\xi(\Omega)\setminus N^1_\xi$, by the Chain Rule for functions of bounded variation  (see \cite[Theorem 3.99]{AFP}) \EEE   we have $(u^{(m)})^\xi_y\in{\rm BV}(\Omega^\xi_y)$ and 
\begin{equation}\label{chain rule slciing}
D (u^{(m)})\EEE^\xi_y=\chi_{\{|\widetilde{u}^\xi_y|\leq m\}}(D^au^\xi_y+D^cu^\xi_y)+\big(((u^\xi_y)^+)^{(m)}-((u^\xi_y)^-)^{(m)}\big)\mathcal{H}^0\mres J_{u^\xi_y}
\end{equation}
as Borel measures on $\Omega^\xi_y$.
In particular, 
\begin{equation*}
|D(u^{(m)})\EEE^\xi_y|(\Omega^\xi_y)\leq |D^au^\xi_y|(\Omega^\xi_y)+|D^cu^\xi_y|(\Omega^\xi_y)+  |D^{j} u^\xi_y|(J_{u^\xi_y}\setminus J^1_{u^\xi_y}) \EEE +2m\mathcal{H}^0 ( J^1_{u^\xi_y}) \EEE
\end{equation*}
  for \EEE every $y\in\pi^\xi(\Omega)\setminus N^1_\xi$ .
Combining the previous inequality with \eqref{boundedness slice}, 
we obtain
\begin{equation*}
\int_{\pi^\xi(\Omega)}|D (u^{(m)})\EEE^\xi_y|(\Omega^\xi_y)\,{\rm d} \Hd(y)<+\infty.
\end{equation*}
By the arbitrariness of $\xi\in\Sd$ and  \cite[Remark~3.104]{AFP}, this implies that $ u^{(m)}\EEE\in{\rm BV}(\Omega)$, whence  $u\in {\rm GBV}(\Omega)$.

To conclude that $u\in {\rm GBV}_\star(\Omega)$, it is then enough to show  (see Remark \ref{remark: troncature}) \EEE that
\begin{gather}
\label{jump finite}
    \int_{ J_{u^{(m)}}\EEE}|[u^{(m)}]|\land 1\, {\rm d} \Hd \le C_1,\\
\label{diffuse slice}
    \int_{\Omega}|(\nabla u^{(m)} )e_i|\,{\rm d}x+|(D^c(u^{(m)})e_i|(\Omega)\leq C_2 \quad \text{ for all $i\in\{1,\dots,d\}$},
\end{gather}
  for constants $C_1$, $C_2>0$ independent of $m$.  
  
Let $\xi\in\Sd$,  let $N^1_\xi\subset \pi^\xi(\Omega)$ be as above, and let $ N^2_\xi \EEE\subset \pi^\xi(\Omega)$ be a Borel set, with $\Hd(N_\xi^2)=0$, such that \eqref{slice GBV jump 3} and \eqref{slice GBV jumps} hold for every $y\in \pi^\xi(\Omega)\setminus N^2_\xi$.  Resorting  to \eqref{chain rule slciing},   we get
\begin{gather*}
J_{ (u^{(m)})\EEE^\xi_y}\subset J_{u^\xi_y}=(J_u)^\xi_y\quad \text{ and }J^1_{ (u^{(m)})\EEE^\xi_y}\subset J^1_{u^\xi_y}=(J^1_u)^\xi_y,
\end{gather*}
and  
\begin{align}\notag
    |D (u^{(m)})\EEE^\xi_y|\big(B^\xi_y \setminus J^1_{ (u^{(m)})\EEE^\xi_y} \big)&\le |D (u^{(m)})\EEE^\xi_y|\big(B^\xi_y \setminus J^1_{u^\xi_y}\big)+\mathcal{H}^0\big(B^\xi_y\cap (J^1_{u^\xi_y}\setminus J^1_{ (u^{(m)})\EEE^\xi_y})\big)\\\label{cgames}
    &\leq |Du^\xi_y|\big(B^\xi_y \setminus J^1_{u^\xi_y} \big)+\mathcal{H}^0\big(B^\xi_y\cap(J^1_{u^\xi_y}\setminus J^1_{ (u^{(m)})\EEE^\xi_y})\big)
\end{align}
for all Borel sets $B\subset \Omega$  \MMM and \EEE for every $y\in \pi^\xi(\Omega)\setminus  (N^1_\xi \cup N^2_\xi)$. \EEE 

We now prove \eqref{jump finite}.
We consider a finite collection of vectors  $\{\xi_j\}_{j=1}^N\subset \Sd$ with the property that the sets  $S_j\coloneq \{\xi\in\Sd\colon \, |\xi-\xi_j|\leq 1/2\}$  cover  $\Sd$. We set  $B_j\coloneq \{x\in J_u\colon\, \nu_u(x)\in S_j\}$ and note that $J_u=\bigcup_{j=1}^NB_j$.  Note also that, if $\xi\in S_j$ for some $j\in\{1,\dots,N\}$, then $\xi\cdot\xi_j=1-|\xi-\xi_j|^2/2\geq7/8$, so that
\begin{equation}\label{normal scalara product}
\nu_u(x)\cdot\xi_j\geq 7/8\quad \text{ for all $x\in B_j$.}    
\end{equation}
By the Area Formula (see  \cite[(2.47)]{AFP}), \eqref{slice GBV jumps},     \eqref{cgames}, and \eqref{normal scalara product}, for all $j\in\{1,\dots,N\}$ we have  
\begin{align*}
\frac{7}{8}\int_{B_j}|[ u^{(m)}]|\land 1\, {\rm d}\Hd
&\leq\int_{B_j}(|[u^{(m)}]|\land 1)(\nu_u\cdot\xi_j)\, {\rm d}\Hd(y)
\\
&=\int_{\pi^{\xi_j}(B_j)}\Big( \MMM \int_{(B_j)^{\xi_j}_y} \EEE |[ (u^{(m)})\EEE^{\xi_j}_y](t)|\land 1\,{\rm d}\mathcal{H}^0(t)\Big)\, {\rm d} \Hd(y)\\&=\int_{\pi^{\xi_j}(B_j)}\Big(|D (u^{(m)})\EEE^{\xi_j}_y|\big((B_j)^{\xi_j}_y\setminus J^1_{ (u^{(m)})\EEE^{\xi_j}_y}\big)+\mathcal{H}^0\big((B_j)^{\xi_j}_y\cap J^1_{ (u^{(m)})\EEE^{\xi_j}_y}\big)\Big) \,  {\rm d}\Hd(y)\\
&\leq \int_{\pi^{\xi_j}(B_j)}\Big(|Du^{\xi_j}_y|\big((B_j)^{\xi_j}_y\setminus J^1_{u^{\xi_j}_y}\big)+\mathcal{H}^0\big((B_j)^{\xi_j}_y\cap J^1_{u^{\xi_j}_y}\big)\Big) \, {\rm d}\Hd(y).
\end{align*}
Summing over $j\in\{1,\dots,N\}$ the previous inequality,    for all $y\in \pi^{\xi_j}(\Omega)\setminus (N^1_{\xi_j}\cup N^2_{\xi_j})$ we find by the Area Formula  
\begin{equation*}
    \int_{J_u}|[u^{(m)}]|\land 1\, {\rm d}\Hd\leq \frac{8}{7}\sum_{j=1}^N\int_{\pi^{\xi_j}(B_j)}\Big(|Du^{\xi_j}_y|\big((\MMM B_j\EEE)^{\xi_j}_y\setminus J^1_{u^{\xi_j}_y}\big)+\mathcal{H}^0\big((\MMM B_j\EEE)^{\xi_j}_y\cap J^1_{u^{\xi_j}_y}\big)\Big) \, {\rm d}\Hd(y)=:C_1.
\end{equation*}
From \eqref{boundedness slice} it  follows $C_1<+\infty$, which proves \eqref{jump finite}.  

To prove \eqref{diffuse slice}, we use the slicing properties of functions of bounded variation (see \cite[Theorem~3.108]{AFP}) to get  
\begin{equation*}
    \int_{\Omega}|(\nabla u^{(m)})e_i|\,{\rm d}x+|(D^c u^{(m)})e_i|(\Omega)=\int_{\pi^{e_i}(\Omega)}\Big(\int_{\Omega^{e_i}_y}|\nabla (u^{(m)})^{e_i}_y|\,{\rm d}t+|D^c(u^{(m)})^{e_i}_y|(\Omega^{e_i}_y)\Big)\, {\rm d}\Hd(y).
\end{equation*}
In view of \eqref{boundedness slice}, of \eqref{cgames} applied with $B=\Omega\setminus J_u$,  and of \eqref{slice GBV jump 3}, \EEE the previous equality proves \eqref{diffuse slice} with  \begin{equation*}
    C_2\coloneq \sum_{i=1}^d\int_{\pi^{e_i}(\Omega)}\Big(\int_{\Omega^{e_i}_y}|\nabla  u^{e_i}_y|\,{\rm d}t+|D^c u^{e_i}_y|(\Omega^{e_i}_y)\Big){\rm d}\Hd(y).
\end{equation*}
This proves that $u\in {\rm GBV}_\star(\Omega)$ and concludes the proof.
\end{proof}

\begin{remark}\label{remark: 1d slices}
    If $\Omega$ has Lipschitz boundary and $u\in {\rm GBV}(\Omega;\Rk)$, then  we may replace ${\rm BV}_{\rm loc}(\Omega^\xi_y;\Rk)$ in  \eqref{slice is BV characterisation} by ${\rm BV}(\Omega^\xi_y;\Rk)$. Indeed, let  $y\in \pi^\xi(\Omega)$ be such that
\begin{align}\label{condicondi}
u^\xi_y\in {\rm BV}_{\rm loc}(\Omega^\xi_y;\Rk), \ \  \mathcal{H}^0(\partial \Omega^\xi_y)<+\infty, \ \  \text{and} \ \ |Du^\xi_y|(\Omega^\xi_y\setminus J^1_{u^\xi_y})+\mathcal{H}^0(\Omega^\xi_y\cap J^1_{u^\xi_y})<+\infty.
\end{align}
  Since $\mathcal{H}^0(J^1_{u^\xi_y})<+\infty$, it follows immediately that $|Du^\xi_y|(\Omega^\xi_y)<+\infty$. Condition   $\mathcal{H}^0(\partial \Omega^\xi_y)<+\infty$ implies that $\Omega^\xi_y$ is a finite union of open intervals, so that by the Fundamental Theorem of Calculus for one-dimensional {\rm BV} functions we also get $u^\xi_y\in L^1(\Omega^\xi_y;\Rk)$. As   the conditions in \eqref{condicondi} \EEE are satisfied by $\Hd$-a.e.\ $y\in \pi^\xi(\Omega)$,   the claim is proved.
\end{remark}

We now show that the measure $D^cu$ enjoys some useful disintegration properties.
\begin{lemma}\label{lemma:disintegration of Dcu}
  Let $u\in {\rm GBV}_\star(\Omega;\Rk)$, let $\varphi\colon\Omega\to \R$ be a bounded Borel function, and let $B\subset \Omega$ be a Borel set. Then  for every $\xi\in\Sd$ we have
\begin{equation}\label{eq:disintegration of Dcu}
    \Big(\int_{B}\varphi\,{\rm d}D^cu\Big)\xi=\int_{\pi^\xi(\Omega)}\Big(\int_{B^\xi_y}\varphi^\xi_y\,{\rm d}D^cu^\xi_y\Big) \, {\rm d}\Hd(y).
\end{equation}
\end{lemma}
\begin{proof}
     The \MMM case \EEE $\varphi=\chi_{ E\EEE}$ for some Borel set $ E \EEE\subset \Rd$ is exactly \eqref{slice cantor}. This result is clearly extended to simple bounded Borel functions by additivity.  The general case can be obtained by approximating each bounded Borel function by a sequence of uniformly bounded simple Borel functions, taking into account \eqref{boundedness slice}, in order to be able to use the Dominated Convergence Theorem on the right-hand side of \eqref{eq:disintegration of Dcu}.
\end{proof}
The following result, which will be used in Section \ref{section: Global Method}, shows that the product of a smooth function $\varphi$ with a function $u\in {\rm GBV}_\star(\Omega;\Rk)$ still belongs to $\GBVs(\Omega;\Rk)$, provided that $u$ satisfies some additional integrability conditions.
\begin{lemma}\label{products with smooth functions}
Suppose that $\Omega$ has Lipschitz boundary.   Let $\varphi\in C^1_c(\overline{\Omega};[0,1])$ and let  $u\in \GBVs(\Omega;\Rk)$. 
    Assume  that $u\in L^1( K \EEE ;\Rk)$, where  $K\coloneq \lbrace y \in \Omega \colon \,  0 < \varphi < 1 \rbrace$.   
    Then $v\coloneq \varphi u\in{\rm GBV}_{\star}(\Omega;\Rk)$
    and 
    \begin{gather}\label{Liebnitz bulk}
        \nabla v = u\otimes\nabla \varphi    + \varphi \nabla u\quad \text{ $\Lb^d$-a.e.\  in $\Omega$},\\\label{Liebnitz Cantor}
    D^cv=  \varphi D^cu\quad \text{ as Borel measures on $\Omega$},\\\label{Liebnitz Jump}
     J_v\subset J_u\quad \text{up to \MMM an \EEE $\Hd$-negligible set } \text{ and } [v](x)=\varphi(x)[u](x)\quad \text{ for $\Hd$-a.e.\ $x\in J_v$.}
    \end{gather}
    \end{lemma}
    
    \begin{proof}
    Let us fix $\xi\in\Sd$.  In view of Remark~\ref{remark: 1d slices}, we have  $u^\xi_y\in {\rm BV}(\Omega^\xi_y)$ for $\Hd$-a.e.\ $y\in\ \pi^\xi(\Omega)$. By the Leibniz rule for one-dimensional ${\rm BV}$-functions we then have that $v^\xi_y\in {\rm BV}(\Omega^\xi_y)$ and  
\begin{gather}\label{slice bulk liebnitz}
        \nabla v^\xi_y= \varphi^\xi_y \nabla u^\xi_y+(\nabla \varphi^\xi_y) u^\xi_y \quad \text{ $\Lb^1$-a.e.\ in $\Omega^\xi_y$}, \\ \label{slice cantor liebnitz}
        D^cv^\xi_y= \varphi^\xi_y D^c u^\xi_y \quad \text{  as Borel measures on $\Omega^\xi_y$\EEE},\\\label{liebnitz jump slice}  
        J_{v^\xi_y}=J_{u^\xi_y}\cap \{t\in \Omega^\xi_y\colon\,\varphi^\xi_y(t)\neq 0\}  \quad\text{ and }\quad (v^\xi_y)^\pm=\varphi^\xi_y(u^\xi_y)^\pm \quad \text{ on }J_{u^\xi_y}
    \end{gather}
    for $\Hd$-a.e.\ $y\in\pi^\xi(\Omega)$.
    In particular, from the previous equality and the fact that $0\leq \varphi\leq 1$, it follows that 
    \begin{equation*}
  |[v^\xi_y]|\leq |[u^\xi_y]|\quad  \text{ on ${J}_{u^\xi_y}$} \quad \text{ and } J^1_{{v}^\xi_y}\subset J^{1}_{{u}^\xi_y}
  \end{equation*}
for $\Hd$-a.e.\ $y\in\pi^\xi(\Omega)$. This inclusion then implies  
\begin{gather}\label{jump 1 on v slice}
    \hspace{-0,4 cm}\int_{\pi^\xi(\Omega)}|Dv^\xi_y|(J_{v^\xi_y} \setminus J^1_{v^\xi_y})\,{\rm d}\Hd(y)\leq  \int_{\pi^\xi(\Omega)}\big(|Du^\xi_y|(J_{u^\xi_y} \setminus J^1_{u^\xi_y})+\mathcal{H}^0(J^1_{u^\xi_y} \setminus J^1_{v^\xi_y})\big)\, {\rm d}\Hd(y)<+\infty,\\
\int_{\pi^\xi(\Omega)}\mathcal{H}^0(J^1_{v^\xi_y})\,{\rm d}\Hd(y)\leq \int_{\pi^\xi(\Omega)}\mathcal{H}^0(J^1_{u^\xi_y})\,{\rm d}\Hd(y)<+\infty.
\end{gather}
In view of  Proposition~\ref{prop:characterisation by slicing}, we get that  $u$ satisfies \eqref{boundedness slice}. Then, by  the fact that $u\in L^1( K;\Rk)$ and Fubini's Theorem we find
    \begin{align}
&\hspace{-0.3 cm}\int_{\pi^\xi(\Omega)}\Big(\int_{\Omega^\xi_y}|\varphi^\xi_y\nabla u^\xi_y|\,{\rm d}t\Big)\,{\rm d}\Hd(y)\leq \int_{\pi^\xi(\Omega)}\Big(\int_{\Omega^\xi_y}|\nabla u^\xi_y\,|{\rm d}t\Big)\,{\rm d}\Hd(y)<+\infty,\\ 
& \hspace{-0.3 cm}
\int_{\pi^\xi(\Omega)}\!\!\!\Big(\int_{\Omega^\xi_y}|(\nabla \varphi^\xi_y) u^\xi_y|\,{\rm d}t\Big)\,{\rm d}\Hd(y)\leq \|\nabla \varphi\|_{L^\infty(\Omega)}\int_{\pi^\xi(\Omega)}\!\!\!\Big(\int_{K^\xi_y}| u^\xi_y|\,{\rm d}t\Big)\,{\rm d}\Hd(y)<+\infty,\\
    \label{cantor v slice}&\hspace{-0.3 cm}\int_{\pi^\xi(\Omega)}\Big(\int_{\Omega^\xi_y}|\varphi^\xi_y|\,{\rm d}|D^cu^\xi_y|\Big)\,{\rm d}\Hd(y)\leq \int_{\pi^\xi(\Omega)}|D^cu^\xi_y|(\Omega^\xi_y)\,{\rm d}\Hd(y)<+\infty  .
    \end{align}
  Recalling \eqref{slice bulk liebnitz} and \eqref{slice cantor liebnitz}, \EEE combining \eqref{jump 1 on v slice}--\eqref{cantor v slice}, we conclude that 
 \begin{equation*}
      \int_{\pi^\xi(\Omega)} \big(|Dv^\xi_y|(\Omega^\xi_y\setminus J^1_{v^\xi_y})+\mathcal{H}^0(\Omega^\xi_y\cap J^1_{v^\xi_y})\big)\,{\rm d}\Hd(y)<+\infty. \end{equation*}
      As $\xi\in\Sd$ was arbitrary, Proposition~\ref{prop:characterisation by slicing} ensures that $v\in {\rm GBV}_\star(\Omega;\Rk)\cap L^1(  K;\Rk).$
  
 Equation \EEE \eqref{Liebnitz Jump} follows immediately from \eqref{slice GBV jumps} and \eqref{liebnitz jump slice}. 
    To prove \eqref{Liebnitz bulk}, it is enough to combine \eqref{slice GBV nabla} and \eqref{slice bulk liebnitz}.
    As for \eqref{Liebnitz Cantor}, we note that by \eqref{slice cantor} applied for $v$  in place of  $u$,  \MMM \eqref{slice cantor liebnitz}, \EEE  and Lemma~\ref{lemma:disintegration of Dcu} we have
     \begin{equation*}
D^cv(B)\xi=\int_{\pi^\xi(\Omega)}D^cv^\xi_y(B^\xi_y)\,{\rm d}\Hd(y)=\int_{\pi^\xi(\Omega)}\Big(\int_{ B^\xi_y}\varphi^\xi_y\,{\rm d}D u^\xi_y \Big)\,{\rm d}\Hd(y)=\Big(\int_{B}\varphi\,{\rm d}D^cu\Big)\xi
     \end{equation*}
for every $\xi\in\Sd$ and every Borel set $B\subset \Omega$. Since the previous equality implies \eqref{Liebnitz Cantor}, the proof is complete.
    \end{proof}

\section{Poincaré inequality for functions with small jump set}\label{sec: Poincarè}

\noindent  This section is devoted \MMM to a \EEE Poincaré-type inequality for functions in ${\rm GBV}_\star(\Omega;\Rk)$,
which will be a key tool for a series of results that we will  present \EEE  in  Appendix \EEE \ref{sec:blow up}. \MMM It is also employed in the proof of Lemma~\ref{Lemma delta=tilde}. \EEE The inequality \MMM is \EEE visibly 
reminiscent of the one  proved by De Giorgi, Carriero, \& Leaci \cite[Theorem~3.1]{CarrieroLeaci}   for functions in ${\rm SBV}(\Omega)$ with small jump set. Indeed, our proof follows closely the argument they devised.

We begin by setting some notation. Throughout this section, we make the further assumption   that $d \ge 2$ and  $\Omega\subset \Rd$ is connected \MMM with \EEE Lipschitz boundary.
Given $u\in {\rm GBV}_\star(\Omega)$, we define
\begin{equation}\label{def ustar}
    u_\star(s,\Omega)\coloneq \inf\big\{t\in[-\infty,+\infty]\colon \, \Lb^d(\{u<t\})\geq  s \big\}
\end{equation}
for all $s\in (0,\Lb^d(\Omega))$. We say that  $m\in\R$ is  a {\it median} of $u$ on $\Omega$ if 
\begin{align}\label{eq: medain}
    \Lb^d(\{u<t\})&\leq   \Lb^d(\Omega)/2\,\,\,\text{ for all $t<m$,}\notag\\  
    \Lb^d(\{u>t\})&\leq \Lb^d(\Omega)/2\,\,\,\text{for all $t>m$}.
\end{align}
It follows immediately from \eqref{def ustar} that $u_\star(\Lb^d(\Omega)/2,\Omega)$ is a median of $u$ on $\Omega$.
In the vector-valued case $u\in {\rm GBV}_\star(\Omega;\Rk)$ with $k>1$, we say that a vector $m\in\Rk$  is a median of $u$ on $\Omega$ if $m_i$ is a median of $u_i$ on $\Omega$ for all $i\in\{1,\dots,k\}$.
For future use, we  report here the statement of the classical relative isoperimetric inequality (see, for instance, \cite[(3.43)]{AFP}).  Recall \eqref{def blow up set}.

\begin{proposition}\label{relative isoperimetric}
    There exists a constant $\gamma=\gamma(\Omega)>0$ such that for  every $\rho>0$, every $x\in\Rd$,  and every set $E\subset\R^d$ \MMM with \EEE $E_{\rho,x} \subset \Omega_{\rho,x\EEE}$ \MMM having \EEE finite perimeter in $\Omega_{\rho,x\EEE}$   it holds  
    \begin{equation*}
        \min\Big\{ \Lb^d(E_{\rho,x}\cap \Omega_{\rho,x}), \Lb^d(\Omega_{\rho,x}\setminus E_{\rho,x})   \Big\}\leq    \big( \gamma\Hd\big(\partial^*E_{\rho,x}\cap \Omega_{\rho,x}\big)\big)^{d/(d-1)}.  
    \end{equation*}
\end{proposition}
 \begin{remark}\label{remark: constant isoperimetric}
The isoperimetric constant $\gamma=\gamma(\Omega)$ can be taken to be the same  as the one \EEE in the  Sobolev Embedding Theorem for functions of bounded variation. In particular, the following property holds:  let $\mathcal{I}\subset \R$ \BBB be \EEE a family of indices, and let   $\MMM \{\Omega_i\} \EEE_{i\in\mathcal{I}}$ \EEE be  a collection of domains with the property that for each $i_1,i_2\in\mathcal{I}$ there exists a bi-Lipschitz homeomorphism $\varphi_{i_1,i_2}$ from $\Omega_{i_1}$ onto $\Omega_{i_2}$ such that, denoting with  $L_{i_1,i_2}$ the Lipschitz constant of $\varphi_{i_1,i_2}$, it holds $\sup_{i_1,i_2}L_{i_1,i_2}<+\infty$. Then it is possible to choose a positive constant $\gamma=\gamma(\{\Omega_i\}_{i\in\mathcal{I}})$, independent of \MMM $i \in \mathcal{I}$, \EEE for which Proposition~\ref{relative isoperimetric} holds true for every $i\in\mathcal{I}$.
\end{remark}
In the following,  $\gamma=\gamma(\Omega)>0$ always denotes  the constant of Proposition~\ref{relative isoperimetric}. We will often assume that the function $u\in {\rm GBV}_\star(\Omega;\Rk)$ \MMM under consideration \EEE satisfies
\begin{equation}\label{smallness per poincaré 1}
\big(2\gamma\Hd(J^r_{u}\cap \Omega)\big)^{d/(d-1)}\leq \frac{\Lb^d(\Omega)}{2}
\end{equation}
for some $r\geq 0$. When \eqref{smallness per poincaré 1} is satisfied, we define the vectors $\tau'(u,\Omega,r),\tau''(u,\Omega,r)\in\Rk$ by setting  
\begin{equation}\label{def tau}
\begin{aligned}
    \tau'_i(u,\Omega,r)&\coloneq (u_i)_\star\big((2\gamma\Hd(J^r_{u_i}\cap \Omega))^{d/(d-1)},\Omega\big),\\ 
    \tau''_i(u,\Omega,r)&\coloneq (u_i)_\star\big(\Lb^d(\Omega)-(2\gamma\Hd(J^r_{u_i}\cap \Omega))^{d/(d-1)},\Omega\big)
    \end{aligned}
    \end{equation}
for all $i\in\{1,\dots, k\}$. For brevity of notation, if $r=0$, we write $\tau'(u,\Omega)$ and $\tau''(u,\Omega)$ in place of $\tau'(u,\Omega,r)$ and  $\tau''(u,\Omega,r)$. Observe that, if $m_i$ is a median of $u_i$, then it  holds 
\begin{equation}\label{leq median leq}
    \tau'_i(u,\Omega,r)\leq m_i\leq \tau''_i(u,\Omega,r).
\end{equation} 
We now present a Poincaré inequality for functions in ${\rm GBV}_\star(\Omega;\Rk)$ with  a \EEE small set of large jumps, whose lines of proof follow closely those of \cite[Theorem~3.1]{CarrieroLeaci} (see also \cite[Theorem~4.14]{AFP}). 
\begin{theorem}\label{Thm: Poincaré inequality}
       There exists a constant $C_{\rm Poin}=C_{\rm Poin}(\Omega,k)>0$ such that, for every $r\geq 0$ and  $u\in {\rm GBV}_\star(\Omega;\Rk)$ satisfying \eqref{smallness per poincaré 1},  denoting by \EEE  $\overline{u}_r\in {\rm BV}(\Omega;\Rk)$  the function  $ \overline{u}_r \coloneq  \big(u \lor \tau'(u,\Omega,r)\big)\land \tau''(u,\Omega,r)$, it holds  
  \begin{align}\label{claim 1 poincaré}
        |D\overline{u}_r|(\Omega)&\leq 2 (r \vee 1)\mathcal{V}(u,\Omega\setminus J^r_u),\\\label{claim 2 poincaré}
       \|\overline{u}_r-m\|_{L^{d/(d-1)}(\Omega)}&\leq C_{\rm Poin}  (r \vee 1) \mathcal{V}(u,\Omega\setminus J^r_u),
\end{align}
for any median $m\in\Rk$ of $u$ on $\Omega$. Moreover, we have
\begin{equation}\label{bad set poincarè}
    \Lb^d(\{\overline{u}_r\neq u\})\leq C_{\rm Poin}\Hd(J^r_u\cap \Omega)^{d/(d-1)}.
\end{equation}
\end{theorem}
\begin{proof}
Since the  proof of the result  in the general  case $k>1$ follows immediately from the case $k=1$,   we may assume $u$ to be scalar-valued. Let $r\geq 0$ and  $u\in {\rm GBV}_\star(\Omega)$ be as in the statement.
For simplicity of notation, throughout   the \EEE  proof we write $\tau'$ and $\tau''$ in place of $\tau'(u,\Omega,r)$ and $\tau''(u,\Omega,r)$. 
Without loss of generality, we also assume that $m=0$, which by \eqref{leq median leq}  implies $\tau'\leq 0$ and $\tau''\geq 0$.
 
 Inequality \eqref{bad set poincarè} is a simple consequence of the fact that 
 $$\{  \overline{u}_r  \neq u\}=\{u<\tau'\}\cup \{u>\tau''\}$$
 and the definition \eqref{def tau} of $\tau'$ and $\tau''$.  We now come to the proof of \EEE \eqref{claim 1 poincaré} and \eqref{claim 2 poincaré}.
We begin  by \EEE observing that, thanks to Remark~\ref{remark: troncature}, $\MMM \overline{u}_r \EEE$ belongs to ${\rm BV}(\Omega)$, and by  the same remark we have 
\begin{equation}\label{first bound poincarè}
        |D\MMM \overline{u}_r \EEE|(\Omega)\leq   (r \vee 1)\mathcal{V}(u,\Omega\setminus J^r_u)+(\tau''-\tau')\Hd(J^r_u \cap \Omega). 
\end{equation}
 As  $\tau'\leq \MMM \overline{u}_r \EEE\leq \tau''$ and  $\MMM \overline{u}_r \EEE=u$  on \EEE $\{\tau'\leq u\leq \tau''\}$, by  the Coarea Formula we get
    \begin{equation}\label{coarea}
        |D\MMM \overline{u}_r \EEE|(\Omega)=\int_{\tau'}^{\tau''}\Hd(\partial^*\{u\geq t\}\cap \Omega)\,{\rm d}t.
    \end{equation}
   By Proposition~\ref{relative isoperimetric} applied with $\rho=1$ and $x=0$, we can estimate from below the right-hand side of the previous equality by    \begin{equation}\label{isoperimetric in poincarè}
\int_{\tau'}^{\tau''}\Hd(\partial^*\{u\geq t\}\cap \Omega)\,{\rm d}t\geq \frac{1}{\gamma}\Big(\int_{\tau'}^0\Lb^d(\{ \MMM u < t \EEE \})^{(d-1)/d}\,{\rm d}t+ \int^{\tau''}_0\Lb^d(\{ \MMM u \ge t \EEE \})^{(d-1)/d}\,{\rm d}t\Big).
    \end{equation}
    From the definitions of $\tau'$ and $\tau''$  \MMM in \EEE \eqref{def tau} it follows that  
    \begin{equation*}
        \begin{aligned}
          &\Lb^d(\{u \MMM < \EEE t\})^{(d-1)/d}\geq 2\gamma\Hd(J^r_{u}\cap \Omega)\quad \text{ for all }t\in (\tau',0),\\
          &\Lb^d(\{u \MMM \ge \EEE t\})^{(d-1)/d}\geq2\gamma\Hd(J^r_{u}\cap \Omega) \quad \text{ for all }t\in (0,\tau''),
        \end{aligned}
    \end{equation*}
which, combined with \eqref{coarea} and \eqref{isoperimetric in poincarè}, yield 
\begin{equation*}
    2(\tau''-\tau')\Hd(J^r_{u}\cap \Omega)\leq |D\MMM \overline{u}_r \EEE|(\Omega).
\end{equation*}
In light of \eqref{first bound poincarè}, this inequality also implies 
\begin{gather*}
     (\tau''-\tau')\Hd(J^r_{u}\cap \Omega)\leq   (r \vee 1) \mathcal{V}(u,\Omega\setminus J^r_u),
\end{gather*}
which in turn gives  \eqref{claim 1 poincaré}. As we assumed $\Omega$ to be connected  with Lipschitz boundary,  the Poincaré-type \EEE inequality \eqref{claim 2 poincaré} then follows from \eqref{claim 1 poincaré}  by applying  to $\MMM \overline{u}_r \EEE$ the classical \EEE Poincaré-Sobolev inequality for functions of bounded variation, see, for instance, \cite[Theorem~3.51]{AFP}. 
\end{proof}

\begin{remark}
    When $r=0$ and $u\in {\rm SBV}(\Omega;\Rk)$, the previous result reduces  exactly to the \EEE Poincaré inequality of De Giorgi, Carriero, \& Leaci \cite[Theorem~3.1]{CarrieroLeaci}. If    $u\in {\rm GBV}_\star(\Omega;\Rk)$ and $\Hd(J^r_u) =0$ for some $r \ge 0$, then $u\in {\rm BV}(\Omega;\Rk)$ and (by convention) $\tau_i'(u,\Omega,r)=-\infty$ and $\tau_i''(u,\Omega,r)=+\infty$, so that \eqref{claim 2 poincaré} reduces to the usual Poincaré inequality for  functions of bounded variation.
\end{remark}

\begin{remark}\label{remark: poincare constant}
    Inspecting the proof of Theorem~\ref{Thm: Poincaré inequality}, it is clear that the constant $C_{\rm Poin}$ depends on $\Omega$ only via the isoperimetric constant $\gamma=\gamma(\Omega)$. In particular,  letting $\MMM\mathcal{I}\EEE$ be a family of indices and $\{\Omega_i\}_{i\in\mathcal{I}}$ be a collection of domains as in Remark~ \ref{remark: constant isoperimetric},  it is possible to choose a \MMM uniform \EEE constant $C_{\rm Poin}$ for which the conclusions of Theorem~\ref{Thm: Poincaré inequality} hold for every $i\in\mathcal{I}$.
\end{remark}

\begin{remark}\label{remark:scaling}
Let $\rho>0$ and $x\in\Rd$, and consider the set $\Omega_{\rho,x}$ given by \eqref{def blow up set}. By H\"older's inequality and a scaling argument it can be shown that for each $u\in {\rm GBV}_\star(\Omega_{\rho,x};\Rk)$ satisfying the hypotheses of Theorem~\ref{Thm: Poincaré inequality} we have    
\begin{equation*}
\| \MMM \overline{u}_r \EEE -m\|_{L^1(\Omega_{\rho,x})}\leq \rho C_{\rm Poin} (r \vee 1) \mathcal{V}(u,\Omega_{\rho,x}\setminus J^r_u),
\end{equation*}
where $C_{\rm Poin}$ is the constant corresponding to $\rho=1$.
\end{remark}
As a simple consequence of the previous result, we show that, in the case \MMM that \EEE $u$ satisfies zero boundary conditions, the  Poincaré-type inequality \EEE \eqref{claim 2 poincaré} holds with $m=0$ \MMM for the original function, \EEE at the expense of removing a set $\omega$ of small Lebesgue measure.
\begin{corollary}\label{cor:small sets}
    Let $\Omega\subset \Rd$ be a bounded Lipschitz domain. Then there exists a constant $\bar{C}_{\rm Poin}=\bar{C}_{\rm Poin}(\Omega,k)>0$ such that for all $r\geq 0$ and $u\in {\rm GBV}_\star(\Omega;\Rk)$, with $u=0$ on $\partial \Omega$, there is a Borel set $\omega\subset \Omega$, with 
\begin{equation*}
    \Lb^d(\omega)\leq \bar{C}_{\rm Poin}\Hd(J^r_u)^{d/(d-1)},
    \end{equation*}
    such that     \begin{equation}\label{poincarè boundary}
        \|u\|_{L^{d/(d-1)}(\Omega\setminus \omega)}\leq \bar{C}_{\rm Poin} (r \vee 1)\mathcal{V}(u,\Omega \setminus J^r_u).
    \end{equation}
\end{corollary}
\begin{proof}
Let $r\geq 0$, let $u\in {\rm GBV}_\star(\Omega;\Rk)$, and let $\Omega'$ be a bounded   Lipschitz domain with $\Omega\subset \subset\Omega'$. Since $u=0$ on $\partial \Omega$,   we can extend  (without relabelling) \EEE $u$ to a function ${u}\in {\rm GBV}_\star(\Omega';\Rk)$ by setting ${u}=0$ on $\Omega'\setminus \Omega$.  Moreover, it is not restrictive to assume that 
\begin{equation}\label{smallness lemma boundary zero}
    \Hd(J^r_u)^{d/(d-1)}\leq \frac{\Lb^d(\Omega'\setminus \Omega)}{2C_{\rm Poin}},
\end{equation}
where $C_{\rm Poin}=C_{\rm Poin}(\Omega'\EEE,  k \EEE )$ is the positive constant of Theorem~\ref{Thm: Poincaré inequality}. Indeed, if \eqref{smallness lemma boundary zero} is not satisfied, it is enough to set $\omega=\Omega$ and choose $\bar{C}_{\rm Poin}$ large enough. By possibly choosing a smaller $\Omega'$, we may also assume that 
\begin{equation*}
   \Hd(J^r_u)^{d/(d-1)}\leq \frac{\Lb^d(\Omega'\setminus \Omega)}{2C_{\rm Poin}}\leq  \frac{1}{(2\gamma)^{ d/(d-1)}}\frac{\Lb^d(\Omega')}{2}.
\end{equation*}
In view of this inequality,  ${u}$ \EEE satisfies \eqref{smallness per poincaré 1}, so that we can  apply Theorem~\ref{Thm: Poincaré inequality} to $ {u}$ on $\Omega'$. \MMM We \EEE obtain a vector $m\in\Rk$ and a Borel set $\omega\coloneq \{{u}\neq \MMM \overline{{u}}_r \EEE\}\subset \Omega'$ such that 
\begin{equation}\label{poincarè piccolina}
    \|{u}-m\|_{L^{d/(d-1)}(\Omega'\setminus \omega)}\leq C_{\rm Poin}  (r \vee 1)\mathcal{V}({u},\Omega\setminus J^r_u)
\end{equation}
and 
\begin{equation}
\label{smallness poincarè zero}\Lb^d(\omega)\leq C_{\rm Poin}\Hd(J^r_{{u}})^{d/(d-1)}\leq \frac{\Lb^d(\Omega'\setminus \Omega)}{2},
\end{equation}
where in the   second inequality   \EEE we have used \eqref{smallness lemma boundary zero}.  Next, we note that, since ${u}=0$ on  $\Omega'\setminus \Omega$, from \eqref{poincarè piccolina} it follows  that 
\begin{equation*}
   |m|\leq \frac{C_{\rm Poin}}{\Lb^d(\Omega'\setminus(\Omega\cup \omega))}  (r \vee 1) \mathcal{V}(u,\Omega\setminus J^r_u).
\end{equation*}
Clearly, $\Lb^d(\Omega'\setminus(\Omega\cup \omega))\geq \Lb^d(\Omega'\setminus \Omega)-\Lb^d(\omega)$, so that  by \eqref{smallness poincarè zero} this implies that $\Lb^d(\Omega'\setminus(\Omega\cup \omega))\geq \Lb^d(\Omega'\setminus\Omega)/2$. Hence,  
\begin{equation*}
    |m|\leq\frac{2 C_{\rm Poin}}{\Lb^d(\Omega'\setminus \Omega)}  (r \vee 1)\mathcal{V}(u,\Omega\setminus J^r_u).
\end{equation*}
\MMM This \EEE shows that  \eqref{poincarè boundary} holds for a suitable choice of $\bar{C}_{\rm Poin}$. 
\end{proof}

\begin{remark}\label{remark: scaling p=1}
Let $\rho>0$ and $x\in\Rd$, and consider the set $\Omega_{\rho,x}$ given by \eqref{def blow up set}. By H\"older's inequality and a scaling argument, it can be shown that for each $u\in {\rm GBV}_\star(\Omega_{\rho,x};\Rk)$, with $u = 0$ on $\partial \Omega_{\rho,x}$, we can find $\omega\subset \Omega_{\rho,x}$  such that   
\begin{equation*}
\Lb^d(\omega)\leq \bar{C}_{\rm Poin}\Hd(J^r_u\cap \Omega_{\rho,x})^{d/(d-1)}\quad \text{ and }\quad \|u\|_{L^1(\Omega_{\rho,x}\setminus \omega)}\leq \rho \bar{C}_{\rm Poin} (r \vee 1)\mathcal{V}(u, \Omega_{\rho,x}\EEE\setminus J^r_u),
\end{equation*}
where the constant $\bar{C}_{\rm Poin}$ is the one obtained for $\rho=1$. 
\end{remark}

We conclude the section by presenting a last consequence of the Poincaré-type inequality \MMM in Theorem~\ref{Thm: Poincaré inequality}, \EEE  namely \EEE a compactness result for sequences $\{u_n\}_n$  with vanishing set of large jumps. Here and in the rest of the work, given a function $u\colon \Omega\to \Rk$, we let ${\rm med}(u,\Omega)\in\Rk$ be the vector whose $i$-th component is given by $(u_i)_\star(\Lb^d(\Omega)/2,\Omega)$, which we recall is a median of $u$ on $\Omega$.
\begin{corollary}\label{cor: vanishing compactness}
    Let $r\geq 0$ and let  $\{u_n\}_n\subset {\rm GBV}_\star(\Omega;\Rk)$ be a sequence such that
    \begin{equation}\label{vanishing jump set}
        \sup_{n\in\N}\mathcal{V}(u_n,\Omega\setminus J^r_{u_n})<+\infty\quad \text{ and }\quad \lim_{n\to +\infty}\Hd(J^r_{u_n}\cap \Omega)=0.
    \end{equation}
Let   $\overline{u}_n$   be the function given by $\big((u_n) \lor \tau'(u_n,\Omega,r)\big)\land \tau''(u_n,\Omega,r)$. Then the sequence $\{\overline{u}_n-{\rm med}(u_n,\Omega)\}_n$ is uniformly bounded in ${\rm BV}(\Omega;\Rk)$ and  there exist a subsequence, not relabelled,  and a function $u\in {\rm BV}(\Omega;\Rk)$ such that 
    \begin{align}\label{claim comp 1}
    &\{\overline{u}_n-{\rm med}(u_n,\Omega)\}_n \text{ converges to }u \text{ strongly in }L^1(\Omega;\Rk) \text{ as }n\to+\infty,\\\label{claim comp 2}
    &\{u_n-{\rm med}(u_n,\Omega)\}_n \text{ converges to }u \text{ in }L^0(\Omega;\Rk) \text{ as }n\to+\infty.
    \end{align}
    \begin{proof}
        Thanks to the equality in  \eqref{vanishing jump set},  we may invoke Theorem~\ref{Thm: Poincaré inequality} to find that, for $n$ large enough, the function $\overline{u}_n$ belongs to ${\rm BV}(\Omega;\Rk)$ and satisfies  
        \begin{gather*}
            |D\overline{u}_n|(\Omega)\leq 2 (r \vee 1)\mathcal{V}(u_n,\Omega\setminus J^r_{u_n}),\\
            \|\overline{u}_n-{\rm med}(u_n,\Omega)\|_{L^{d/(d-1)}(\Omega)}\leq C_{\rm Poin}  (r \vee 1) \mathcal{V}(u_n,\Omega\setminus J^r_{u_n}).
        \end{gather*}
         In light of the   inequality \EEE in \eqref{vanishing jump set},  we deduce that $\{\overline{u}_n-{\rm med}(u_n,\Omega)\}_n$ is uniformly bounded in ${\rm BV}(\Omega;\Rk)$. Hence, there exist a subsequence, not relabelled, and a function $u\in {\rm BV}(\Omega;\Rk)$ such that \eqref{claim comp 1} holds.  By  \EEE \eqref{bad set poincarè} applied to $u_n$ and \eqref{vanishing jump set}--\eqref{claim comp 1}  we finally get \EEE  \eqref{claim comp 2}.
    \end{proof}
\end{corollary}

\section{Blow-ups of \texorpdfstring{${\rm GBV}_\star(\Omega;\Rk)$}{GBV*(Ω;Rk)}-functions}\label{sec:blow up}

In this section, we analyse properties of blow-ups of functions in ${\rm GBV}_\star(\Omega;\Rk)$ around a point $x\in\Omega$. The key tool for this analysis will be the Poincaré-type inequality  introduced in the previous section.
We carry out this study in two distinct cases: first when $x$ is a point of approximate differentiability of $u$, and then when $x$ is a jump point of $u$. In both cases, we will show that it is always possible to find modifications of the function $u$ which are close in measure to $u$, but that  behave better in the blow-up limit. \EEE  The tools we  develop  here are crucial in the proof of Propositions \ref{prop: repre AC} and \ref{prop:repre surf}.

We begin by treating the case of points of approximate differentiability. \MMM The proof of the result \EEE borrows ideas from \cite[Lemma~2]{BFML2002}.

\begin{lemma}[Blow-up at points of approximate differentiability]\label{lemma blow up ac}
    Let $u\in {\rm GBV}_\star(\Omega;\Rk)$, and  \BBB for  $x\in \Omega$ \EEE let
         $u_\rho$ and $\overline{u}_\rho$ be the functions defined  by \BBB {\rm (}recall the  notation in Theorem~\ref{Thm: Poincaré inequality} for $r=0${\rm)}   
         $$u_\rho(y)\coloneq ({u}(y)-\widetilde{u}(x))/\rho\quad \text{ and }\quad \overline{u}_\rho(y)\coloneq \big(u_\rho(y)\lor \tau'(u_\rho,Q(x,\rho)\big)\land \tau''(u_\rho,Q(x,\rho) )  \ \ \  \text{for } y\in Q(x,\rho). $$ Then \BBB for $\mathcal{L}^d$-a.e.\ $x \in \Omega$ we have \EEE 
    \begin{align}
   \label{claim 2 blowup ac}
    & \lim_{\rho\to 0^+}\frac{1}{\rho^{d+1}}\Lb^d(\{ \overline{u}_\rho\neq u_\rho\})=0,
 \\ \label{claim 3 blowup ac}  &  \lim_{\rho\to 0^+}\frac{1}{\rho^{d}}\int_{Q(x,\rho)}\Big|\overline{u}_\rho(y)-\nabla u(x)\frac{(y-x)}{\rho}\Big|\,{\rm d}y=0,\\ \label{claim 4 blowup ac}
    &\lim_{\rho\to 0^+}\frac{1}{\rho^d}\int_{Q(x,\rho)}|\rho\nabla \overline{u}_\rho(y)-\nabla u(x)|\,{\rm d}y=0,\\\label{claim 5 blowup ac}
    &\lim_{\rho\to 0^+}\frac{1}{\rho^d}\Hd(J_{\overline{u}_\rho}\cap Q(x,\rho))=0 \quad \text{ and } \lim_{\rho\to 0^+}\frac{1}{\rho^{d-1}}|D^c\overline{u}_\rho|(Q(x,\rho))=0.
    \end{align}
\end{lemma}
\begin{proof}
We give the proof only in the scalar valued case $k=1$, as the general case $k>1$ is an immediate consequence of the former. \BBB We note that for $\mathcal{L}^d$-a.e.\ $x \in \Omega$ we have
    \begin{align}\label{Lebesgue gradients}
        &\lim_{\rho\to 0^+}\frac{1}{\rho^{d}}\int_{Q(x,\rho)}|\nabla u(y)-\nabla u(x)|\,{\rm d}y=0,\\\label{non singular}
       & \lim_{\rho\to 0^+}\frac{1}{\rho^d}\Hd(J_{u}\cap Q(x,\rho))=0 \quad \text{ and } \lim_{\rho\to 0^+}\frac{1}{\rho^d}|D^cu|(Q(x,\rho))=0,\\\label{approx differentiabilty}
       & \lim_{\rho\to 0^+}\frac{1}{\rho^d}\Lb^d\Big(\Big\{y\in Q(x,\rho)\colon\,\frac{|u(y)-\widetilde{u}(x)-\nabla u(x) (y-x)|}{|y-x|}>\lambda\Big\}\Big)=0 \text{ for all }\lambda>0.
    \end{align}
It suffices to prove the statement for such $x \in \Omega$ satisfying \eqref{Lebesgue gradients}--\eqref{approx differentiabilty}. \EEE   We begin by observing that \eqref{approx differentiabilty} implies  
    \begin{equation}\label{slightly weaker approx diff}
        \lim_{\rho\to 0^+}\frac{1}{\rho^d}\Lb^d\Big(\Big\{y\in Q(x,\rho)\colon\,\frac{|u(y)-\widetilde{u}(x)-\nabla u(x)\cdot(y-x)|}{\rho} >\lambda\Big\}\Big)=0 \text{ for all }\lambda>0,
    \end{equation} 
 that $u_\rho\in {\rm GBV}_\star(Q(x,\rho))$ and, by Remark~\ref{remark:  troncature},  that $\overline{u}_\rho$ belongs to ${\rm BV}(Q(x,\rho))\cap L^\infty(Q(x,\rho))$. By the same remark it also follows that  $J_{\overline{u}_\rho}\subset J_u$ up to   an \EEE $\Hd$-negligible set and that $|D^c\overline{u}_\rho|(Q(x,\rho))\leq \rho^{-1}|D^cu|(Q(x,\rho))$.  Thus, \EEE by \eqref{non singular} we get  
\begin{align*}
    \limsup_{\rho\to 0^+}\frac{1}{\rho^d}\Hd(J_{\overline{u}_\rho}\cap Q(x,\rho)) & \leq\limsup_{\rho\to 0^+}\frac{1}{\rho^d}\Hd(J_{{u}}\cap Q(x,\rho))=0,\\
     \limsup_{\rho\to 0^+}\frac{1}{\rho^{d-1}}|D^c\overline{u}_\rho|(Q(x,\rho)) & \leq  \limsup_{\rho\to 0^+}\frac{1}{\rho^{d}}|D^c{u}|(Q(x,\rho))=0.
\end{align*}
 This \EEE proves \eqref{claim 5 blowup ac},
so that we are left with proving \eqref{claim 2 blowup ac}--\eqref{claim 4 blowup ac}. 

To this end, we
 consider  suitable rescalings of $u_\rho$. More precisely, for $z\in Q$ we define $v_\rho(z)\coloneq u_\rho(x+\rho z)$.  We have \EEE $\{v_\rho\}_\rho\subset {\rm GBV}_\star(Q)$ and \BBB by \MMM  a change of variables \EEE  it holds  
     \begin{equation}\notag
     \mathcal{V}(v_\rho,Q\setminus J_{v_\rho} )\leq \frac{1}{\rho^{d}}\Big(\int_{Q(x,\rho)}|\nabla {u}|\,{\rm d}x+|D^c{u}|\big(Q(x,\rho)\big)\Big).
\end{equation}
     Since $x+\rho J_{v_\rho}=J_{u_\rho}=J_u\cap Q(x,\rho)$, up to $\Hd$-negligible sets,   by a change of variables we get \EEE     \begin{equation*}
    \Hd(Q\cap J_{v_\rho}) = \frac{1}{\rho^{d-1}}\Hd( J_u\cap Q(x,\rho)).
\end{equation*}
Combining the two previous  displayed formulas \EEE with \eqref{Lebesgue gradients} and \eqref{non singular}, we get
\begin{equation}\label{vanishing V}
    \limsup_{\rho\to 0^+} \mathcal{V}(v_\rho,Q\setminus J_{v_\rho} )<+\infty \quad \text{and}\quad \lim_{\rho\to 0^+}\Hd( J_{v_\rho}\cap Q)=0.
\end{equation}
Therefore, for $\rho>0$ small enough we may apply Theorem~\ref{Thm: Poincaré inequality} with $r=0$, and by \eqref{bad set poincarè} we get that the function defined  by $\overline{v}_\rho(z)\coloneq (v_\rho(z)\lor \tau'(v_\rho,Q))\land \tau''(v_\rho,Q)$ for $z\in Q$ satisfies
\begin{equation}\label{bad set weps}
        \Lb^d(\{\overline{v}_\rho\neq v_\rho\})\leq C_{\rm Poin}\Hd(J_{v_\rho}\cap Q)^{d/(d-1)}.
    \end{equation}
For every $0< s\leq \rho^d$ we have the equality $(u_\rho)_\star(s,Q(x,\rho))=(v_\rho)_\star(s/\rho^d,Q)$, so that (recall \eqref{eq: medain})
    \begin{gather*}
        \tau'(v_\rho,Q)= \tau'(u_\rho,Q(x,\rho)) \quad \text{ and }\quad   \tau''(v_\rho,Q)= \tau''(u_\rho,Q(x,\rho)),\\
        \text{med}(v_\rho,Q)=\text{med}(u_\rho,Q(x,\rho)),
    \end{gather*}
and,    as consequence, $\overline{v}_\rho(z)=\overline{u}_\rho(x+\rho z)$. A change of variables and \eqref{bad set weps} then show that 
\begin{equation*}
     \frac{1}{\rho^d}\Lb^d(\{\overline{u}_\rho\neq u_\rho\})\leq  C_{\rm Poin} \EEE\Big(\frac{1}{\rho^{d-1}}\Hd(J_{u_\rho}\cap Q(x,\rho))\Big)^{d/(d-1)}.
\end{equation*}
 Therefore, \EEE
\begin{equation*}
    \frac{1}{\rho^{d^2/(d-1)}}\Lb^d(\{\overline{u}_\rho\neq u_\rho\})\leq C_{\rm Poin}\Big(\frac{1}{\rho^d}\Hd\big(J_{u_\rho}\cap Q(x,\rho)\big)\Big)^{\frac{d}{d-1}}=C_{\rm Poin}\Big(\frac{1}{\rho^d}\Hd\big(J_{u}\cap Q(x,\rho)\big)\Big)^{\frac{d}{d-1}}.
\end{equation*}
Recalling the first equality in  \eqref{non singular}, and observing that $d^2/(d-1)\geq (d+1)$, we have just proved   
\begin{equation}\label{vanishig bad set}
    \lim_{\rho\to 0^+}\frac{1}{\rho^{d+1}}\Lb^d(\{\overline{u}_\rho\neq u_\rho\})=0,
\end{equation}
which is exactly \eqref{claim 2 blowup ac}.  To get \eqref{claim 4 blowup ac}, we simply observe that   
\begin{equation*}
    \nabla  \overline{u}_\rho \EEE (y)=\begin{cases}\displaystyle\rho^{-1}\nabla u(y)& \text{ if $y\in \{\overline{u}_\rho=u_\rho\}$},\\
      0  & \text{ if $y\in Q(x,\rho)\cap \{\overline{u}_\rho\neq u_\rho\}$},
    \end{cases}
\end{equation*}
so that, in view of  \eqref{Lebesgue gradients} and  \eqref{vanishig bad set}, we infer
\begin{align*}
    \limsup_{\rho\to 0^+}\frac{1}{\rho^d}&\int_{Q(x,\rho)}|\rho\nabla \overline{u}_\rho(y)-\nabla u(x)|\,{\rm d}y \\ & \leq \lim_{\rho\to 0^+}\Big(\frac{1}{\rho^d}\int_{Q(x,\rho)}|\nabla u(y)-\nabla u(x)|\,{\rm d}y +\frac{1}{\rho^d}\Lb^d(\{\overline{u}_\rho\neq u_\rho\})|\nabla u(x)|\Big)=0.
    \end{align*}
We  conclude by proving \eqref{claim 3 blowup ac}.
To this end, we observe that, thanks to \eqref{vanishing V}, we may apply Corollary~\ref{cor: vanishing compactness} with $r=0$ to obtain  a sequence $\{\rho_n\}_n$, with $\rho_n\to 0$ as $n\to+\infty$, and a function $v_0\in {\rm BV}(Q)$ such that 
\begin{align}\label{comp w 1}
    &v_{\rho_n}-{\rm med}(v_{\rho_n},Q) \  \text{ converges to }v_0 \text{ in }L^0(Q;\Rk) \text{ as }n\to+\infty,\\\label{comp w 2}
    &\overline{v}_{\rho_n}-{\rm med}(v_{\rho_n},Q) \  \text{ converges to }v_0\text{ strongly in }L^1(Q;\Rk) \text{ as }n\to+\infty.
    \end{align}
 The same arguments portrayed in \cite[Lemma~2, Page~196]{BFML2002} show that, by   using   \eqref{slightly weaker approx diff},
\begin{equation*}
    \lim_{n\to +\infty}\text{med}(v_{\rho_n},Q)=0.
\end{equation*}
 In view of \eqref{slightly weaker approx diff}, \eqref{comp w 1}, and \eqref{comp w 2}, this implies that  $v_0(z)=\nabla u(x) z$ for all $z\in Q$, so that $v_0$ does not depend on the chosen subsequence $\{\rho_n\}_n$. Thus,
\begin{align*}
 \{\overline{v}_{\rho}\}_\rho \text{ converges to }\nabla u(x)z\text{ strongly in }L^1(Q;\Rk) \text{ as }\rho\to0^+.
    \end{align*}
\MMM It is now enough to observe that, by another change of variables, \EEE  this can be rewritten as 
\begin{equation*}
    \lim_{\rho\to 0^+}\frac{1}{\rho^d}\int_{Q(x,\rho)}\Big|\overline{u}_\rho(y) -\nabla u(x)\frac{(y-x)}{\rho}\Big|\,{\rm d}y=0,
\end{equation*}
which is exactly \eqref{claim 3 blowup ac}. This concludes the proof of the lemma.
\end{proof}

In the following, we analyse the behaviour of blow-ups of functions in ${\rm GBV_\star}(\Omega;\Rk)$ around jump points, drawing ideas from   \cite[Lemma~6.1]{CriFriSolo} and \cite[Lemma~3.4]{ContiFocardiIurlanoSBDp}.    The  strategy is  similar to the one employed in Lemma~\ref{lemma blow up ac},  based on the following observation: since the jump set $J_u$ is $(d-1)$-rectifiable, around a jump point $x$ and for $\rho>0$ small enough, $J_u\cap Q(x,\rho)$  is  mostly contained in a smooth hypersurface which subdivides the cube into two connected halves.  As the restrictions of the function $u$ to each of the two halves have small jump sets, it is then possible to perform the modifications proposed in Lemma~\ref{lemma blow up ac} \MMM in each half.  \BBB Recall  \MMM  \eqref{def blow up set} and \EEE \eqref{stepfun}. \EEE  

\begin{lemma}[Blow-up at jump points]\label{lemma: blow up jump points} 
Let $u \in {\rm GBV}_\star(\Omega;\Rk)$. Then, \BBB for $\mathcal{H}^{d-1}$-a.e.\ $x \in J_u$, letting $\nu\coloneq \nu_u(x)$ and $v_{x}\coloneq v_{x,u^+,u^-,\nu}$,  there  exists \EEE $u_\rho\in{\rm BV}(Q_\nu(x,\rho);\Rk)$ such that 
\begin{align}\label{claim blow up jump 1}
 & \lim_{\rho\to 0^+}\frac{1}{\rho^{d}}\Lb^d(\{u_\rho\neq u\}\cap Q_\nu(x,\rho))=0,  
 \\ \label{claim blow up jump 2}&  \lim_{\rho \to 0^+}   \frac{1}{\rho^{d}} \int_{Q_\nu(x,\rho)}  |u_\rho -  {v}_{x}  |\, {\rm d}y = 0,\\\label{claim blow up jump 3}
 & \lim_{\rho \to 0^+}     \frac{1}{\rho^{d-1}}\,\mathcal{H}^{d-1}(J_{u_\rho}\cap E_{\rho,x})\leq \MMM \Hd(\Pi_x^\nu \cap E) \EEE \quad \text{for all Borel sets } E \subset \MMM  Q_\nu(x,1), \EEE \\\label{claim blow up jump 4}
 &    \lim_{\rho \to 0^+} \ \frac{1}{\rho^{d-1}} \int_{Q_\nu(x,\rho)} | \nabla u_\rho| \, \mathrm{d}y = 0\quad \text{and}\quad \lim_{\rho \to 0^+} \ \frac{1}{\rho^{d-1}}|D^cu_\rho|(Q_\nu(x,\rho))=0.
\end{align}
\end{lemma}
\BBB Note that the function $u_\rho$ is different from the one constructed in Lemma~\ref{lemma blow up ac}  although we use the same notation.  
\begin{proof}
Clearly, it is enough to prove the result in the scalar case $k=1$. 

\BBB
We observe that for $\mathcal{H}^{d-1}$-a.e.\ $x \in J_u$ the following holds: first, we have \begin{align}\label{blow up jump 1}
   &\lim_{\rho\to 0^+}\frac{1}{\rho^{d-1}}\Hd(J_u\cap Q_\nu(x,\rho))=1,\\ \label{blow up jump 2}
   &\lim_{\rho\to 0^+}\frac{1}{\rho^{d-1}}\int_{Q_\nu(x,\rho)}|\nabla u|\,{\rm d}y=0\quad \text{ and }\quad    \lim_{\rho\to 0^+}\frac{1}{\rho^{d-1}}|D^cu|(Q_\nu(x,\rho))=0,\\\label{blow up jump 3}
   &\lim_{\rho\to 0^+}\frac{1}{\rho^d}\Lb^d(\{y\in  Q_\nu(x,\rho) \EEE \colon \, |u(y)-v_{x}(y)|>\lambda\})=0\quad \text{for all $\lambda>0$}.
\end{align}
Indeed, almost every jump point   is  a point of density $1$ for $J_u$, i.e., \eqref{blow up jump 1} holds.  Condition \eqref{blow up jump 2}   is satisfied for  $\Hd$-a.e.\ point of $\Omega$ since the measures $\nabla u \Lb^d$ and $D^cu$ are diffuse (see, for instance, \cite[Theorem~2.56]{AFP}).  Condition \eqref{blow up jump 3} is a   rephrasing of the definition of jump points given in  \eqref{def Jump set}. 

Moreover,  there exist $\overline{\rho}>0$ and  a $(d-1)$-dimensional surface $\Gamma\subset Q_\nu(x,\bar{\rho})$ of class $C^1$ with $x \in \Gamma$ such that $\Gamma$ is the graph of a function $h$ defined on the hyperplane $\Pi^\nu_{x}=\{(y-x)\cdot\nu=0\}$, the hyperplane $\Pi^\nu_x$ is tangent to $\Gamma$ in $x$, for all $0<\rho<\overline{\rho}$ the surface $\Gamma\cap Q_\nu(x,\rho)$ separates $Q_\nu(x,\rho)$ into two open connected components $Q^{\Gamma,\pm}_{\nu}(x,\rho)$,  and 
\begin{align}
 & \label{blow up jump 4}\lim_{\rho\to 0^+} \frac{1}{\rho^{d-1}} \Hd\big((J_u \triangle \Gamma) \cap Q_\nu(x,\rho)\big)=0,  \\ & \label{blow up jump 5}
 \lim_{\rho\to 0^+}  \frac{1}{\rho^{d-1}} \Hd(\Gamma\cap  E_{\rho,x}  )= \Hd(\Pi^\nu_x \cap E) \quad \text{for all Borel sets } E \subset Q_\nu( \MMM x \EEE ,1).
\end{align}
This property can be justified  by \EEE arguing exactly as in the beginning of the proof of \cite[Theorem~2]{ChambolleApprox} (see also the comments in \cite[Section~3]{AddendumChambolle}), noting that  the   arguments rely only on the rectifiability of the set $J_u$. 

We start with a first observation: the Lipschitz constant of $h \, \mres_{ \Pi^\nu_x\cap Q_\nu(x,\rho)\EEE}$ \EEE  converges to zero as $\rho\to 0^+$.  \EEE  This implies that     
 \begin{equation}\label{splitting}
 \lim_{\rho\to 0^+}\frac{1}{\rho^{d}}\Lb^d\big(Q_\nu^{\Gamma,\pm}(x,\rho) \triangle Q_\nu^\pm(x,\rho)\big)=0,
 \end{equation}
where we have set $Q_\nu^\pm(x,\rho) \coloneq \{ y \in Q_\nu(x,\rho)\colon \pm  (y-x)  \cdot \nu >0\}$.  

In view of   \eqref{blow up jump 4}, the restriction of $u$ to $Q^{\Gamma,\pm}_{\nu}(x,\rho)$ satisfies the hypothesis of Theorem~\ref{Thm: Poincaré inequality} with $\Omega$ replaced by $Q^{\Gamma,\pm}_{\nu}(x,\rho)$. As $\{  \frac{1}{\rho}(\Gamma\cap Q_\nu(x,\rho))\}_\rho$ are graphs of functions whose Lipschitz constant converges to zero as $\rho\to 0^+$, the constant in the Poincar\'e-type inequality can be chosen independently of $\rho$, see  Remarks \ref{remark: poincare constant}--\ref{remark:scaling}.    Hence, letting $\overline{u}^{\pm}_\rho\in {\rm BV}(Q^{\Gamma,\pm}_\nu(x,\rho))$ be the functions defined by 
\begin{equation*}
\overline{u}_\rho^{\pm}\coloneq \big(u\lor \tau'(u,Q_\nu^{\Gamma,\pm}(x,\rho)\big)\land\tau''(u,Q_\nu^{\Gamma,\pm}(x,\rho))
\end{equation*}
and letting $m^\pm_\rho\coloneq {\rm med}(u,Q_\nu^{\Gamma,\pm}(x,\rho))$, recalling also Remark~\ref{remark:scaling}, 
we have that   
\begin{align}
\label{bounded derivative}
&|D\overline{u}^\pm_\rho|(Q_\nu^{\Gamma,\pm})\leq 2\mathcal{V}\big(u,Q^{\Gamma,\pm}_\nu(x,\rho)\setminus J_u\big),\\
\label{poincare blow up jump}&\int_{Q_\nu^{\Gamma,\pm}(x,\rho)}|\overline{u}^\pm_\rho-m^\pm_\rho|\,{\rm d}y\leq \rho C_{\rm Poin} \mathcal{V}\big(u,Q^{\Gamma,\pm}_\nu(x,\rho)\setminus J_u\big),
\\\label{smallness on Gamma}
&\Lb^d(\{\overline{u}^\pm_{\rho}\neq u\})\leq C_{\rm Poin}\Hd\big(J_{u}\cap Q^{\Gamma,\pm}_\nu(x,\rho)\big)^{d/(d-1)},\\
\label{inclusion jump sets truncations}
& J_{\overline{u}^\pm_\rho}\subset J_u\cap Q^{\Gamma,\pm}_\nu(x,\rho) \quad \text{up to an $\Hd$-negligible set}
\end{align}
for a constant $C_{\rm Poin}>0$  independent of $\rho$. Note that from \eqref{blow up jump 4} and \eqref{smallness on Gamma} it follows that 
\begin{equation}\label{smalness ubarrapiu}
    \lim_{\rho\to 0^+}\frac{1}{\rho^d}\Lb^d\big(\{\overline{u}^\pm_{\rho}\neq u\}\cap Q^{\Gamma,\pm}_{\nu}(x,\rho)\big)=0.
\end{equation}

 We now introduce \MMM the \EEE modifications $u_\rho\in {\rm BV}(Q_\nu(x,\rho))$ by setting 
\begin{equation*}u_\rho(y)\coloneq \begin{cases}\overline{u}^{+}_\rho &\text{ if $y\in Q^{\Gamma,+}_\nu(x,\rho)$},\\
    \overline{u}^{-}_\rho &\text{ if $y\in Q^{\Gamma,-}_\nu(x,\rho)$,}        
    \end{cases}
\end{equation*}
and show that $\{u_\rho\}_\rho$ \MMM satisfy \EEE properties \eqref{claim blow up jump 1}--\eqref{claim blow up jump 4}. Property \eqref{claim blow up jump 1} has already been addressed in \eqref{smalness ubarrapiu}. To prove \eqref{claim blow up jump 3}, we first observe that by \eqref{inclusion jump sets truncations} we have  
 \begin{equation*}
 J_{u_\rho} \subset \Gamma \cup J_u\quad \text{ up to an $\Hd$-negligible set},
 \end{equation*}
 which, together with \eqref{blow up jump 4} and \eqref{blow up jump 5}, shows \eqref{claim blow up jump 3}.  
 The equalities in \eqref{claim blow up jump 4} are a consequence of \eqref{blow up jump 2} and of \eqref{bounded derivative}. 

 It remains to prove  \eqref{claim blow up jump 2}.  We note that this would follow from
\begin{equation}\label{median to jumps}
\lim_{\rho\to 0^+}{ m^\pm_\rho}=u^{\pm}(x).
\end{equation}
Indeed, \eqref{blow up jump 2} and \eqref{poincare blow up jump}  give that
\begin{equation*}
\lim_{\rho\to 0^+} \frac{1}{\rho^d} \int_{ Q^{\Gamma,\pm}_\nu(x,\rho)}|  u_\rho   -m^\pm_\rho|\, {\rm d}y=0.
\end{equation*}
Recalling \eqref{splitting}, by the triangle inequality this implies   
\begin{equation*}
\lim_{\rho\to 0^+} \frac{1}{\rho^d} \int_{ Q^{\pm}_\nu(x,\rho)}|  {u}_\rho-m^\pm_\rho|\, {\rm d}y=0.
\end{equation*}
Then, \eqref{median to jumps} and the triangle inequality  yield, in turn, \eqref{claim blow up jump 2}.

We now show \eqref{median to jumps}. In view of \eqref{blow up jump 3}, by a diagonal argument we can find a sequence $\{\lambda_\rho\}_\rho \subset (0,+\infty)$ with $\lambda_\rho\to 0$ as $\rho\to 0^+$ such that the sets
\begin{equation*}
{\omega}^{\pm}_\rho\coloneq \{y \in Q^\pm_\nu(x,\rho) \colon\, |u(y) - u^\pm(x)|> \lambda_\rho\}
\end{equation*}
\MMM satisfy \EEE   
\begin{equation}\label{smallness lambda rho}
\lim_{\rho\to 0^+}\frac{1}{\rho^{d}}\Lb^d(\omega^\pm_\rho)=0.
\end{equation}
Recalling \eqref{poincare blow up jump},  we have that
\begin{equation}\label{poincare u meno cose}
\int_{Q^{\Gamma,\pm}_\nu(x,\rho)\setminus\{\overline{u}^\pm_{\rho}\neq u\}}|u(y) - m^\pm_\rho|\,{\rm d}y  \le \rho C_{\rm Poin} \mathcal{V}\big(u,Q^{\Gamma,\pm}_\nu(x,\rho)\setminus J_u\big).
\end{equation}
By the  definition of ${\omega}^\pm_\rho$, using \eqref{poincare u meno cose} and the  triangle inequality, we then get  
\begin{equation}\label{final countdown}
|u^\pm(x)-m^\pm_\rho|\Lb^d\big(( Q^\pm_\nu(x,\rho) \cap Q_\nu^{\Gamma,\pm}(x,\rho))\setminus (\{\overline{u}^\pm_{\rho}\neq u\}\cup \omega^\pm_\rho)\big)  \le \rho C_{\rm Poin} \mathcal{V}\big(u,Q^{\Gamma,\pm}_\nu(x,\rho)\setminus J_u\big)+\rho^d\lambda_\rho.
\end{equation}
 From \eqref{splitting}, \eqref{smalness ubarrapiu}, and \eqref{smallness lambda rho} we deduce the existence  of \EEE ${\rho}_0>0$ such that 
 \begin{equation*}
     \frac{1}{\rho^d}\Lb^d\big((Q^\pm_\nu(x,\rho) \cap Q_\nu^{\Gamma,\pm}(x,\rho))\setminus (\{\overline{u}^\pm_{\rho}\neq u\}\cup \omega^\pm_\rho)\big) \geq\frac14 
 \end{equation*}
 for every $\rho\in (0,{\rho}_0)$. Then from inequality \eqref{final countdown} we infer   \begin{equation*}
     |u^\pm(x)-m^\pm_\rho| \le \frac{4C_{\rm Poin}}{\rho^{d-1}}  \mathcal{V}\big(u,Q^{\Gamma,\pm}_\nu(x,\rho)\setminus J_u\big)+4\lambda_\rho
 \end{equation*}
 for all $\rho\in (0,{\rho}_0)$. Letting $\rho \to 0^+$ and resorting to \eqref{blow up jump 2} and  to the fact that $\lambda_\rho\to 0$ as $\rho\to 0^+$, we obtain \eqref{median to jumps}. \MMM This  \EEE  concludes the proof of the lemma.
\end{proof}

 \section{A lower semicontinuity lemma}\label{sec:lower}
In this section, we present a well-known property for lower semicontinuous integral functionals on ${\rm BV}(\Omega;\Rk)$ with linear growth,  employed in the proof of Theorem~\ref{thm:relaxation full}. In the following, we denote by $\Rkd_1 \subset \R^{k \times d}$ the subspace of rank-one matrices. 

\EEE 
\begin{lemma}\label{lemma: recession}
    Let $f\colon\Rkd\to [0,+\infty)$, $h\colon\Rkd_1\to [0,+\infty)$,  \BBB and \EEE $g\colon \Rk\times \Sd\to [0,+\infty)$ be Borel integrands such that 
    \begin{gather*}
  0\leq f(A)\leq C(|A|+1) \quad \text{ for all $A\in\Rkd$,}\\
        h \text{ is  positively \EEE $1$-homogeneous},\\
    g(\zeta,\nu)=g(-\zeta,-\nu)    \quad \text{ and }\quad  0\leq g(\zeta,\nu)\leq C|\zeta| \quad \text{for all $\zeta\in\Rk$ and $\nu\in\Sd$}
    \end{gather*}
    for some  $C>0$. Consider the functional $F\colon {\rm BV}(\Omega;\Rk)\times \Op(\Omega)\to [0,+\infty)$ given by 
\begin{equation*}
    F(u,O)\coloneq \int_{O}f(\nabla u)\,{\rm d}x+\int_{O}h\Big(\frac{{\rm d} D^cu}{{\rm d}|D^cu|}\Big)\,{\rm d}|D^cu|+\int_{J_u\cap O}g( \MMM [u], \EEE \nu_u)\,{\rm d}\Hd
\end{equation*}
for all $u\in {\rm BV}(\Omega;\Rk)$ and all $O\in\Op(\Omega)$, and  assume that $F(\cdot,O)$ is $L^1(\Omega;\Rk)$-lower semicontinuous for all $O\in\Op(\Omega)$. Then
\begin{equation*}
h(A)=f^\infty(A)\quad \text{for all $ \BBB A \in \EEE \Rkd_1$,}
\end{equation*}
where $f^\infty(A)\coloneq \limsup_{t\to +\infty}\frac{f(tA)}{t}$.
 
\end{lemma}

\begin{proof}
 Let us fix $a\in\Rk$, $\nu\in\mathbb{S}^{d-1}$, and set $A\coloneq a\otimes \nu$. Without loss of generality, we may assume that $Q_\nu\subset \Omega$, where we recall that $Q_\nu$ is a cube \MMM centred at 0 \EEE with two faces orthogonal to $\nu$ and sides of unit length.

We first prove that $h(A)\leq f^\infty(A).$ To this end, we  let $\psi\colon  [-1/2,1/2]\to[-1/2,1/2]\EEE$  \EEE (be  a translation of) \EEE Cantor's staircase function and we consider the function  defined for $x\in Q_\nu$ by $u(x)\coloneq a\psi(x\cdot \nu)$. Clearly,  $u\in {\rm BV}(Q_\nu;\Rk)$, $|Du|(Q_\nu)=|D^cu|(Q_\nu)=|A|$, and 
\begin{equation*}
    \frac{{\rm d}D^cu}{{\rm d}|D^cu|}=\frac{A}{|A|} \quad \text {$|D^cu|$-a.e.\ on $Q_\nu$}. 
\end{equation*}
In particular, $F(u,Q_\nu)=f(0)+h(A)$. Then, we let $\{\psi_n\}_n$ be the sequence of piecewise affine functions   considered in the construction of $\psi$, which   converges to $\psi$ uniformly and satisfies
\begin{equation*}
    \psi'_n=t_n\chi_{E_n}\quad \text{ and }\quad \Lb^1(E_n)=t_n^{-1}
\end{equation*}
for all $n\in\N$, where  $t_n\coloneq (3/2)^n$ and  $E_n$ is the union of the $2^n$ intervals remaining in the $n$-th step of the construction of the Cantor set. \MMM We \EEE set $u_n(x)\coloneq a\psi_n(x\cdot \nu)$. We note that $\{u_n\}_n$ converges to $u$ strongly in $L^1(Q_\nu;\Rk)$ and that $\nabla u_n(x)=t_n\chi_{E_n}(x\cdot \nu)A.$ Thus, 
\begin{equation*}
    F(u_n,Q_\nu)=\int_{Q_\nu}f(\nabla u_n)\,{\rm d}x=(1-t_n^{-1})f(0)+\frac{f(t_n A)}{t_n}.
\end{equation*}
By the $L^1$-lower semicontinuity of $F$, this implies 
\begin{equation*}
    f(0)+h(A)=F(u,Q_\nu)\leq \liminf_{n\to+\infty}F(u_n,Q_\nu) \leq \EEE f(0)+f^\infty(A), 
\end{equation*}
which proves $h(A)\leq f^\infty(A)$.

To prove the converse inequality, we let $t>0$ and consider the affine map defined  by $v(x)=ta(x\cdot \nu)$  for $x\in Q_\nu$. Then, we set
\begin{equation*}
    w_n(s)\coloneq \frac{1}{n}\big(\lfloor ns\rfloor+\psi(ns-\lfloor ns\rfloor)\big )\quad\text{ for $s\in [ -1/2,1/2\EEE]$},
\end{equation*}
where $\lfloor\cdot\rfloor$ denotes the floor function.
Clearly, $\{ w\EEE_n\}_n$ converges uniformly to the identity on $[-1/2,1/2\EEE]$, and $|Dw_n|(-1/2,1/2\EEE)=|D^cw_n|(-1/2,1/2\EEE)=1$ for all $n\in\N$, so that the functions $v_n(x)\coloneq taw_n(x\cdot \nu)$ converge to $v$ strongly in $L^1(Q_\nu;\Rk)$, $|Dv_n|(Q_\nu)=|D^cv_n|(Q_\nu)= t|a| \EEE=t|A|$, and 
\begin{equation*}
    \frac{{\rm d}D^cv_n}{{\rm d}|D^cv_n|}=\frac{A}{|A|}\quad \text{$|D^cv_n|$-a.e.\ in $Q_\nu$}.
\end{equation*}
Thus, $F(v_n,Q_\nu)=f(0)+th(A)$, which by lower semicontinuity gives
\begin{equation*}
    f(tA)=F(v,Q_\nu)\leq \liminf_{n\to+\infty}F(v_n,Q_\nu)=f(0)+th(A).
\end{equation*}
\MMM Dividing \EEE  by $t$ and letting $t\to +\infty$ we obtain $f^\infty(A)\leq h(A)$, concluding the proof.
\end{proof}

\end{document}